\documentclass[reqno,a4paper]{amsart}
\usepackage[utf8]{inputenc}

\usepackage{amsmath,amssymb}
\usepackage{graphicx}
\usepackage{mathrsfs}
\usepackage{yhmath}

\usepackage{xcolor}

\usepackage{booktabs}

\usepackage[font=footnotesize,labelfont=bf]{caption}

\newtheorem{theorem}{Theorem}
\newtheorem{proposition}{Proposition}
\newtheorem{corollary}{Corollary}
\newtheorem{lemma}{Lemma}

\newcommand{\de}[0]{\mathrel{\mathop:}=}
\newcommand{\ed}[0]{=\mathrel{\mathop:}}
\newcommand{\ie}[0]{\mathrm{i}}
\newcommand{\dif}[1]{\mathrm{d}#1}
\newcommand{\R}[0]{\mathbb{R}}
\newcommand{\C}[0]{\mathbb{C}}
\newcommand{\Z}[0]{\mathbb{Z}}
\newcommand{\N}[0]{\mathbb{N}}

\newcommand{\D}[1]{\mathbb{D}\left(#1\right)}
\newcommand{\cD}[1]{\overline{\mathbb{D}}\left(#1\right)}

\DeclareMathOperator{\arsinh}{\mathrm{arsinh}}
\DeclareMathOperator{\Li2}{\mathrm{Li}_2}

\title[Atkinson's formula for the mean square]{Atkinson's formula for the mean square of $\zeta(s)$ with an explicit error term}
\author[A.~Simoni\v{c}]{Aleksander Simoni\v{c}}
\address{School of Science, The University of New South Wales (Canberra), ACT, Australia}
\email{a.simonic@student.adfa.edu.au}
\author[V.~V.~Starichkova]{Valeriia V. Starichkova}
\address{School of Science, The University of New South Wales (Canberra), ACT, Australia}
\email{v.starichkova@adfa.edu.au}
\subjclass[2010]{11M06, 11N37; 11Y35}
\keywords{Riemann zeta-function, Mean square theorems, Explicit results}

\begin{document}

\begin{abstract}
We provide an explicit $O\left(\log^2{T}\right)$-term of the celebrated Atkinson's formula for the error term $E(T)$ of the second power moment of the Riemann zeta-function on the critical line. As an application, we obtain an explicit version of well-known estimate $E(T)\ll_{\varepsilon} T^{\frac{1}{3}+\varepsilon}$.
\end{abstract}

\maketitle
\thispagestyle{empty}

\section{Introduction}

A classical problem in the mean value theory of the Riemann zeta-function $\zeta(s)$ is to find the asymptotic formula for the mean square
\[
I_{\sigma}(T) \de \int_{0}^{T}\left|\zeta\left(\sigma+\ie t\right)\right|^2 \dif{t},
\]
where $\sigma\geq1/2$ and $T\geq 2$. The first result in this direction was obtained independently by Landau and Schnee in 1909, namely that $I_{\sigma}(T)\sim\zeta(2\sigma)T$ for $\sigma>1/2$, which easily follows from classical approximation of $\zeta(s)$ with Dirichlet polynomials, see~\cite[Theorems 4.11 and 7.2]{Titchmarsh}. A much harder problem was to study the asymptotic behaviour of $I_{1/2}(T)$. Hardy and Littlewood proved in 1918 that $I_{1/2}(T)\sim T\log{T}$, and reproved it again in 1923 by using the approximate functional equation for $\zeta(s)$, see~\cite[Theorem 7.3]{Titchmarsh}. Littlewood announced\footnote{In fact Littlewood stated~\eqref{eq:Littlewood} without $-2\gamma$.} without proof in 1922 that
\begin{equation}
\label{eq:Littlewood}
I_{\frac{1}{2}}(T) = T\log{T} - \left(1+\log{2\pi}-2\gamma\right)T + E(T)
\end{equation}
with $E(T)\ll_{\epsilon} T^{\frac{3}{4}+\varepsilon}$, $\varepsilon>0$ and $\gamma$ is the Euler--Mascheroni constant. Six years later, Ingham~\cite{Ingham1928} proved equation~\eqref{eq:Littlewood} with
\begin{equation}
\label{eq:InghamE}
E(T)\ll \sqrt{T}\log{T},
\end{equation}
see also~\cite{TitchmarshCorput3} for an alternative proof,~\cite[Theorem 7.4]{Titchmarsh} for the proof of $E(T)\ll_{\varepsilon} T^{\frac{1}{2}+\varepsilon}$, and also~\cite{Atkinson39} for the proof of $E(T)\ll \sqrt{T}\log^2{T}$. Titchmarsh~\cite{TitchmarshCorput5} was the first who improved the exponent $1/2$. He used van der Corput's theory of exponential sums in combination with the Riemann--Siegel formula instead of the approximate functional equation to obtain $E(T)\ll T^{\frac{5}{12}}\log^2{T}$. Balasubramanian~\cite{Balasub} improved\footnote{In the introduction of~\cite{Balasub} the bound with the exponent $346/1067+\varepsilon$ is stated. However, his proof (see Lemmas 32 and 33) provides the value $27/82+\varepsilon$. See also~\cite{HBFourth} for an alternative approach to the former bound.} Titchmarsh's result to $E(T)\ll_{\varepsilon} T^{\frac{27}{82}+\varepsilon}$, and currently the strongest known estimate $E(T)\ll_{\varepsilon} T^{\frac{1515}{4816}+\varepsilon}$ is due to Bourgain and Watt~\cite{BourgainWatt}. It is conjectured that $E(T)\ll_{\varepsilon} T^{\frac{1}{4}+\varepsilon}$, which is well-supported by omega results, e.g., Good~\cite{Good1977} proved that $E(T)=\Omega\left(T^{\frac{1}{4}}\right)$. For $\Omega_{\pm}$-improvements of this result, as well as for an overview of the mean value theory, see the expository papers~\cite{MatsumotoRecent} and~\cite{Ivic2014}.

In 1849, Dirichlet proved
\begin{equation}
\label{eq:divisorSum}
D(x) \de\sum_{n\leq x}d(n) = x\log{x}+(2\gamma-1)x+\Delta(x)
\end{equation}
with $\Delta(x)\ll\sqrt{x}$, where $x\geq 1$ and $d(n)$ is the number of divisors of $n\in\N$. The estimation of $\Delta(x)$ is known as the \emph{Dirichlet divisor problem}, and conjecture $\Delta(x)\ll_{\varepsilon}x^{\frac{1}{4}+\varepsilon}$ as the \emph{divisor conjecture}. The strongest known result $\Delta(x)\ll_{\varepsilon}x^{\frac{131}{416}+\varepsilon}$ is due to Huxley. Comparing equations~\eqref{eq:Littlewood} and~\eqref{eq:divisorSum} reveals
\[
I_{\frac{1}{2}}(T) = 2\pi D\left(\frac{T}{2\pi}\right) + E(T) - 2\pi\Delta\left(\frac{T}{2\pi}\right).
\]
Deeper connection between $E(T)$ and $\Delta(x)$ was discovered by Atkinson in~\cite{Atkinson}, where he presented a new method which allows to express $E(T)$ in quite exact form. We will state the following version of his theorem.

\begin{theorem}
\label{thm:AtkinsonOriginal}
Let $A'\leq A''$ be any two fixed positive constants and $N$ a positive integer such that $A'T \leq N \leq A''T$. Then
\begin{equation}
\label{eq:GenExprForE}
E(T) = \Sigma_{1}(T,N) + \Sigma_2(T,N) + \mathfrak{E}(T),
\end{equation}
where
\begin{equation}
\label{eq:AtkSigma1}
\Sigma_1(T,N) \de \sqrt{2}\left(\frac{T}{2\pi}\right)^{\frac{1}{4}}\sum_{n\leq N}(-1)^n \frac{d(n)}{n^{\frac{3}{4}}}e(T,n)\cos{f(T,n)},
\end{equation}
\begin{equation}
\label{eq:AtkSigma2}
\Sigma_2(T,N) \de -2\sum_{n\leq \mathcal{Z}(T,N)} \frac{d(n)}{n^{\frac{1}{2}}}\left(\log{\frac{T}{2\pi n}}\right)^{-1}\cos{g(T,n)},
\end{equation}
and
\begin{equation}
\label{eq:AtkError}
\mathfrak{E}(T) \ll \log^2{T},
\end{equation}
with
\begin{gather}
    \mathcal{Z}(T,N) \de \frac{1}{4}\left(\sqrt{N+\frac{1}{2}+\frac{2T}{\pi}}-\sqrt{N+\frac{1}{2}}\right)^2, \label{eq:AtkZ} \\
    e(T,n) \de \left(1+\frac{\pi n}{2T}\right)^{-\frac{1}{4}}\left(\frac{2T}{\pi n}\right)^{-\frac{1}{2}}\left(\arsinh{\sqrt{\frac{\pi n}{2T}}}\right)^{-1}, \label{eq:Atke} \\
    f(T,n) \de 2T\arsinh{\sqrt{\frac{\pi n}{2T}}} + \sqrt{(\pi n)^2+2\pi nT} - \frac{\pi}{4}, \label{eq:Atkf} \\
    g(T,n) \de T\log{\frac{T}{2\pi n}} - T + \frac{\pi}{4}. \nonumber 
\end{gather}
\end{theorem}

Equation~\eqref{eq:GenExprForE}, together with its main terms~\eqref{eq:AtkSigma1} and~\eqref{eq:AtkSigma2}, is known as \emph{Atkinson's formula} for $E(T)$. Our formulation of Theorem~\ref{thm:AtkinsonOriginal} differs from the original statement, see\footnote{In the formulation of the main theorem of~\cite{Atkinson} the sign in front of the second sum should be $+$ instead of $-$.}~\cite{Atkinson,Jutila83} or~\cite[p.~442]{Ivic}; there $N$ can be an arbitrary positive real number and $\mathcal{Z}(T,N)$ is defined as $\mathcal{Z}(T,N-1/2)$ from~\eqref{eq:AtkZ}. Both versions are equivalent in the sense that one implies the other with the same bound~\eqref{eq:AtkError} on the error term.

As it is already noticed by Atkinson, trivial estimation of $\Sigma_1$ and $\Sigma_2$ reproves Ingham's result $E(T)\ll \sqrt{T}\log{T}$. However, the first nontrivial applications of Atkinson's formula were found by Heath-Brown almost thirty years later. In~\cite{HBTheMeanValue} he proved
\[
\int_{2}^{T}E^2(t)\dif{t} = \frac{2\zeta^{4}\left(\frac{3}{2}\right)}{3\sqrt{2\pi}\zeta(3)}T^{\frac{3}{2}} + O\left(T^{\frac{5}{4}}\log^{2}{T}\right),
\]
which immediately implies $E(T)=\Omega\left(T^{\frac{1}{4}}\right)$, and in~\cite{HBTheTwelfth} that
\[
\int_{0}^{T} \left|\zeta\left(\frac{1}{2}+\ie t\right)\right|^{12}\dif{t} \ll T^{2}\log^{17}{T}.
\]
His proof of the latter inequality (see~\cite[Section 5]{Ivic2014} for a different approach) depends on a certain estimate of the mean square in short intervals, see~\cite[Lemma 1]{HBTheTwelfth} and also~\cite[Theorem 7.2]{Ivic}, where in the integration process Atkinson's formula is used with the Gaussian exponential factor as a truncating device. Jutila extended Heath-Brown's method in~\cite{Jutila83} where, among other things, he obtained
\begin{equation}
\label{eq:JutilaOrigEst}
E(T) \ll X + T^{\frac{1}{4}}\sup_{\substack{|t-T|\leq X \\ \xi\leq M}}\left\{\left|\sum_{n\leq \xi}(-1)^{n}\frac{d(n)}{n^{\frac{3}{4}}}\cos{f(t,n)}\right|\right\}.
\end{equation}
Here, $T^{\frac{1}{4}}\leq X\leq T^{\frac{1}{3}}$ and $M=TX^{-2}\log^{8}{T}$. Observe that $\Sigma_1$ is the main contributor to the above estimate. It is expected that $\Sigma_2\ll_{\varepsilon} T^{\varepsilon}$ because the Lindel\"{o}f Hypothesis, i.e., $\zeta\left(1/2+\ie t\right)\ll_{\varepsilon} |t|^{\varepsilon}$, guarantees
\[
\sum_{n\leq x}\frac{d(n)}{n^{\frac{1}{2}+\ie T}} \ll_{\varepsilon} T^{\varepsilon}
\]
for $x\ll T$, see~\cite[Section 4]{GGYLindelof} for details. Inequality~\eqref{eq:JutilaOrigEst} further connects the Dirichlet divisor problem with $E(T)$ via the truncated Vorono\"{i} formula (see Section~\ref{sec:Voronoi}), although we need to emphasize that such connection is merely a heuristic evidence for an unproven claim that bounds for $E(T)$ and $\Delta(T)$ should be the same. Taking $X=T^{\frac{1}{3}}$ into~\eqref{eq:JutilaOrigEst} and estimating trivially, we obtain $E(T)\ll T^{\frac{1}{3}}\log^{3}{T}$. Motohashi~\cite[Theorem 9]{MotohashiLectures} refined Jutila's method to obtain
\begin{equation}
\label{eq:JutilaE}
E(T)\ll T^{\frac{1}{3}}\log^{2}{T}.
\end{equation}
Note that we cannot improve the exponent $1/3$ with such a simple process. However, it is possible to use~\eqref{eq:JutilaOrigEst} to prove $E(T)\ll_{\varepsilon} T^{\frac{35}{108}+\varepsilon}$ by applying the delicate method of Kolesnik for handling multiple exponential sums, see~\cite[Lemma 7.3]{Ivic}.

Very little work has been done on explicit estimates in the mean value theory of $\zeta(s)$. Dona, Helfgott and Zuniga Alterman obtained in~\cite[Theorem 4.3]{DHA} several lower and upper explicit bounds for $I_{\sigma}(T)$ where $\sigma\in[0,1]$. In our notation, their result on the second power moment on the critical line is
\[
-2T\sqrt{\log{T}}+0.69283\cdot T+1.242\leq E(T)\leq 2T\sqrt{\log{T}}+24.74642\cdot T+1.243
\]
for $T\geq 4$. The first author~\cite{SimonicEZDE} provided a better upper bound, namely
\[
E(T) \leq 70.26\cdot T^{\frac{3}{4}}\sqrt{\log{\frac{T}{2\pi}}}
\]
for $T\geq 2000$. Dona and Zuniga Alterman~\cite{DonaZA} succeeded to improve this result, showing that
\begin{gather}
\left|E(T)\right| \leq 18.169\sqrt{T}\log^2{T}, \nonumber \\
\left|E(T)\right| \leq 2.552\sqrt{T}\log^2{T} + 8.177\sqrt{T}\log{T} \label{eq:DonaZA}
\end{gather}
for $T\geq 100$ and $T\geq 10^{30}$, respectively, by adapting Atkinson's proof~\cite{Atkinson39}. As remarked in~\cite[p.~4]{DHA}, having an explicit version of Atkinson's formula would give, at least for large $T$, a better estimate. Moreover, in addition to the above applications such a result could be useful in a wider setting, not just in the estimation of $E(T)$. The main purpose of this paper is to provide an explicit bound on the error term $\mathfrak{E}(T)$ from Theorem~\ref{thm:AtkinsonOriginal}.

\begin{theorem}
\label{thm:OurAtkinson}
Let $A'T \leq N \leq A''T$, $N$ be a positive integer and $T\geq T_0$. Then~\eqref{eq:GenExprForE} is true with
\[
\left|\mathfrak{E}(T)\right| \leq a_1\left(A',A'',T_0\right)\log^2{T} + a_2\left(A',A'',T_0\right)\log{T} + a_3\left(A',A'',T_0\right),
\]
where values for $a_1$, $a_2$ and $a_3$ are provided in Table~\ref{tab:thm} for various $A'$, $A''$ and $T_0$.
\end{theorem}

\begin{table}[h]
   \centering
\begin{tabular}{lllllll}
\toprule
$A'$ & $A''$ & $\alpha_0$ & $a_1\left(A',A'',T_0\right)$ & $a_2\left(A',A'',T_0\right)$ & \multicolumn{2}{l}{$a_3\left(A',A'',T_0\right)$} \\
\midrule
$0.9$ & $1.1$ & $0.05$ & $95$ & $275$ & $2.4\cdot 10^{20}$ & $3.5\cdot10^{12}$ \\
 & & & & & $7.3\cdot10^{18}$ & $8.5\cdot10^{10}$ \\
 & & & & & $2.1\cdot10^{17}$ & $2.1\cdot10^{9}$ \\
 & & & & & $5.5\cdot10^{15}$ & $4.9\cdot10^{7}$ \\
 & & & & & $1.4\cdot10^{14}$ & $825$ \\
\midrule
$0.5$ & $1.5$ & $0.04$ & $217$ & $814$ & $3.4\cdot 10^{20}$ & $5.0\cdot10^{12}$ \\
 & & & & & $1.1\cdot10^{19}$ & $1.3\cdot10^{11}$ \\
 & & & & & $3.0\cdot10^{17}$ & $3.0\cdot10^{9}$ \\
 & & & & & $7.8\cdot10^{15}$ & $7.1\cdot10^{7}$ \\
 & & & & & $2.0\cdot10^{14}$ & $2194$ \\
\midrule
$0.1$ & $1.9$ & $0.02$ & $2444$ & $13197$ & $2.0\cdot 10^{21}$ & $3.1\cdot10^{13}$ \\
 & & & & & $6.3\cdot10^{19}$ & $7.5\cdot10^{11}$ \\
 & & & & & $1.8\cdot10^{18}$ & $1.8\cdot10^{10}$ \\
 & & & & & $4.8\cdot10^{16}$ & $4.3\cdot10^{8}$ \\
 & & & & & $1.3\cdot10^{15}$ & $33903$ \\
\bottomrule
\end{tabular}
   \caption{Values for $a_1$, $a_2$ and $a_3$ for particular parameters $A'$, $A''$ and $T_0$, where $T_0=10^{20},10^{30},\ldots,10^{100},10^{1000}$. Numbers in the last two columns are written in this order from left up -- left down to right up -- right down.}
   \label{tab:thm}
\end{table}

Our choice for $A'$ and $A''$ is reasonable because in all of the previously mentioned applications of Atkinson's formula $N=\lfloor{T}\rfloor$ is taken. The values in the last two columns in Table~\ref{tab:thm} are large due to large error terms in the saddle point lemma (Proposition~\ref{prop:saddle}). Also, our proof\footnote{There exist three different proofs of Theorem~\ref{thm:AtkinsonOriginal}. Motohashi~\cite{MotohashiLectures} used his version of the approximate functional equation for $\zeta^{2}(s)$ to prove~\eqref{eq:AtkError} with even better error term $O\left(\log{T}\right)$, while Jutila~\cite{Jutila97} and Lukkarinen~\cite{Lukkarinen} derived~\eqref{eq:AtkError} via the Laplace transform.} strongly relies on Atkinson's original approach. As an application of Theorem~\ref{thm:OurAtkinson}, we make~\eqref{eq:JutilaE} explicit by following Jutila's proof of inequality~\eqref{eq:JutilaOrigEst}, see Section~\ref{sec:Coro}. However, careful examination of the log-terms in his approach allows us to improve~\eqref{eq:JutilaE} to
\begin{equation}
\label{eq:ourJutila}
E(T)\ll_{\varepsilon} T^{\frac{1}{3}}\log^{\frac{3}{2}+\varepsilon}{T}
\end{equation}
with fully explicit constants which depend on $\varepsilon>0$. For the sake of simplicity we set $\varepsilon=1/6$ to obtain the following.

\begin{corollary}
\label{cor:main}
We have
\begin{equation}
\label{eq:ExplJut}
\left|E(T)\right| \leq J\left(T_0\right) T^{\frac{1}{3}}\log^{\frac{5}{3}}{T}
\end{equation}
for $T\geq 1.1T_0$, where values for $J\left(T_0\right)$ are given by Table~\ref{tab:cor}.
\end{corollary}

A precise version of Corollary~\ref{cor:main} is given in Theorem~\ref{thm:Jutila}. Observe that the last estimate improves~\eqref{eq:DonaZA} for $T\geq 2.4\cdot 10^{30}$. With the method presented here, the limit of $J\left(T_0\right)$ when $T_0\to\infty$ is around $2.2$.

\begin{table}[h]
   \centering
\begin{tabular}{lllllllll}
\toprule
$T_0$ & $10^{30}$ & $10^{40}$ & $10^{50}$ & $10^{60}$ & $10^{70}$ & $10^{80}$ & $10^{90}$ & $10^{100}$ \\
\midrule
$J\left(T_0\right)$ & $1.272\cdot10^6$ & $14.366$ & $3.194$ & $3.09$ & $3.011$ & $2.95$ & $2.9$ & $2.856$  \\
\bottomrule
\end{tabular}
   \caption{Values for $J\left(T_0\right)$ from Corollary~\ref{cor:main}.}
   \label{tab:cor}
\end{table}

There exist several extensions of Atkinson's formula: Matsumoto and Meurman proved in~\cite{MatsumotoMeurman3} the analogue of~\eqref{eq:GenExprForE} for $\sigma\in(1/2,1)$; Jutila extended~\eqref{eq:GenExprForE} to a primitive of Hardy's $Z$-function, see~\cite[Theorem 8.2]{IvicHardy}; and Motohashi in the spectral theory of the Riemann zeta-function, see~\cite[Theorem 4.1]{MotohashiSpectral}. Because the former two extensions rely on the saddle point technique and are in this view closer to the original proof, our results may prove useful in future investigations in the mean value theory of $\zeta(s)$ and $Z(t)$.

The outline of this paper is as follows. In Section~\ref{sec:RevAtk} we give a concise review of the proof of Theorem~\ref{thm:AtkinsonOriginal}. Vorono\"{i}'s formula with an explicit error term (Lemma~\ref{lem:Voronoi}) is given in Section~\ref{sec:Voronoi}, and the saddle point lemma (Proposition~\ref{prop:saddle}) is proved in Section~\ref{sec:saddle}. Proposition~\ref{prop:saddle} is used in Section~\ref{sec:TwoIntegrals} to obtain explicit bounds for two integrals (Propositions~\ref{prop:FirstIntegral} and~\ref{prop:SecondIntegral}), which are then applied in Section~\ref{sec:proof} to provide the proof of Theorem~\ref{thm:OurAtkinson}. Finally, the proof of Corollary~\ref{cor:main} is given in Section~\ref{sec:Coro}.

\section{Quick review of the proof of Atkinson's formula}
\label{sec:RevAtk}

The proof of Theorem~\ref{thm:AtkinsonOriginal} consists of two main parts: derivation of the exact expression for $E(T)$ in terms of four definite integrals and negligible error term, see equation~\eqref{eq:MainForE}, and delicate estimation of these integrals by the saddle point technique (Proposition~\ref{prop:saddle}). Because our aim is to obtain an explicit error term in Atkinson's formula, it is clear that the second part of the proof is of main interest in the present paper, and thus it will be completely presented in the forthcoming sections. However, for the sake of completeness, we will provide the principal ideas from the first part of the proof. The reader is advised to consult details in~\cite[pp.~443--451]{Ivic} or~\cite{Atkinson}, see also~\cite{Ivic2014} for a similar review.

The proof begins with the simple equation\footnote{In spectral theory setting also known as the Atkinson dissection, see~\cite[p.~151]{MotohashiSpectral}.}
\begin{gather}
\zeta(u)\zeta(v) = \sum_{n=1}^{\infty}\sum_{m=1}^{\infty}\frac{1}{n^{u}m^{v}} = \zeta(u+v) + \mathfrak{F}(u,v) + \mathfrak{F}(v,u), \label{eq:AtkinsonStart} \\
\mathfrak{F}(u,v) \de \sum_{n=1}^{\infty}\sum_{m=1}^{\infty}\frac{1}{n^{u}(n+m)^{v}}, \nonumber
\end{gather}
which is valid for $\Re\{u\}>1$ and $\Re\{v\}>1$. It is clear that in this domain the function $\mathfrak{F}(u,v)$ is holomorphic in both variables. By means of partial summation it is possible to obtain analytic continuation of $\mathfrak{F}(u,v)$, which will be denoted by the same letter, to a larger domain
\[
\left\{(u,v)\in\C^2 \colon \Re\{u+v\}>0,u+v\notin\{1,2\},v\neq 1\right\}.
\]
Define
\begin{equation}
\label{eq:frakG}
\mathfrak{G}(u,v) \de \mathfrak{F}(u,v) - \frac{\Gamma(1-u)\Gamma(u+v-1)}{\Gamma(v)}\zeta(u+v-1).
\end{equation}
Then $\mathfrak{G}(u,v)$ is a holomorphic function on the domain
\[
\left\{(u,v)\in\C^2 \colon \Re\{u+v\}>0,u+v\notin\{1,2\},u\notin\N,v\notin\Z_{\leq 1} \right\}.
\]
After combining~\eqref{eq:AtkinsonStart} and~\eqref{eq:frakG}, we obtain the equality
\begin{multline*}
\zeta(u)\zeta(v) = \zeta(u+v) + \left(\frac{\Gamma(1-u)}{\Gamma(v)}+\frac{\Gamma(1-v)}{\Gamma(u)}\right) \times \\
\times\Gamma(u+v-1)\zeta(u+v-1) + \mathfrak{G}(u,v) + \mathfrak{G}(v,u),
\end{multline*}
valid for $0<\Re\{u\}<1$, $0<\Re\{v\}<1$ and $u+v\neq 1$. Taking $0<\Re\{u\}<1$ and setting $v=1-u+\delta$ for sufficiently small $|\delta|$, it can be shown that $\lim_{\delta\to0}\mathfrak{G}(u,1-u+\delta)$ and $\lim_{\delta\to0}\mathfrak{G}(1-u+\delta,u)$ exist, and
\begin{multline}
\label{eq:zetazeta}
\zeta(u)\zeta(1-u) = \frac{1}{2}\left(\frac{\Gamma'}{\Gamma}(1-u)+\frac{\Gamma'}{\Gamma}(u)\right)+2\gamma-\log{2\pi} \\
+\lim_{\delta\to0}\mathfrak{G}(u,1-u+\delta)+\lim_{\delta\to0}\mathfrak{G}(1-u+\delta,u).
\end{multline}
Therefore,~\eqref{eq:zetazeta} is true for all $0<\Re\{u\}<1$.

The Poisson summation formula implies
\begin{equation}
\label{eq:poisson}
\mathfrak{G}(u,v) = 2\sum_{m=1}^{\infty}m^{1-u-v}\sum_{k=1}^{\infty}\int_{0}^{\infty}y^{-u}(1+y)^{-v}\cos{\left(2\pi kmy\right)}\dif{y}
\end{equation}
for $\Re\{u\}<0$ and $\Re\{u+v\}>2$. However, it is possible to show further that function on the right-hand side of~\eqref{eq:poisson} is holomorphic on the domain
\[
\left\{(u,v)\in\C^2 \colon \Re\{u+v\}>0, \Re\{u\}<0, \Re\{v\}>1\right\}
\]
since the double series in~\eqref{eq:poisson} is absolutely convergent for values in this domain. Moreover, we can obtain
\[
\mathfrak{G}(u,v) = 2\sum_{n=1}^{\infty}\sum_{d\mid n}d^{1-u-v}\int_{0}^{\infty}y^{-u}(1+y)^{-v}\cos{\left(2\pi ny\right)}\dif{y}
\]
by grouping terms with $km=n$ in~\eqref{eq:poisson} together. Therefore,
\[
\mathfrak{G}(u,1-u) = \sum_{n=1}^{\infty}d(n)\mathfrak{H}(u,n), \quad \mathfrak{H}(u,x) \de 2\int_{0}^{\infty}y^{-u}(1+y)^{u-1}\cos{\left(2\pi xy\right)}\dif{y}
\]
for $\Re\{u\}<0$. The next step is to obtain analytic continuation of $\mathfrak{G}(u,1-u)$ to some domain which includes the critical line $\Re\{u\}=1/2$.

Let $N$ be a positive integer. Taking~\eqref{eq:divisorSum} and writing sums with Stieltjes integrals gives
\begin{flalign*}
\mathfrak{G}(u,1-u) &= \left(\sum_{n\leq N}+\sum_{n>N}\right)d(n)\mathfrak{H}(u,n) = \mathfrak{G}_1(u) + \int_{N+\frac{1}{2}}^{\infty}\mathfrak{H}(u,x)\dif{D(x)} \\
&= \mathfrak{G}_1(u) + \mathfrak{G}_3(u) + \int_{N+\frac{1}{2}}^{\infty}\mathfrak{H}(u,x)\dif{\Delta(x)} \\
&= \mathfrak{G}_1(u) - \mathfrak{G}_2(u) + \mathfrak{G}_3(u) - \mathfrak{G}_4(u),
\end{flalign*}
where
\begin{gather*}
\mathfrak{G}_1(u) \de \sum_{n\leq N}d(n)\mathfrak{H}(u,n), \quad \mathfrak{G}_2(u) \de \Delta^{\ast}\left(N+\frac{1}{2}\right)\mathfrak{H}\left(u,N+\frac{1}{2}\right), \\
\mathfrak{G}_3(u) \de \int_{N+\frac{1}{2}}^{\infty}\left(\log{x}+2\gamma\right)\mathfrak{H}(u,x)\dif{x}, \quad \mathfrak{G}_4(u) \de \int_{N+\frac{1}{2}}^{\infty} \Delta^{\ast}(x)\frac{\partial{\mathfrak{H}(u,x)}}{\partial{x}}\dif{x}.
\end{gather*}
Here, $\Delta^{\ast}(x)$ is defined by~\eqref{eq:DeltaStar}. Replacing $\Delta(x)$ with $\Delta^{\ast}(x)$ is justified because $\dif{\Delta(x)}=\dif{\Delta^{\ast}(x)}$ for $x\notin\N$, see~\eqref{eq:deltadiference}. Detailed analysis of the above functions shows that $\mathfrak{G}_1(u)$ and $\mathfrak{G}_2(u)$ admit analytic continuation to $\Re\{u\}<1$, $\mathfrak{G}_3(u)$ can be analytically continued to $\Re\{u\}<1, u\neq0$, while a bound like~\eqref{eq:explicitDeltaStar} guarantees analytic continuation of $\mathfrak{G}_4(u)$ to $\Re\{u\}<2/3$. This allows us to take $u=1/2+\ie t$ in~\eqref{eq:zetazeta}, thus deriving
\begin{multline*}
\left|\zeta\left(\frac{1}{2}+\ie t\right)\right|^2 = \frac{1}{2}\left(\frac{\Gamma'}{\Gamma}\left(\frac{1}{2}+\ie t\right)+\frac{\Gamma'}{\Gamma}\left(\frac{1}{2}-\ie t\right)\right)+2\gamma-\log{2\pi} \\
+ \sum_{j=1}^{4}(-1)^{j+1}\left(\mathfrak{G}_j\left(\frac{1}{2}+\ie t\right) + \mathfrak{G}_j\left(\frac{1}{2}-\ie t\right)\right).
\end{multline*}
Integrating the previous equality over $t$ from $0$ to $T>0$ implies
\[
I_{\frac{1}{2}}(T) = \frac{1}{2\ie}\log{\frac{\Gamma\left(\frac{1}{2}+\ie T\right)}{\Gamma\left(\frac{1}{2}-\ie T\right)}} + \left(2\gamma-\log{2\pi}\right)T + \sum_{j=1}^{4}(-1)^{j+1}\mathcal{E}_{j},
\]
where
\[
\mathcal{E}_{j} \de \int_{-T}^{T}\mathfrak{G}_j\left(\frac{1}{2}+\ie t\right)\dif{t} = -\ie\int_{\frac{1}{2}-\ie T}^{\frac{1}{2}+\ie T}\mathfrak{G}_j\left(u\right)\dif{u}.
\]
Using Stirling's formula for $\log{\Gamma(z)}$ with an explicit error term, see~\cite[p.~294]{Olver}, together with the exact representations for holomorphic extensions of $\mathfrak{G}_j(u)$, we can write
\begin{equation}
\label{eq:MainForE}
E(T) = \mathcal{E}_1 - \mathcal{E}_2 + \mathcal{E}_3 - \mathcal{E}_4 + R(T),
\end{equation}
where
\begin{equation*}
\mathcal{E}_1 \de 4\sum_{n\leq N} d(n)\int_{0}^{\infty} \frac{\sin{\left(T\log{\frac{1+y}{y}}\right)}\cos{\left(2\pi ny\right)}}{\sqrt{y(1+y)}\log{\frac{1+y}{y}}}\dif{y},
\end{equation*}
\begin{equation*}
\mathcal{E}_2 \de 4\Delta^{\ast}\left(N+\frac{1}{2}\right)\int_{0}^{\infty} \frac{\sin{\left(T\log{\frac{1+y}{y}}\right)}\cos{\left(2\pi\left(N+\frac{1}{2}\right)y\right)}}{\sqrt{y(1+y)}\log{\frac{1+y}{y}}}\dif{y},
\end{equation*}
\begin{multline*}
\mathcal{E}_3 \de -\frac{2}{\pi}\left(\log{\left(N+\frac{1}{2}\right)}+2\gamma\right)\int_{0}^{\infty}\frac{\sin{\left(T\log{\frac{1+y}{y}}\right)}\sin{\left(2\pi\left(N+\frac{1}{2}\right)y\right)}}{y\sqrt{y(1+y)}\log{\frac{1+y}{y}}}\dif{y} \\
+ \frac{1}{\pi\ie}\int_{0}^{\infty}\frac{\sin{\left(2\pi\left(N+\frac{1}{2}\right)y\right)}}{y}\int_{\frac{1}{2}-\ie T}^{\frac{1}{2}+\ie T}\frac{1}{u}\left(\frac{1+y}{y}\right)^{u}\dif{u}\dif{y},
\end{multline*}
and
\begin{multline*}
\mathcal{E}_4 \de 4\int_{N+\frac{1}{2}}^{\infty} \frac{\Delta^{\ast}(x)}{x} \int_{0}^{\infty} \frac{\cos{\left(2\pi xy\right)}}{(1+y)\sqrt{y(1+y)}\log{\frac{1+y}{y}}}\times \\
\times\left(T\cos{\left(T\log{\frac{1+y}{y}}\right)}-\left(\frac{1}{2}+\log^{-1}{\frac{1+y}{y}}\right)\sin{\left(T\log{\frac{1+y}{y}}\right)}\right)\dif{y}\dif{x},
\end{multline*}
together with
\begin{equation}
\label{eq:RStirling}
\left|R(T)\right| \leq \frac{T}{2}\log{\left(1+\frac{1}{4T^2}\right)} + \frac{1}{12T} + \frac{1}{90T^3}.
\end{equation}

This concludes the first part of the proof. From now on the full attention is dedicated to technically difficult estimations of the terms $\mathcal{E}_j$.

\section{On Vorono\"{i}'s formula}
\label{sec:Voronoi}

Turning back to the Dirichlet divisor problem, by the work of Berkane, Bordell\`{e}s and Ramar\'{e} we know that $\left|\Delta(x)\right|\leq\sqrt{x}$ for $x\geq 1$, and $\left|\Delta(x)\right|\leq0.764 x^{\frac{1}{3}}\log{x}$ for $x\geq5$, see\footnote{In~\cite[Theorem 1.2]{BBR} it is stated that the latter inequality is true and sharp for $x\geq9995$, but numerical computation confirms that it is true for all $x\geq 5$.}~\cite[Theorems 1.1 and 1.2]{BBR}. Let
\begin{equation}
\label{eq:DeltaStar}
\Delta^{\ast}(x) \de  \sideset{}{'}\sum_{n\leq x}d(n) - x\left(\log{x}+2\gamma-1\right) - \frac{1}{4},
\end{equation}
where $'$ means that in the case when $x\in\N$ we take $d(x)/2$ instead of $d(x)$. Observe that
\begin{equation}
\label{eq:deltadiference}
\Delta(x)-\Delta^{\ast}(x)=\frac{1}{4} \quad \textrm{for} \quad x\notin\N.
\end{equation}
Using~\eqref{eq:deltadiference} and the previously stated explicit bounds for $\left|\Delta(x)\right|$ we can show that
\begin{equation}
\label{eq:explicitDeltaStar}
\left|\Delta^{\ast}(x)\right| \leq x^{\frac{1}{3}}\log{\left(ex\right)}
\end{equation}
for $x\geq 1$. We will use~\eqref{eq:explicitDeltaStar} in estimating the integrals $\mathcal{E}_2$ and $\mathcal{E}_4$. The resulting error terms are $o(1)$ when $T\to\infty$, therefore, having a hypothetical explicit version of the divisor conjecture would not significantly improve our constants from Theorem~\ref{thm:OurAtkinson}.

G.~F.~Vorono\"{i} proved in 1904 that
\begin{equation}
\label{eq:Voronoi}
\Delta^{\ast}(x) = -\frac{2\sqrt{x}}{\pi}\sum_{n=1}^{\infty}\frac{d(n)}{\sqrt{n}}\left(K_1\left(4\pi\sqrt{nx}\right) + \frac{\pi}{2}Y_1\left(4\pi\sqrt{nx}\right)\right),
\end{equation}
where $Y_1(u)$ and $K_1(u)$ are Bessel and modified Bessel functions of the second kind, respectively, see~\cite[Chapter 3]{Ivic}. Note that this infinite series converges uniformly on every compact subset of $[1,\infty)$ which does not contain an integer, and convergence is bounded on every bounded subset of the same interval. The next lemma gives an explicit error term for $\Delta^{\ast}(x)$ when Bessel functions are replaced by trigonometric functions.

\begin{lemma}
\label{lem:Voronoi}
Let $x\geq x_0 \geq 1$. Then
\[
\Delta^{\ast}(x) = \frac{x^{\frac{1}{4}}}{\pi\sqrt{2}}\sum_{n=1}^{\infty} \frac{d(n)}{n^{\frac{3}{4}}}\left(\cos\left(4\pi\sqrt{nx}-\frac{\pi}{4}\right)-\frac{3}{32\pi\sqrt{nx}}\sin\left(4\pi\sqrt{nx}-\frac{\pi}{4}\right)\right) + \frac{V}{x^{\frac{3}{4}}},
\]
where
\[
\left|V\right| \leq V\left(x_0\right) \de \frac{15\zeta^2\left(\frac{7}{4}\right)}{2^{11}\pi^3\sqrt{2}} + \frac{105\zeta^2\left(\frac{9}{4}\right)}{2^{16}\pi^4\sqrt{2}}x_0^{-\frac{1}{2}} + \frac{\zeta^2\left(\frac{11}{4}\right)}{\pi\sqrt{2}}\left(1+\frac{3}{32\pi\sqrt{x_0}}\right)x_0^{-1}.
\]
\end{lemma}

\begin{proof}
Let $u\geq0$. By~\cite{NemesHankelBessel} we know that
\begin{gather*}
K_1(u) = \sqrt{\frac{\pi}{2u}}e^{-u}R_1(u), \\
Y_1(u) = -\sqrt{\frac{2}{\pi u}}\left(\cos\left(u-\frac{\pi}{4}\right) - \frac{3}{8u}\sin\left(u-\frac{\pi}{4}\right)\right) + R_2(u),
\end{gather*}
where
\[
\left|R_1(u)\right| \leq 1 + \frac{3}{8u}, \quad \left|R_2(u)\right| \leq \frac{15}{128}\sqrt{\frac{2}{\pi}}\left(1+\frac{7}{8u}\right)\frac{1}{u^2\sqrt{u}}.
\]
With this notation~\eqref{eq:Voronoi} implies
\begin{flalign*}
\Delta^{\ast}(x) = \frac{x^{\frac{1}{4}}}{\pi\sqrt{2}}\sum_{n=1}^{\infty} \Bigg(&\frac{d(n)}{n^{\frac{3}{4}}}\left(\cos\left(4\pi\sqrt{nx}-\frac{\pi}{4}\right)-\frac{3}{32\pi\sqrt{nx}}\sin\left(4\pi\sqrt{nx}-\frac{\pi}{4}\right)\right) \\
&- \frac{d(n)}{n^{\frac{3}{4}}}e^{-4\pi\sqrt{nx}}R_1\left(4\pi\sqrt{nx}\right) - \pi\sqrt{2}x^{\frac{1}{4}}\frac{d(n)}{\sqrt{n}}R_2\left(4\pi\sqrt{nx}\right)\Bigg).
\end{flalign*}
Noticing that $e^{-4\pi u}<u^{-4}$ for $u\geq 1$, we obtain
\[
\frac{x^{\frac{1}{4}}}{\pi\sqrt{2}}\sum_{n=1}^{\infty}\frac{d(n)}{n^{\frac{3}{4}}}e^{-4\pi\sqrt{nx}}\left|R_1\left(4\pi\sqrt{nx}\right)\right| \leq \frac{\zeta^2\left(\frac{11}{4}\right)}{\pi\sqrt{2}}\left(1+\frac{3}{32\pi\sqrt{x}}\right)x^{-\frac{7}{4}}.
\]
Also,
\[
\sqrt{x}\sum_{n=1}^{\infty}\frac{d(n)}{\sqrt{n}}\left|R_2\left(4\pi\sqrt{nx}\right)\right| \leq \frac{15}{2^{11}\pi^3\sqrt{2}}\left(\zeta^2\left(\frac{7}{4}\right)x^{-\frac{3}{4}} + \frac{7}{32\pi}\zeta^2\left(\frac{9}{4}\right)x^{-\frac{5}{4}}\right).
\]
Now the statement of Lemma~\ref{lem:Voronoi} clearly follows.
\end{proof}

Vorono\"{i} also proved the truncated version of~\eqref{eq:Voronoi}, namely
\begin{equation}
\label{eq:VoronoiTrun}
\Delta(x) = \frac{x^{\frac{1}{4}}}{\pi\sqrt{2}}\sum_{n\leq N}\frac{d(n)}{n^{\frac{3}{4}}}\cos{\left(4\pi\sqrt{nx}-\frac{\pi}{4}\right)} + O\left(x^{\varepsilon}\right) + O\left(x^{\frac{1}{2}+\varepsilon}N^{-\frac{1}{2}}\right).
\end{equation}
Observe that taking $N=\lfloor{x^{\frac{1}{3}}}\rfloor$ gives $\left|\Delta(x)\right|\ll_{\varepsilon} x^{\frac{1}{3}+\varepsilon}$. In view of Taylor's expansion
\[
f(t,n) = 4\pi\sqrt{\frac{nt}{2\pi}}-\frac{\pi}{4}+O\left(n^{\frac{3}{2}}t^{-\frac{1}{2}}\right),
\]
we can see a strong resemblance between~\eqref{eq:JutilaOrigEst} and~\eqref{eq:VoronoiTrun} for $N=M$ and $x=t/(2\pi)$, save only for the $(-1)^n$ term. However, Jutila was able to ``correct'' the function $\Delta(x)$ to show that~\eqref{eq:VoronoiTrun} with the $(-1)^n$ term is true for $-\Delta(x)+2\Delta(2x)-\frac{1}{2}\Delta(4x)$, see~\cite{Jutila83} and~\cite[pp.~472--473]{Ivic}.

\section{The saddle point lemma}
\label{sec:saddle}

The purpose of this section is to obtain explicit asymptotics of a particular family of trigonometric integrals
\[
\mathcal{E}\left(a,b,k;\varphi,f\right) \de \int_{a}^{b}\varphi(x)e^{2\pi\ie\left(f(x)+kx\right)}\dif{x},
\]
where $\varphi(z)$ and $f(z)$ are complex functions satisfying some special properties, and $a,b,k$ are real numbers with $a<b$, see Proposition~\ref{prop:saddle}. One of these conditions assures that $f''(x)$ is positive for $a\leq x\leq b$. This implies that $f'(x)+k$ is a strictly increasing function. Therefore, if there exists a zero $x_0\in[a,b]$ of $f'(x)+k$, it must be unique and simple. Furthermore, we must also have $f'(x)+k\leq 0$ for $x\in \left[a,x_0\right]$, and $f'(x)+k\geq 0$ for $x\in \left[x_0,b\right]$.

In the following proposition $\D{x, r}$ is used for the open disk centered at $x$ with radius $r$ and $\cD{x, r}$ for the closure of the disk.

\begin{proposition}[Saddle point lemma]
\label{prop:saddle}
Let $f(z)$ and $\varphi(z)$ be two complex functions, and $[a,b]\subset\R$. Assume that:
\begin{enumerate}
    \item For $x\in[a,b]$ the function $f(x)$ is real and $f''(x)>0$;
    \item There exists $\mu(x)\in\mathcal{C}^1\left([a,b]\right)$, such that $\mu(x)>0$, and $f(z)$ and $\varphi(z)$ are holomorphic functions on the set
    \[
    \bigcup_{x\in[a,b]} \cD{x,\mu(x)};
    \]
    \item There exist positive functions $F(x)$ and $\Phi(x)$ defined on $[a,b]$, and positive constants $A_1$, $A_2$ and $A_3$, such that for $x\in[a,b]$ and $z\in\cD{x,\mu(x)}$, we have
    \[
    \left|\varphi(z)\right|\leq A_1\Phi(x), \quad \left|f'(z)\right|\leq A_2\frac{F(x)}{\mu(x)}, \quad
    \frac{1}{\left|f''(z)\right|} \leq A_3\frac{\mu^2(x)}{F(x)}.
    \]
\end{enumerate}
Take $k\in\R$, and let
\begin{equation}
\label{eq:alphacond}
0 < \alpha_0 < \min\left\{\frac{1}{\left(1+\frac{2A_2A_3}{3}\right)\sqrt{2}},\frac{A_3\left(1+2A_2\right)}{2},\left(\frac{48A_3^3}{\pi^3}\right)^{\frac{1}{6}}\right\}.
\end{equation}
\begin{enumerate}
    \item If there exists $x_0\in[a,b]$ such that $f'\left(x_0\right)=-k$, then
\begin{equation}
\label{eq:Emain}
\mathcal{E}\left(a,b,k;\varphi,f\right) = \frac{\varphi\left(x_0\right)}{\sqrt{f''\left(x_0\right)}}e^{2\pi\ie\left(f\left(x_0\right)+kx_0+\frac{1}{8}\right)} + \mathcal{R}_1 + \mathcal{R}_2 + \mathcal{R}_3;
\end{equation}
\item Otherwise, $\mathcal{E}\left(a,b,k;\varphi,f\right) = \mathcal{R}_1 + \mathcal{R}_2$.
\end{enumerate}
Here,
\begin{gather}
\left|\mathcal{R}_1\right| \leq A_1B_1\left(\frac{\Phi(a)}{\left|f'(a)+k\right|+\sqrt{f''(a)}}+\frac{\Phi(b)}{\left|f'(b)+k\right|+\sqrt{f''(b)}}\right), \label{eq:R1} \\
\left|\mathcal{R}_2\right| \leq A_1\int_{a}^{b}\Phi(x)e^{B_2\left(-|k|\mu(x)-F(x)\right)}\left(1+\alpha_0\sqrt{2}\left|\mu'(x)\right|\right)\dif{x}, \label{eq:R2} \\
\left|\mathcal{R}_3\right| \leq A_1\left(A_3^4B_3+\sqrt{A_3}B_4\right)\Phi\left(x_0\right)\mu\left(x_0\right)F^{-\frac{3}{2}}\left(x_0\right), \label{eq:R3}
\end{gather}
with
\begin{gather}
B_1\left(A_2,A_3,\alpha_0\right)\de \frac{1}{1-\beta\left(A_2,A_3,\alpha_0\right)}+\frac{1}{\pi}\exp{\left(-\frac{\pi}{2\left(1-\beta\left(A_2,A_3,\alpha_0\right)\right)}\right)}, \label{eq:B1} \\
B_2\left(A_2,A_3,\alpha_0\right)\de \frac{2\pi\alpha_0^2\left(1-\beta\left(A_2,A_3,\alpha_0\right)\right)}{A_3\left(1+2A_2\right)}, \label{eq:B2} \\
B_3\left(A_2,A_3,\alpha_0\right) \de \frac{48\sqrt{2}}{\pi^4\alpha_0^7}\left(1+\left(1-\beta\left(A_2,A_3,\alpha_0\right)\right)^{-4}\right) \label{eq:B3}
\end{gather}
and
\begin{multline}
\label{eq:B4}
B_4\left(A_2,A_3,\alpha_0\right) \de \frac{1+\alpha_0^3\omega_3\left(A_2,\alpha_0\right)}{\left(2\pi/A_3\right)\left(1-\alpha_0\sqrt{2}\right)} \\
+ \frac{3}{\left(4\pi/ A_3\right)^2}\left(\frac{2\pi A_2}{1-\alpha_0\sqrt{2}}+\omega_3\left(A_2,\alpha_0\right)\sqrt{2}\right) + \frac{15\omega_2\left(A_2,\alpha_0\right)}{\left(4\pi/A_3\right)^3},
\end{multline}
where
\begin{equation}
\label{eq:beta}
\beta\left(A_2,A_3,\alpha_0\right)\de \frac{2\alpha_0\sqrt{2}A_2A_3}{3\left(1-\alpha_0\sqrt{2}\right)},
\end{equation}
and $\omega_2\left(A_2,\alpha_0\right)$ and $\omega_3\left(A_2,\alpha_0\right)$ are defined by~\eqref{eq:omega2} and~\eqref{eq:omega3}, respectively.
\end{proposition}

Before proceeding to the proof of Proposition~\ref{prop:saddle}, we need the following two lemmas. The first one is Taylor's formula for holomorphic functions with an explicit remainder term, stated suitably for our purposes, while the second one is a simple estimation of a particular exponential integral.

\begin{lemma}
\label{lem:taylor}
Let $0<r_1<r_2$. Let $f$ be a holomorphic function on $\cD{z_0,r_2}$ such that $\left|f(z)\right|\leq M_1$ and $\left|f'(z)\right|\leq M_2$ for $z\in\partial\D{z_0,r_2}$. Then
\[
f(z) = \sum_{n=0}^{k} \frac{f^{(n)}\left(z_0\right)}{n!}\left(z-z_0\right)^{n} + T_k(z)
\]
for $z\in\cD{z_0,r_1}$, where
\[
\left|T_k(z)\right| \leq \frac{M_1\left|z-z_0\right|^{k+1}}{\left(1-\frac{r_1}{r_2}\right)r_2^{k+1}}, \quad
\left|T_k(z)\right| \leq \frac{M_2\left|z-z_0\right|^{k+1}}{(k+1)\left(1-\frac{r_1}{r_2}\right)r_2^{k}}.
\]
\end{lemma}

\begin{proof}
Let $\gamma$ be the positively oriented boundary of $\D{z_0,r_2}$. By Cauchy's integral formula for derivatives we have
\begin{gather*}
T_k(z) = \frac{\left(z-z_0\right)^{k+1}}{2\pi\ie}\int_{\gamma} \frac{f(w)}{\left(w-z_0\right)^{k+2}}\sum_{n=0}^{\infty}\left(\frac{z-z_0}{w-z_0}\right)^n \dif{w}, \\
T_k(z) = \frac{\left(z-z_0\right)^{k+1}}{2\pi\ie}\int_{\gamma} \frac{f'(w)}{(k+1)\left(w-z_0\right)^{k+1}}\sum_{n=0}^{\infty}\frac{k+1}{k+1+n}\left(\frac{z-z_0}{w-z_0}\right)^n \dif{w}.
\end{gather*}
Since $w=z_0+r_2e^{\ie\varphi}$, $\varphi\in[0,2\pi]$, and $\left|z-z_0\right|\leq r_1$, $\left|w-z_0\right|=r_2$, the final bounds easily follow from the above exact formulas.
\end{proof}

\begin{lemma}
\label{lem:integral}
Let $A$ be a non-negative real number, and let $B$ and $a$ be positive real numbers. Then
\[
\int_{0}^{\infty} \exp{\left(-Ax-\frac{aB}{4}x^2\right)}\dif{x} \leq \frac{\left(\frac{2}{a}+e^{-\frac{1}{a}}\right)\sqrt{2}}{A+\sqrt{B}}.
\]
\end{lemma}

\begin{proof}
By introducing a new variable $u=Ax+(aB/4)x^2$, we have
\begin{flalign*}
\int_{0}^{\infty} \exp{\left(-Ax-\frac{aB}{4}x^2\right)}\dif{x} &= \int_{0}^{\infty} \frac{e^{-u}\dif{u}}{\sqrt{A^2+aBu}} = \left(\int_{0}^{\frac{1}{a}}+\int_{\frac{1}{a}}^{\infty}\right)\frac{e^{-u}\dif{u}}{\sqrt{A^2+aBu}} \\
& \leq \int_{0}^{\frac{1}{a}} \frac{\dif{u}}{\sqrt{A^2+aBu}} + \frac{1}{\sqrt{A^2+B}}\int_{\frac{1}{a}}^{\infty} e^{-u}\dif{u} \\
& = \frac{2/a}{\sqrt{A^2+B}+A} + \frac{e^{-\frac{1}{a}}}{\sqrt{A^2+B}}.
\end{flalign*}
The final bound follows from the fact that $\sqrt{A^2+B}\geq \left(A+\sqrt{B}\right)/\sqrt{2}$.
\end{proof}

\begin{proof}[Proof of Proposition~\ref{prop:saddle}]
Firstly, assume that there exists a zero $x_0\in[a,b]$ of $f'(x)+k$, and that $\alpha_0$ satisfies the inequality~\eqref{eq:alphacond}. Let
\begin{flalign*}
I_{1}&\de \left\{z=a-(1+\ie)y \colon y\in\left[0,\alpha_0\mu\left(a\right)\right]\right\}, \\
I_{2}&\de \left\{z=x-\alpha_0(1+\ie)\mu(x) \colon x\in\left[a,x_0\right]\right\}, \\
I_{3}&\de \left\{z=x_0+(1+\ie)y \colon y\in\left[-\alpha_0\mu\left(x_0\right),\alpha_0\mu\left(x_0\right)\right]\right\}, \\
I_{4}&\de \left\{z=x+\alpha_0(1+\ie)\mu(x) \colon x\in\left[x_0,b\right]\right\}, \\
I_{5}&\de \left\{z=b-(1+\ie)y \colon y\in\left[-\alpha_0\mu\left(b\right),0\right]\right\},
\end{flalign*}
see Figure~\ref{pic:intervals}.

\begin{figure}[h]
    \includegraphics[scale = 0.5]{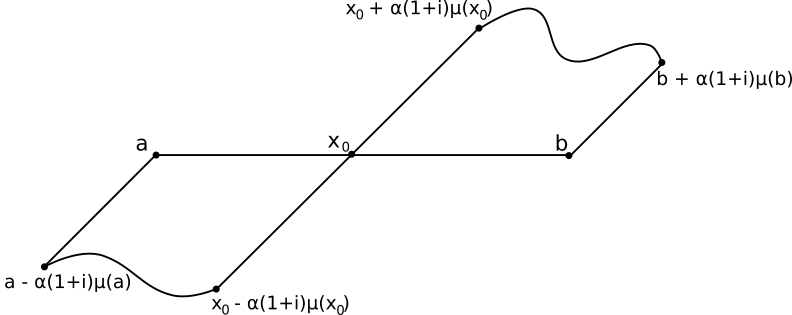}
    \caption{Paths $I_j$ for $j\in\left\{ 1,2,\dots,5\right\}$.}
    \label{pic:intervals}
\end{figure}

For integers $j$, $1\leq j\leq 5$, let us define
\[
\mathcal{I}_j \de \int_{I_j}\varphi(z)e^{2\pi\ie\left(f(z)+kz\right)}\dif{z}.
\]
From~\eqref{eq:alphacond}, $\alpha_0 < \frac{1}{\sqrt{2}}$, so $|\alpha_0(1 + i)\mu(x)| < \mu(x)$ for every $x\in[a,b]$. Hence, $\bigcup_{j=1}^5 I_j$ is an oriented path from $a$ to $b$ in a simply connected domain $\bigcup_{x\in[a,b]} \cD{x,\mu(x)}$ on which $\varphi$ and $f$ are holomorphic functions. Cauchy's theorem then guarantees that $\mathcal{E}\left(a,b,k;\varphi,f\right)=\sum_{j=1}^5\mathcal{I}_j$. Write $\mathcal{R}_1 \de \mathcal{I}_1+\mathcal{I}_5$, $\mathcal{R}_2 \de \mathcal{I}_2+\mathcal{I}_4$, and
\[
\mathcal{R}_3 \de  \mathcal{I}_3 - \frac{\varphi\left(x_0\right)}{\sqrt{f''\left(x_0\right)}}e^{2\pi\ie\left(f\left(x_0\right)+kx_0+\frac{1}{8}\right)}.
\]
With such definitions equation~\eqref{eq:Emain} is true. In the first two paragraphs we are going to prove the stated upper bounds for $\left|\mathcal{R}_1\right|$ and $\left|\mathcal{R}_2\right|$, while the last paragraph is devoted to the most difficult case $\mathcal{R}_3$.

Let $x\in[a,b]$ and $z=x+(1+\ie)y$ with $|y|\leq \alpha_0\mu(x)$. Lemma~\ref{lem:taylor} with $r_1=\alpha_0\sqrt{2}\mu(x)$, $r_2=\mu(x)$, $z_0=x$, $k=2$, and $M_2=A_2F(x)/\mu(x)$ gives
\[
f(z)+kz=f(x)+kx+(1+\ie)\left(f'(x)+k\right)y+\ie f''(x)y^2+\theta_1(y)
\]
with $\left|\theta_1(y)\right|\leq \beta f''(x)y^2$, where $\beta=\beta\left(A_2,A_3,\alpha_0\right)$ is defined by~\eqref{eq:beta}. Observe that~\eqref{eq:alphacond} guarantees that $0<\beta<1$. This implies
\[
\Re\left\{2\pi\ie\left(f(z)+kz\right)\right\} \leq -2\pi\left(f'(x)+k\right)y-2\pi\left(1-\beta\right)f''(x)y^2.
\]
From this inequality we obtain
\begin{gather*}
\left|\mathcal{I}_1\right| \leq A_1\sqrt{2}\Phi(a)\int_{0}^{\alpha\mu(a)} \exp{\left(2\pi\left(f'(a)+k\right)y-2\pi\left(1-\beta\right)f''(a)y^2\right)}\dif{y}, \\
\left|\mathcal{I}_5\right| \leq A_1\sqrt{2}\Phi(b)\int_{-\alpha\mu(b)}^{0} \exp{\left(2\pi\left(f'(b)+k\right)y-2\pi\left(1-\beta\right)f''(b)y^2\right)}\dif{y}.
\end{gather*}
Remembering that $f'(a)+k\leq 0$ and $f'(b)+k \geq 0$, extending the limits of integration, and using Lemma~\ref{lem:integral}, we obtain the upper bound for $\left|\mathcal{I}_1+\mathcal{I}_5\right| = \left|\mathcal{R}_1\right|$ as stated in~\eqref{eq:R1}.

Arguing as before, integration along $I_2$ and $I_4$ gives
\[
\left|\mathcal{I}_2+\mathcal{I}_4\right| \leq A_1\int_{a}^{b} \Phi(x)e^{g(x)}\left(1+\alpha_0\sqrt{2}\left|\mu'(x)\right|\right)\dif{x},
\]
where
\[
g(x) \de -2\pi\alpha_0\left|f'(x)+k\right|\mu(x)-2\pi\left(1-\beta\right)f''(x)\left(\alpha_0\mu(x)\right)^2.
\]
We always have $\left(\alpha_0\mu(x)\right)^2f''(x)\geq \alpha_0^2F(x)/A_3$, which gives
\[
-2\pi\left(1-\beta\right)f''(x)\left(\alpha_0\mu(x)\right)^2 \leq -B_2\left(1+2A_2\right)F(x),
\]
where $B_2=B_2\left(A_2,A_3,\alpha_0\right)$ is defined by~\eqref{eq:B2}. Assume that $|k|\leq 2\left|f'(x)\right|$. Then
\[
-2\pi\left(1-\beta\right)f''(x)\left(\alpha_0\mu(x)\right)^2 \leq B_2\left(-|k|\mu(x)-F(x)\right),
\]
obviously with the same upper bound also for $g(x)$. Now assume that $|k|>2\left|f'(x)\right|$. This gives $\alpha_0\left|f'(x)+k\right|\mu(x)>\frac{\alpha_0}{2}|k|\mu(x)$, which further implies
\[
g(x) \leq -\pi\alpha_0|k|\mu(x) - B_2F(x).
\]
Because $\alpha_0\leq A_3\left(1+2A_2\right)/2$, we have $A_3\left(1+2A_2\right)\geq 2\alpha_0\left(1-\beta\right)$ and thus $\pi\alpha_0\geq B_2$. Therefore, $g(x)\leq B_2\left(-|k|\mu(x)-F(x)\right)$ also in this case. This gives the upper bound for $\left|\mathcal{I}_2+\mathcal{I}_4\right|=\left|\mathcal{R}_2\right|$ as stated in~\eqref{eq:R2}.

For the sake of brevity we shall write $\Phi\left(x_0\right)=\Phi_0$, $F\left(x_0\right)=F_0$ and $\mu\left(x_0\right)=\mu_0$. Let $v\de \alpha_0\mu_0/\left(1+F_0^{1/3}\right)$. Divide the integration along $I_3$ into three parts, $I_{31}$, $I_{32}$ and $I_{33}$, such that
\begin{flalign*}
I_{31}&\de \left\{z=x_0+(1+\ie)y \colon y\in\left[-\alpha_0\mu_0,-v\right]\right\}, \\
I_{32}&\de \left\{z=x_0+(1+\ie)y \colon y\in[-v,v]\right\}, \\
I_{33}&\de \left\{z=x_0+(1+\ie)y \colon y\in\left[v,\alpha_0\mu_0\right]\right\}.
\end{flalign*}
Denote by $\mathcal{I}_{31}$, $\mathcal{I}_{32}$ and $\mathcal{I}_{33}$ the corresponding integrals, therefore $\mathcal{I}_3=\mathcal{I}_{31}+\mathcal{I}_{32}+\mathcal{I}_{33}$. We have
\[
\left|\mathcal{I}_{31}+\mathcal{I}_{33}\right| \leq 2\sqrt{2}A_1\Phi_0\int_{v}^{\alpha_0\mu_0} \exp{\left(-2\pi\left(1-\beta\right)f''\left(x_0\right)y^2\right)}\dif{y}.
\]
Introducing the new variable $u=2\pi\left(1-\beta\right)f''\left(x_0\right)y^2$, we have
\begin{flalign*}
\left|\mathcal{I}_{31}+\mathcal{I}_{33}\right| &\leq 2\sqrt{2}A_1\Phi_0 \int_{2\pi\left(1-\beta\right)f''\left(x_0\right)v^2}^{\infty} \frac{e^{-u}\dif{u}}{\sqrt{8\pi(1-\beta)f''\left(x_0\right)u}} \\
&\leq \frac{A_1A_3\sqrt{2}\Phi_0\mu_0}{2\pi\alpha_0\left(1-\beta\right)}\frac{1+F_0^{\frac{1}{3}}}{F_0}\exp{\left(-\frac{2\pi\alpha_0^2\left(1-\beta\right)F_0}{A_3\left(1+F_0^{1/3}\right)^2}\right)}.
\end{flalign*}
Here we used the inequalities
\[
\frac{1}{f''\left(x_0\right)v} \leq \frac{A_3\mu_0\left(1+F_0^{1/3}\right)}{\alpha_0 F_0}, \quad f''\left(x_0\right)v^2 \geq \frac{\alpha_0^2 F_0}{A_3\left(1+F_0^{1/3}\right)^2},
\]
which follow from the third condition in the proposition and from the definition of $v$. Assume that $F_0\geq1$. Using the inequality $x^{-\frac{2}{3}}e^{-A\sqrt[3]{x}} \leq 6A^{-3}x^{-\frac{3}{2}}$, valid for $x\geq1$ and $A>0$, we obtain
\begin{equation}
\label{eq:I31I33}
\left|\mathcal{I}_{31}+\mathcal{I}_{33}\right| \leq \frac{48\sqrt{2}A_1A_3^4}{\pi^4\alpha_0^7\left(1-\beta\right)^4}\Phi_0\mu_0F_0^{-\frac{3}{2}}.
\end{equation}
Now assume that $0<F_0<1$. Then
\[
\left|\mathcal{I}_{31}+\mathcal{I}_{33}\right| \leq \frac{\sqrt{2}A_1A_3}{\pi\alpha_0\left(1-\beta\right)}\Phi_0\mu_0F_0^{-\frac{3}{2}}
\]
since $F_0^{-1} \leq F_0^{-\frac{3}{2}}$. By~\eqref{eq:alphacond} we always have~\eqref{eq:I31I33}. Now we focus on $\mathcal{I}_{32}$. Let $z=x_0+(1+\ie)y$ with $|y|\leq v$. Lemma~\ref{lem:taylor} with $r_1=v\sqrt{2}$, $r_2=\mu_0$, $z_0=x_0$, $k=3$, and $M_2=A_2F_0/\mu_0$ gives
\[
f(z)+kz = f\left(x_0\right)+kx_0+\ie f''\left(x_0\right)y^2 - \frac{1-\ie}{3}f'''\left(x_0\right)y^3 + \theta_2(y)
\]
with
\[
\left|\theta_2(y)\right|\leq \frac{A_2F_0}{\mu_0^4\left(1-\alpha_0\sqrt{2}\right)}y^4.
\]
By the ordinary Taylor's formula we also have $e^{w}=1+w+\theta_3(w)$ with $\left|\theta_3(w)\right|\leq \frac{1}{2}e^{|w|}|w|^2$ for $w\in\C$. Lemma~\ref{lem:taylor} with $r_1=v\sqrt{2}$, $r_2=\mu_0$, $z_0=x_0$, $k=1$, and $M_1=A_1\Phi_0$ also implies
\[
\varphi(z) = \varphi\left(x_0\right) + (1+\ie)\varphi'\left(x_0\right)y + \theta_4(y), \quad \left|\theta_4(y)\right| \leq \frac{2A_1\Phi_0}{\mu_0^2\left(1-\alpha_0\sqrt{2}\right)}y^2.
\]
Define $w\de-\left(2\pi/3\right)(1+\ie)f'''\left(x_0\right)y^3+2\pi\ie\theta_2(y)$. Then
\begin{flalign}
\label{eq:mainthird}
\varphi(z)e^{2\pi\ie\left(f(z)+kz\right)} &= e^{2\pi\ie\left(f\left(x_0\right)+kx_0\right)}\Bigg(\varphi\left(x_0\right)e^{-2\pi f''\left(x_0\right)y^2} \nonumber \\
&+ (1+\ie)\left(\varphi'\left(x_0\right)y - \frac{2\pi}{3}f'''\left(x_0\right)\varphi\left(x_0\right)y^3\right)e^{-2\pi f''\left(x_0\right)y^2} \nonumber \\
&+ \Omega(y)e^{-2\pi f''\left(x_0\right)y^2} \Bigg),
\end{flalign}
where
\begin{flalign*}
\Omega(y) \de \left(2\pi\ie\theta_2(y)+\theta_3(w)\right)\varphi\left(x_0\right) &+ (1+\ie)\varphi'\left(x_0\right)\left(w+\theta_3(w)\right)y \\
&+ \left(1+w+\theta_3(w)\right)\theta_4(y).
\end{flalign*}
Since the middle term in~\eqref{eq:mainthird} is an odd function in the variable $y$, we obtain
\[
\mathcal{I}_{32} = \frac{\varphi\left(x_0\right)}{\sqrt{f''\left(x_0\right)}}e^{2\pi\ie\left(f\left(x_0\right)+kx_0+\frac{1}{8}\right)} + \widetilde{\mathcal{R}}_{31} + \widetilde{\mathcal{R}}_{32},
\]
where
\begin{gather*}
\widetilde{\mathcal{R}}_{31} \de -2\sqrt{2}e^{2\pi\ie\left(f\left(x_0\right)+kx_0+\frac{1}{8}\right)}\varphi\left(x_0\right)\int_{v}^{\infty}e^{-2\pi f''\left(x_0\right)y^2}\dif{y}, \\
\widetilde{\mathcal{R}}_{32} \de (1+\ie) e^{2\pi\ie\left(f\left(x_0\right)+kx_0\right)} \int_{-v}^{v}\Omega(y)e^{-2\pi f''\left(x_0\right)y^2}\dif{y}.
\end{gather*}
With this definitions we have $\mathcal{R}_3=\mathcal{I}_{31}+\mathcal{I}_{33}+\widetilde{\mathcal{R}}_{31} + \widetilde{\mathcal{R}}_{32}$, where $\mathcal{R}_3$ is from~\eqref{eq:Emain}. Arguing as for the bound of $\left|\mathcal{I}_{31}+\mathcal{I}_{33}\right|$, we see that
\[
\left|\widetilde{\mathcal{R}}_{31}\right| \leq \frac{48\sqrt{2}A_1A_3^4}{\pi^4\alpha_0^7}\Phi_0\mu_0F_0^{-\frac{3}{2}},
\]
which implies
\begin{equation}
\label{eq:part1}
\left|\mathcal{I}_{31}+\mathcal{I}_{33}+\widetilde{\mathcal{R}}_{31}\right| \leq A_1A_3^4B_3\Phi_0\mu_0F_0^{-\frac{3}{2}},
\end{equation}
where $B_3=B_3\left(A_2,A_3,\alpha_0\right)$ is defined by~\eqref{eq:B3}. Therefore, it remains to estimate $\widetilde{\mathcal{R}}_{32}$. Because $\left|\varphi'\left(x_0\right)\right|\leq A_1\Phi_0/\mu_0$, we have
\begin{flalign}
\label{eq:Omega1}
\left|\Omega(y)\right| &\leq \frac{2\pi A_1A_2\Phi_0 F_0}{\left(1-\alpha_0\sqrt{2}\right)\mu_0^4}y^4 + A_1\Phi_0\left|\theta_3(w)\right| +   \frac{A_1\Phi_0\sqrt{2}}{\mu_0}|y|\left(|w|+\left|\theta_3(w)\right|\right) \nonumber \\
&+\frac{2A_1\Phi_0}{\left(1-\alpha_0\sqrt{2}\right)\mu_0^2}\left(1+|w|+\left|\theta_3(w)\right|\right)y^2.
\end{flalign}
Because $\left|f'''\left(x_0\right)\right|\leq 2A_2F_0/\mu_0^3$, we have
\begin{flalign}
|w| &\leq \frac{4\pi\sqrt{2}A_2F_0}{3\mu_0^3}|y|^3 + \frac{2\pi A_2F_0}{\left(1-\alpha_0\sqrt{2}\right)\mu_0^4}y^4 \nonumber \\
&\leq \pi\alpha_0^3A_2\left(\frac{4\sqrt{2}}{3}+\frac{27\alpha_0}{2^7\left(1-\alpha_0\sqrt{2}\right)}\right)\ed \omega_1\left(A_2,\alpha_0\right), \label{eq:omega1}
\end{flalign}
where we also used the fact that $2F_0/\left(1+F_0^{1/3}\right)^{4}$ is not greater than $27/2^7$ since $F_0\geq0$. Both previous inequalities imply
\begin{equation}
\label{eq:theta3}
\left|\theta_3(w)\right| \leq \omega_2\left(A_2,\alpha_0\right)\frac{F_0^2}{\mu_0^6} y^6,
\end{equation}
where
\begin{equation}
\label{eq:omega2}
\omega_2\left(A_2,\alpha_0\right) \de e^{\omega_1\left(A_2,\alpha_0\right)}\left(\frac{4\pi A_2}{3}\right)^2\left(1+\frac{3\alpha_0}{\sqrt{2}-2\alpha_0}\left(1+\frac{3\sqrt{2}\alpha_0}{8\left(1-\alpha_0\sqrt{2}\right)}\right)\right),
\end{equation}
and $\omega_1\left(A_2,\alpha_0\right)$ is defined by~\eqref{eq:omega1}. Therefore,
\begin{equation}
\label{eq:wplustheta}
|w|+\left|\theta_3(w)\right| \leq \omega_3\left(A_2,\alpha_0\right)\frac{F_0}{\mu_0^3}|y|^3 \leq \alpha_0^3\omega_3\left(A_2,\alpha_0\right),
\end{equation}
where
\begin{equation}
\label{eq:omega3}
\omega_3\left(A_2,\alpha_0\right) \de \frac{4\pi\sqrt{2}A_2}{3}+\frac{2\pi\alpha_0 A_2}{1-\alpha_0\sqrt{2}}+\alpha_0^3\omega_2\left(A_2,\alpha_0\right).
\end{equation}
Taking into account~\eqref{eq:theta3} and the first inequality in~\eqref{eq:wplustheta}, together with~\eqref{eq:Omega1} we finally arrive to
\[
\left|\Omega(y)\right| \leq \frac{2A_1\left(1+\alpha_0^3\omega_3\right)}{1-\alpha_0\sqrt{2}}\frac{\Phi_0}{\mu_0^2}y^2 + \left(\frac{2\pi A_1A_2}{1-\alpha_0\sqrt{2}}+A_1\omega_3\sqrt{2}\right)\frac{\Phi_0 F_0}{\mu_0^4}y^4 + A_1\omega_2\frac{\Phi_0F_0^2}{\mu_0^6}y^6.
\]
Using
\begin{flalign*}
\int_{-\infty}^{\infty} |y|^{m}e^{-2\pi f''\left(x_0\right)y^2}\dif{y} &= \left(2\pi f''\left(x_0\right)\right)^{-\frac{m+1}{2}}\Gamma\left(\frac{m+1}{2}\right) \\
&\leq \left(\frac{2\pi}{A_3}\right)^{-\frac{m+1}{2}}\Gamma\left(\frac{m+1}{2}\right) \mu_0^{m+1}F_0^{-\frac{m+1}{2}}
\end{flalign*}
for $m\geq1$, see~\cite[Equation \textbf{3.326}(2)]{GradRyz}, we obtain at once that
\begin{equation}
\label{eq:part2}
\left|\widetilde{\mathcal{R}}_{32}\right| \leq A_1\sqrt{A_3}B_4\Phi_0\mu_0F_0^{-\frac{3}{2}},
\end{equation}
where $B_4=B_4\left(A_2,A_3,\alpha_0\right)$ is defined by~\eqref{eq:B4}. Inequalities~\eqref{eq:part1} and~\eqref{eq:part2} give the final bound in~\eqref{eq:R3}, which concludes the proof in the case that there exists a zero $x_0\in[a,b]$ of $f'(x)+k$. If this is not the case, then we already proved that $\mathcal{E}(a,b,k;\varphi,f)=\mathcal{R}_1+\mathcal{R}_2$, with the same bounds~\eqref{eq:R1} and~\eqref{eq:R2} if $f'(a)+k<0$. But if $f'(a)+k>0$, then we integrate in a similar fashion from $a$ to $a+(1+\ie)\alpha_0\mu(a)$ to $b+(1+\ie)\alpha_0\mu(b)$ to $b$, along $x+(1+\ie)\alpha_0\mu(x)$ in the middle. We can easily observe that the stated bounds are also true in this case. The proof is thus complete.
\end{proof}

\section{Estimation of two integrals}
\label{sec:TwoIntegrals}

A first glance at the terms $\mathcal{E}_{j}$ reveals that we deal with two main types of integrals, namely the one in the first three terms, and the other one in the last term. In this section we estimate these two families by means of Proposition~\ref{prop:saddle}.

\begin{proposition}
\label{prop:FirstIntegral}
Let $\alpha, \beta, \gamma, a, b, k, T, A',A''$ be positive real numbers such that $\alpha \ne 1$, $0<a\leq\min\left\{1,T/\left(8\pi k\right)\right\}$, $b \geq T$, $k \geq 1$, and $T\geq T_0\geq 1$. Additionally, let $1\leq k\leq A''T$ or $k\geq A'T$. Define
\begin{gather*}
U(T,k) \de \sqrt{\frac{T}{2\pi k} + \frac{1}{4}}, \quad \eta(k)\de U(1,k)-\frac{1}{2}, \\
\delta(T) \de \left\{
\begin{array}{ll}
    \log{T}, & \gamma-\alpha-\beta=-1,  \\
    \frac{T^{1+\gamma-\alpha-\beta}-1}{1+\gamma-\alpha-\beta}, &  \gamma-\alpha-\beta\neq -1.
\end{array}
\right.
\end{gather*}
Then
\begin{flalign}
\label{eq:FirstIntegral}
& \int_{a}^{b} {y^{-\alpha}(1+y)^{-\beta}\left(\log{\frac{1+y}{y}}\right)^{-\gamma} \exp\left(\ie\left(T \log{\frac{1+y}{y}} + 2\pi ky\right)\right) \dif{y}} \nonumber \\
&= 2^{-\frac{1}{2}-\gamma}k^{-\frac{3}{4}}\left(\frac{T}{2\pi}\right)^{\frac{1}{4}}\left(1+\frac{k\pi}{2T}\right)^{-\frac{1}{4}}\left(U(T,k) - \frac{1}{2}\right)^{-\alpha} \left(U(T,k) + \frac{1}{2}\right)^{-\beta} \times \nonumber \\
&\times\left( \arsinh\sqrt{\frac{\pi k}{2 T}}\right)^{-\gamma}\exp\left(\ie\left(2T \arsinh\sqrt{\frac{\pi k}{2 T}} + \sqrt{\left(\pi k\right)^2+2\pi kT} - \pi k + \frac{\pi }{4}\right)\right) \nonumber \\
&+ \mathcal{U}\left(\alpha,\beta,\gamma,a,b,k,T,A',A''\right),
\end{flalign}
where
\[
\left|\mathcal{U}\right| \leq \mathcal{U}_1a^{1-\alpha}T^{-1} + \mathcal{U}_2b^{\gamma-\alpha-\beta} k^{-1} + \mathcal{U}_3 + \mathcal{U}_4,
\]
and
\begin{gather}
\mathcal{U}_1\left(\alpha,\beta,\gamma,a\right) \de 2^{\left|\gamma-\beta\right|+1+\alpha+\beta-\gamma}3^{\gamma}\left(\frac{34.26(1+a)}{3-a} + \frac{39038}{\left|\alpha-1\right|}\right), \label{eq:U1} \\
\mathcal{U}_2 \left(\alpha,\beta,\gamma\right) \de 2.962\cdot2^{\left|\gamma-\beta\right|+\alpha+\beta-\gamma}3^{\gamma}, \label{eq:U2}
\end{gather}
\begin{flalign}
\label{eq:U3}
\mathcal{U}_3\left(\alpha,\beta,\gamma,b,k,T\right) &\de 2^{\left|\gamma-\beta\right|+\alpha+\beta-\gamma}3^{\gamma}\bigg(1.02\delta(T)\exp{\left(-9.58\cdot10^{-6}\sqrt{Tk}\right)} \nonumber \\
&+ \frac{212232.3}{k}\max\left\{b^{\gamma-\alpha-\beta},T^{\gamma-\alpha-\beta}\right\}\exp{\left(-4.79\cdot10^{-6}Tk\right)}\bigg),
\end{flalign}
together with: If $1\leq k\leq A''T$, then
\begin{equation*}
\mathcal{U}_4\left(\alpha,\beta,\gamma,A'',k,T\right) \de 2^{\alpha+\beta-\gamma}3^{\gamma}\left(\mathcal{U}_{41} +\mathcal{U}_{42}\right)T^{\frac{1}{2}\left(\gamma-\alpha-\beta\right)-\frac{1}{4}}k^{-\frac{1}{2}\left(\gamma-\alpha-\beta\right)-\frac{5}{4}},
\end{equation*}
where
\begin{multline}
\label{eq:U41}
\mathcal{U}_{41}\left(\alpha,\beta,\gamma,A''\right)\de \frac{1.02\left(A''\right)^{\frac{5}{4}}\max\left\{1,\left(A''\right)^{\frac{1}{2}(\gamma-\alpha-\beta)}\right\}}{\left|\alpha-1\right|} \\
\times 2^{\left|\gamma-\beta\right|}\left(38386.4\left(\left|\gamma-\alpha-\beta\right|+3\right)\right)^{\frac{1}{2}\left|\gamma-\alpha-\beta\right|+\frac{3}{2}},
\end{multline}
\begin{multline}
\label{eq:U42}
\mathcal{U}_{42}\left(\alpha,\beta,\gamma,A''\right)\de 8.75\cdot10^{19}\left(1+\frac{1}{\eta\left(A''\right)}\right)^{\left|\gamma-\beta\right|}\left(\sqrt{A''}+\frac{1}{\sqrt{2\pi}}\right)^{\frac{3}{2}} \\
\times \max\left\{(2\pi)^{-\frac{1}{2}\left(\gamma-\alpha-\beta\right)},\left(\sqrt{A''}\eta\left(A''\right)\right)^{\gamma-\alpha-\beta}\right\}.
\end{multline}
If $k\geq A'T$, then
\begin{equation*}
\mathcal{U}_4 \left(\alpha,\beta,\gamma,A',T_0,k,T\right) \de 2^{\alpha+\beta-\gamma}3^{\gamma}\left(\mathcal{U}_{43}
+\mathcal{U}_{44}\right)T^{-\frac{1}{2}-\alpha}k^{\alpha-1},
\end{equation*}
where
\begin{gather}
\mathcal{U}_{43}\left(\alpha,\beta,\gamma,A',T_0\right)\de 2^{\left|\gamma-\beta\right|}\left(A'\right)^{-\alpha}\left(\frac{3.975\cdot10^7}{\left|\alpha-1\right|}A'+\frac{1.63\cdot10^{10}}{\sqrt{T_0}}\right), \label{eq:U43} \\
\mathcal{U}_{44}\left(\alpha,\beta,\gamma,A'\right)\de 3.5\cdot10^{19}\left(A'\eta\left(A'\right)\right)^{-\alpha}\left(1+\frac{1}{2\pi A'}\right)^{\left|\gamma-\beta\right|+\frac{3}{2}}. \label{eq:U44}
\end{gather}
We also have
\begin{equation}
\label{eq:FirstIntegral2}
\int_{a}^{b} {y^{-\alpha}(1+y)^{-\beta}\left(\log{\frac{1+y}{y}}\right)^{-\gamma} \exp\left(\ie\left(T \log{\frac{1+y}{y}} - 2\pi ky\right)\right) \dif{y}} = \widehat{\mathcal{U}},
\end{equation}
where
\[
\left|\widehat{\mathcal{U}}\right| \leq \mathcal{U}_1a^{1-\alpha}T^{-1} + \mathcal{U}_2b^{\gamma-\alpha-\beta} k^{-1} + \mathcal{U}_3 + \widehat{\mathcal{U}}_4,
\]
together with: If $1\leq k\leq A''T$, then
\begin{equation*}
\widehat{\mathcal{U}}_4\left(\alpha,\beta,\gamma,A'',k,T\right) \de 2^{\alpha+\beta-\gamma}3^{\gamma}\mathcal{U}_{41}T^{\frac{1}{2}\left(\gamma-\alpha-\beta\right)-\frac{1}{4}}k^{-\frac{1}{2}\left(\gamma-\alpha-\beta\right)-\frac{5}{4}}.
\end{equation*}
If $k\geq A'T$, then
\begin{equation*}
\widehat{\mathcal{U}}_4\left(\alpha,\beta,\gamma,A',T_0,k,T\right) \de 2^{\alpha+\beta-\gamma}3^{\gamma}\mathcal{U}_{43}T^{-\frac{1}{2}-\alpha}k^{\alpha-1}.
\end{equation*}
\end{proposition}

In the proof we are using Proposition~\ref{prop:saddle} with the following functions:
\begin{gather}
\varphi(z) = z^{-\alpha} (1+z)^{-\beta} \left( \log{\frac{1+z}{z}} \right)^{-\gamma}, \quad f(z) = \frac{T}{2\pi} \log{\frac{1 + z}{z}}, \label{eq:AuxFun1} \\
\Phi(x) = x^{-\alpha} (1 + x)^{\gamma - \beta}, \quad F(x) = \frac{T}{1+x}, \quad \mu(x) = \frac{x}{2}, \label{eq:AuxFun2}
\end{gather}
where we are taking the principal branch of the logarithm. Then $f(z)$ and $\varphi(z)$ are holomorphic functions on the set $\bigcup_{x\in[a,b]}\cD{x,\mu(x)}$, where $a$ and $b$ are from Proposition~\ref{prop:FirstIntegral}. Also, for $x\in[a,b]$ the function $f(x)$ is real and
\[
f''(x) = \frac{T(2x+1)}{2\pi x^2(1+x)^2} > 0.
\]

Firstly, let us provide the constants $A_1$, $A_2$, and $A_3$ from Proposition~\ref{prop:saddle} for this particular set of functions.

\begin{lemma}
\label{lem:ConstantsFirstIntegral}
With $A_1=2^{\alpha+\beta-\gamma}3^{\gamma}$, $A_2=1/\pi$ and $A_3=27\pi$, the conditions of Proposition~\ref{prop:saddle} are true for the functions defined in~\eqref{eq:AuxFun1} and~\eqref{eq:AuxFun2}.
\end{lemma}

\begin{proof}
We are going to show that $\left|\varphi(z)\right|\leq A_1\Phi(x)$, $\left|f'(z)\right| \leq A_2 F(x)/\mu(x)$, and $1/|f''(z)| \leq A_3\mu^2(x)/F(x)$ for $x\in[a,b]$ and $z\in\cD{x,x/2}$.

As $x/2\leq\Re\{z\}\leq|z|$ and $1+2/(3x)\leq 1+\Re\{1/z\}\leq\left|1+1/z\right|$, where we also used the fact that $z\mapsto 1/z$ maps $\cD{x,x/2}$ bijectively to $\cD{4/(3x),2/(3x)}$, we have
\[
\left|\varphi(z)\right| = |z|^{-\alpha} |1+z|^{-\beta} \left| \log{\frac{1+z}{z}} \right|^{-\gamma} \leq \left(\frac{x}{2}\right)^{-\alpha} \left(1 + \frac{x}{2}\right)^{-\beta} \log\left( 1 + \frac{2}{3x} \right)^{-\gamma}.
\]
Therefore,
\[
\frac{\left|\varphi(z)\right|}{\Phi(x)} \leq 2^{\alpha} \left(\frac{1+x}{1+ x/2}\right)^{\beta}\left( (1+x)\log{\left(1+\frac{2}{3x}\right)} \right)^{-\gamma} \leq 2^{\alpha}2^{\beta}\left(\frac{2}{3}\right)^{-\gamma},
\]
thus proving the first bound.

Similarly,
\[
\left|f'(z)\right| = \frac{T}{2\pi|z|\cdot|1+z|} \leq \frac{T}{\pi x(1+\frac{x}{2})}.
\]
Therefore,
\[
\frac{\left|f'(z)\right|\mu(x)}{F(x)} \leq \frac{1+x}{\pi(2+x)} \leq \frac{1}{\pi},
\]
thus proving the second bound.

Proceeding to the third constant, we have
\begin{flalign*}
\frac{1}{\left|f''(z)\right|} &= \frac{2\pi}{T}\left|\frac{z^2(1+z)^2}{2z+1}\right| = \frac{2\pi}{T}|z|^2\cdot|1+z|\cdot\left|\frac{1}{2}+\frac{1}{2+4z}\right| \\
&\leq \frac{2\pi}{T}\left(\frac{3x}{2}\right)^2\left(1+\frac{3x}{2}\right)\left(\frac{1}{2}+\frac{1}{2+2x}\right) \leq \frac{9\pi\left(2+3x\right)}{4T}x^2
\end{flalign*}
since $z\mapsto1/(2+4z)$ maps $\cD{x,x/2}$ bijectively to
\[
\cD{\frac{1+2x}{2(1+x)(1+3x)},\frac{x}{2(1+x)(1+3x)}}.
\]
Therefore,
\[
\frac{F(x)}{\left|f''(z)\right|\mu^2(x)} \leq 9\pi\frac{2+3x}{1+x}\leq 27\pi,
\]
thus completing the proof of the lemma.
\end{proof}

\begin{proof}[Proof of Proposition~\ref{prop:FirstIntegral}]
Let
\[
x_0 \de U(T,k) - \frac{1}{2}.
\]
Then $f'\left(x_0\right)=-k$ and $x_0 \in [a,b]$. According to Proposition~\ref{prop:saddle}, we have
\begin{equation}
\label{eq:EFirstIntegral}
\mathcal{E}\left(a,b,k;\varphi,f\right)= \frac{\varphi\left(x_0\right)}{\sqrt{f''\left(x_0\right)}}e^{2\pi\ie\left(f\left(x_0\right)+kx_0+\frac{1}{8}\right)} + \mathcal{R}_1 + \mathcal{R}_2 + \mathcal{R}_3,
\end{equation}
where $\mathcal{E}\left(a,b,k;\varphi,f\right)$ is exactly the integral from Proposition~\ref{prop:FirstIntegral}. We need to estimate each remainder term in~\eqref{eq:EFirstIntegral}. We are using estimates~\eqref{eq:R1},~\eqref{eq:R2} and~\eqref{eq:R3} with
\begin{equation*}
0 < \alpha_0 < \min\left\{ \frac{1}{19\sqrt{2}}, 27\left(1+ \frac{\pi}{2}\right), 3\cdot 6^{2/3} \right\} = \frac{1}{19\sqrt{2}}
\end{equation*}
and constants $A_1$, $A_2$, and $A_3$ from Lemma~\ref{lem:ConstantsFirstIntegral}. We choose $\alpha_0\de 0.645/\left(19\sqrt{2}\right)$ which, rounded to 5 decimal places, is equal to the minimum point of $B_3=B_3\left(A_2,A_3,\alpha_0\right)$ from~\eqref{eq:B3}. We seek to minimize $B_3$ as it appears to be the largest value among $B_i=B_i\left(A_2,A_3,\alpha_0\right)$ for $i\in\{1,3,4\}$. We will see that these values contribute to~\eqref{eq:U42} and~\eqref{eq:U44}, which are significant terms in~\eqref{eq:FirstIntegral}.

As
\[
f''\left(x_0\right) = \frac{4\pi k^2}{T}U(T,k)
\]
and
\begin{flalign*}
\log{\left(\frac{1 + x_0}{x_0}\right)} &= \log{\left(\frac{\sqrt{\frac{T}{2\pi k} + \frac{1}{4}} + \frac{1}{2}}{\sqrt{\frac{T}{2\pi k} + \frac{1}{4}} - \frac{1}{2}} \right)} = \log{\left(\frac{\sqrt{1 + \frac{\pi k}{2T}} + \sqrt{\frac{\pi k}{2T}}}{\sqrt{1 + \frac{\pi k}{2T}} - \sqrt{\frac{\pi k}{2T}}}\right)} \\
& =
2\log{\left(\sqrt{1 + \frac{\pi k}{2T}} + \sqrt{\frac{\pi k}{2T}}\right)} = 2 \arsinh\sqrt{\frac{\pi k}{2 T}},
\end{flalign*}
we can easily verify that the main parts in~\eqref{eq:FirstIntegral} and~\eqref{eq:EFirstIntegral} coincide. In what follows are estimations of the terms $\mathcal{R}_i$.

Beginning with $\mathcal{R}_1$, our condition on $a$ implies
\begin{flalign}
\label{conda}
|f'(a)+k| = -f'(a)-k &= \frac{T}{2\pi a (a+1)} - k \nonumber \\
&\geq \frac{T}{2\pi a (a+1)} - \frac{T}{8 \pi a} = \frac{T(3-a)}{8\pi a(a+1)}.
\end{flalign}
Therefore,
\begin{equation*}
\frac{\Phi(a)}{\left|f'(a)+k\right|+\sqrt{f''(a)}} \leq \frac{a^{-\alpha}(1+a)^{\gamma-\beta}}{\left|f'(a)+k\right|} \leq \frac{8\pi(a+1)2^{\left|\gamma-\beta\right|}}{3-a} a^{1-\alpha}T^{-1}.
\end{equation*}
Similarly, for $b \geq T$ we have
\begin{equation}
\label{condb}
f'(b) + k = k - \frac{T}{2 \pi b(b+1)} \geq k - \frac{1}{2 \pi (T+1)} \geq k - \frac{1}{4\pi} \geq \left( 1 - \frac{1}{4\pi} \right) k
\end{equation}
and therefore
\begin{equation*}
\frac{\Phi(b)}{\left|f'(b)+k\right|+\sqrt{f''(b)}} \leq \frac{b^{-\alpha}(1+b)^{\gamma-\beta}}{f'(b)+k} \leq \frac{2^{\left|\gamma-\beta\right|}}{1 - \frac{1}{4\pi}} b^{\gamma-\alpha-\beta} k^{-1}.
\end{equation*}
Hence,
\begin{equation*}
|\mathcal{R}_1| \leq 4\pi A_1B_1 \left(\frac{2^{\left|\gamma-\beta\right|+1}(1+a)}{3-a}a^{1-\alpha}T^{-1}+\frac{2^{\left|\gamma-\beta\right|}}{4\pi-1}b^{\gamma-\alpha-\beta}k^{-1}\right).
\end{equation*}
In the final estimate for $\mathcal{U}$ this bound contributes~\eqref{eq:U2} and the first term in~\eqref{eq:U1}.

Proceeding to $\mathcal{R}_2$, we have
\begin{flalign*}
\left|\mathcal{R}_2\right| &\leq A_1\int_{a}^{b}\Phi(x)e^{B_2\left(-|k|\mu(x)-F(x)\right)}\left(1+\alpha_0\sqrt{2}\left|\mu'(x)\right|\right)\dif{x} \nonumber \\
&= A_1\left(1+\frac{\alpha_0 }{\sqrt{2}}\right)\left(\int_{a}^{1}+\int_{1}^{b}\right) x^{-\alpha} (1+x)^{\gamma - \beta}e^{-B_2\left( \frac{|k|x}{2}+\frac{T}{1+x}\right) }\dif{x},
\end{flalign*}
which gives
\begin{equation}
\label{eq:R2FirstIntegral}
\left|\mathcal{R}_2\right| \leq A_1\left(1+\frac{\alpha_0 }{\sqrt{2}}\right)2^{\left|\gamma-\beta\right|}\left(\int_{a}^{1}x^{-\alpha}e^{-\frac{B_2(|k|x+T)}{2}}\dif{x}+\int_{1}^{b}x^{\gamma-\alpha-\beta}e^{-\frac{B_2\left(|k|x+T/x\right)}{2}}\dif{x}\right).
\end{equation}
Note that $k\geq1$. We are going to estimate the first integral. Having
\[
\int_{a}^{1} x^{-\alpha}e^{-\frac{B_2(kx+T)}{2}}\dif{x} \leq e^{-\frac{B_2}{2} T}\int_{a}^{1} {x^{-\alpha} \dif{x}} \leq \frac{e^{-\frac{B_2}{2} T} a^{1 - \alpha}}{\left|\alpha-1\right|} + \frac{e^{-\frac{B_2}{2} T}}{\left|\alpha-1\right|},
\]
this integral is therefore less than
\[
\frac{2}{B_2 e\left|\alpha-1\right|} T^{-1} a^{1 - \alpha} +  \frac{1}{\left|\alpha-1\right|}\left( \frac{|\gamma - \alpha - \beta| + 3}{B_2 e} \right)^{\frac{1}{2}|\gamma - \alpha - \beta| + \frac{3}{2}} T^{-\frac{1}{2}|\gamma - \alpha - \beta| - \frac{3}{2}}.
\]
Here we used the inequality
\begin{equation}
\label{eq:exponential}
e^{-A T} \leq \left( \frac{\lambda}{AeT} \right)^{\lambda},
\end{equation}
which is valid for positive $A$, $T$ and $\lambda$. If $1\leq k\leq A''T$, then
\begin{multline}
\label{eq:R2PartOneFirstIntegral}
\int_{a}^{1} {x^{-\alpha} e^{ -\frac{B_2(kx+T)}{2}} \dif{x}} \leq \frac{2}{B_2 e\left|\alpha-1\right|} T^{-1} a^{1 - \alpha} + \frac{\left(A''\right)^{\frac{5}{4}}\max\left\{1,\left(A''\right)^{\frac{1}{2}(\gamma-\alpha-\beta)}\right\}}{\left|\alpha-1\right|} \\
\times\left(\frac{|\gamma - \alpha - \beta| + 3}{B_2 e} \right)^{\frac{1}{2}|\gamma - \alpha - \beta| + \frac{3}{2}}T^{\frac{1}{2}(\gamma - \alpha - \beta) - \frac{1}{4}} k^{-\frac{1}{2}(\gamma - \alpha - \beta) - \frac{5}{4}}.
\end{multline}
If $k\geq A'T$, then we obtain
\begin{flalign}
\label{eq:R2PartTwoFirstIntegral}
\int_{a}^{1} {x^{-\alpha} e^{ -\frac{B_2(kx+T)}{2}} \dif{x}} &\leq e^{-\frac{B_2 T}{2}} \left(\int_{a}^{A'T/k} {x^{-\alpha} dx} + \left(\frac{k}{A'T} \right)^{\alpha} \int_{A'T/k}^{\infty} {e^{-\frac{B_2 k x}{2}} \dif{x}}\right) \nonumber \\
&\leq e^{-\frac{B_2 T}{2}} \left(\frac{1}{\left|\alpha-1\right|} a^{1 - \alpha} + \frac{\left(A'\right)^{1-\alpha}T^{1 - \alpha} k^{\alpha - 1}}{\left|\alpha-1\right|} + \frac{2T^{-\alpha} k^{\alpha - 1}}{B_2\left(A'\right)^\alpha} \right) \nonumber \\
&\leq \frac{2}{e B_2 \left|\alpha-1\right|} T^{-1} a^{1 - \alpha} \nonumber \\
&+ \frac{1}{\left(A'\right)^\alpha}\left(\frac{A'}{\left|\alpha-1\right|}\left(\frac{3}{e B_2} \right)^{3/2} + \frac{4}{eB_2^2\sqrt{T_0}}\right) T^{- \frac{1}{2}-\alpha} k^{\alpha - 1},
\end{flalign}
where we used~\eqref{eq:exponential} in order to obtain the last inequality. Turning to the estimation of the second integral in~\eqref{eq:R2FirstIntegral}, we have
\begin{flalign}
\label{eq:R2PartThreeFirstIntegral}
\int_{1}^{b}x^{\gamma-\alpha-\beta}e^{-\frac{B_2\left(kx+T/x\right)}{2}}\dif{x} &\leq \int_{1}^{T} {x^{\gamma - \alpha - \beta} e^{-B_2 \sqrt{T k}} \dif{x}} + \int_{T}^{b} {x^{\gamma - \alpha - \beta} e^{-\frac{B_2kx}{2}} \dif{x}} \nonumber \\
&\leq \delta(T)e^{-B_2 \sqrt{Tk}} + \frac{2}{B_{2}k}\max\left\{b^{\gamma - \alpha - \beta}, T^{\gamma - \alpha - \beta}\right\}e^{-\frac{B_2Tk}{2}},
\end{flalign}
where in the first inequality we used the fact that the maximum of $-\frac{B_2}{2}\left(kx+T/x\right)$, $x\in[1,\infty)$, does not exceed $-B_2\sqrt{T k}$. Taking~\eqref{eq:R2PartOneFirstIntegral} and~\eqref{eq:R2PartThreeFirstIntegral} in~\eqref{eq:R2FirstIntegral} contributes~\eqref{eq:U3}, \eqref{eq:U41}, and the second term in~\eqref{eq:U1}, while taking~\eqref{eq:R2PartTwoFirstIntegral} in place of~\eqref{eq:R2PartOneFirstIntegral} contributes~\eqref{eq:U43} in place of~\eqref{eq:U41}.

Finally, we need to estimate also $\mathcal{R}_3$. By Proposition~\ref{prop:saddle} we have
\begin{equation}
\label{eq:R3FirstIntegral}
\left|\mathcal{R}_3\right| \leq A_1\left(A_3^4B_3+\sqrt{A_3}B_4\right)\frac{x_0}{2}\Phi\left(x_0\right)F^{-\frac{3}{2}}\left(x_0\right).
\end{equation}
If $1\leq k\leq A''T$, then
\[
\sqrt{A''}\eta\left(A''\right)\sqrt{\frac{T}{k}} \leq x_0 \leq \sqrt{\frac{T}{2\pi k}},
\]
which implies
\begin{flalign*}
\Phi\left(x_0\right) &\leq \left(1+\frac{1}{\eta\left(A''\right)}\right)^{\left|\gamma-\beta\right|}x_0^{\gamma-\alpha-\beta} \\
&\leq \left(1+\frac{1}{\eta\left(A''\right)}\right)^{\left|\gamma-\beta\right|}\times \\
&\times\max\left\{(2\pi)^{-\frac{1}{2}(\gamma-\alpha-\beta)},\left(\sqrt{A''}\eta\left(A''\right)\right)^{\gamma-\alpha-\beta}\right\}\left(\frac{T}{k}\right)^{\frac{1}{2}(\gamma-\alpha-\beta)}
\end{flalign*}
and
\[
F\left(x_0\right) \geq \frac{1}{\sqrt{A''}+\frac{1}{\sqrt{2\pi}}}\sqrt{kT}.
\]
Taking these bounds in~\eqref{eq:R3FirstIntegral} contributes~\eqref{eq:U42}. Now let $k\geq A'T$. We then have
\[
A'\eta\left(A'\right)\frac{T}{k} \leq x_0 \leq \frac{T}{2\pi k},
\]
which implies
\[
\Phi\left(x_0\right) \leq \left(1+\frac{1}{2\pi A'}\right)^{\left|\gamma-\beta\right|}x_0^{-\alpha} \leq \left(A'\eta\left(A'\right)\right)^{-\alpha}\left(1+\frac{1}{2\pi A'}\right)^{\left|\gamma-\beta\right|}\left(\frac{T}{k}\right)^{-\alpha}
\]
and
\[
F\left(x_0\right) \geq T\left(1+\frac{1}{2\pi A'}\right)^{-1}.
\]
Taking these bounds in~\eqref{eq:R3FirstIntegral} contributes~\eqref{eq:U44}. This concludes the proof of~\eqref{eq:FirstIntegral}.

We need to prove also~\eqref{eq:FirstIntegral2}. We have
\[
\widehat{\mathcal{U}}=\mathcal{E}\left(a,b,-k;\varphi,f\right)=\mathcal{R}_1+\mathcal{R}_2
\]
by Proposition~\ref{prop:saddle}, where $k$ is replaced by $-k$ in the bounds for $\mathcal{R}_1$ and $\mathcal{R}_2$. Trivially, the estimate~\eqref{eq:R2FirstIntegral} does not change as it depends on $|k|$. The estimation of $\left|\mathcal{R}_1 \right|$ is slightly different. In this case $|f'(a) - k| = k - f'(a)$ and $|f'(b) - k| = k - f'(b)$. But it is not difficult to see that lower bounds~\eqref{conda} and~\eqref{condb} are also valid in this case. Because the term $\mathcal{R}_3$ contributes only~\eqref{eq:U42} and~\eqref{eq:U44}, the proof of~\eqref{eq:FirstIntegral2}, and also Proposition~\ref{prop:FirstIntegral}, is thus complete.
\end{proof}

\begin{proposition}
\label{prop:SecondIntegral}
Let $A,A',\alpha$ be positive real numbers and $A\sqrt{T}\leq a\leq A'\sqrt{T}$. Additionally, let $n\geq 1$ and
\begin{equation}
\label{eq:SecondIntegralCond}
n < \frac{T}{2 \pi}, \quad \frac{1}{n}\left( \frac{T}{2 \pi}-n\right)^2 \geq a^2.
\end{equation}
Then
\begin{flalign}
\label{eq:SecondIntegral}
&\int_{a}^{\infty} \frac{\exp{\left(\ie\left(4 \pi x \sqrt{n} - 2 T \arsinh\left( x\sqrt{\frac{\pi}{2T}} \right) - \sqrt{2\pi x^2 T + \pi^{2} x^4} + \pi x^2 \right)\right)}}{x^{\alpha} \arsinh \left( x\sqrt{\frac{\pi}{2T}} \right) \left(\sqrt{\frac{T}{2 \pi x^2} + \frac{1}{4}} + \frac{1}{2} \right) \left(\frac{T}{2 \pi x^2} + \frac{1}{4} \right)^{\frac{1}{4}}} \dif{x} \nonumber \\
&= \frac{4\pi}{T} n^{\frac{\alpha - 1}{2}} \left(\log \frac{T}{2 \pi n} \right)^{-1} \left(\frac{T}{2 \pi} - n \right)^{\frac{3}{2} - \alpha} \exp{\left(\ie\left(T-T \log\frac{T}{2\pi n} - 2 \pi n + \frac{\pi}{4}\right)\right)} \nonumber \\
&+ \mathcal{V}\left(\alpha,a,n,T,A,A'\right),
\end{flalign}
where
\begin{multline}
\label{eq:V}
\left|\mathcal{V}\right| \leq \mathcal{V}_1 \cdot T^{-\frac{\alpha}{2}} \min{\left\{1, \left| 2\sqrt{n} + a - \sqrt{a^2 + \frac{2 T}{\pi} }\right|^{-1}\right\}} \\
+ \mathcal{V}_2 \cdot T^{-\frac{3}{2}}n^{\frac{\alpha - 1}{2}} \left(\frac{T}{2 \pi} - n \right)^{1 - \alpha} + \mathcal{V}_3\cdot T^{-\frac{\alpha}{2}}n^{-\frac{1}{2}}e^{-\mathcal{B}_2(A,\alpha_0)\left(T+A\sqrt{T}\right)}
\end{multline}
Here,
\begin{equation}
\label{eq:alpha0}
0 < \alpha_0 < \min\left\{\frac{1}{\left(1+\frac{2\mathcal{A}_2\mathcal{A}_3(A)}{3}\right)\sqrt{2}},\frac{\mathcal{A}_3(A)\left(1+2\mathcal{A}_2\right)}{2},\left(\frac{48\mathcal{A}_3(A)^3}{\pi^3}\right)^{\frac{1}{6}}\right\},
\end{equation}
\begin{gather}
\mathcal{B}_i\left(A,\alpha_0\right)\de B_i\left(\mathcal{A}_2,\mathcal{A}_3(A),\alpha_0\right), \quad i\in\{1,2,3,4\}, \label{eq:CalB} \\
\mathcal{V}_1\left(\alpha,\alpha_0,A,A'\right) = \mathcal{A}_1\left(\alpha,A\right) \mathcal{B}_1(A,\alpha_0)A^{-\alpha}\left(1-\left(1+ \frac{1}{2\pi\left(A'\right)^2} \right)^{-\frac{1}{2}}\right)^{-\frac{1}{2}}, \label{eq:V1} \\
\mathcal{V}_2\left(\alpha,\alpha_0,A\right) = \frac{1}{2}\mathcal{A}_1\left(\alpha,A\right)\left(\mathcal{A}_3\left(A\right)^4\mathcal{B}_3(A,\alpha_0)+\sqrt{\mathcal{A}_3\left(A\right)}\mathcal{B}_4(A,\alpha_0)\right), \label{eq:V2} \\
\mathcal{V}_3\left(\alpha,\alpha_0,A\right) = \frac{\mathcal{A}_1\left(\alpha,A\right)}{\mathcal{B}_2(A,\alpha_0)}A^{-\alpha}\left(1+\frac{\alpha_0}{\sqrt{2}}\right), \label{eq:V3}
\end{gather}
with $\mathcal{A}_1,\mathcal{A}_2,\mathcal{A}_2$, and $B_i$ defined by~\eqref{eq:calA} and~\eqref{eq:B1}--\eqref{eq:B4}, respectively.

If one of the conditions in~\eqref{eq:SecondIntegralCond} is not satisfied, or if $\sqrt{n}$ is replaced by $-\sqrt{n}$ in the integral, then the main term in~\eqref{eq:SecondIntegral} and the second term on the right-hand side of~\eqref{eq:V} are to be omitted.
\end{proposition}

In the proof we are using Proposition~\ref{prop:saddle} with $k=2\sqrt{n}$ and the following functions:
\begin{gather}
\varphi(z) = z^{-\alpha} \left(\arsinh{\left(z \sqrt{\frac{\pi}{2 T}}\right)}\right)^{-1} \left( \sqrt{\frac{T}{2 \pi z^2} + \frac{1}{4}} + \frac{1}{2} \right)^{-1} \left(\frac{T}{2 \pi z^2} + \frac{1}{4} \right)^{-\frac{1}{4}}, \label{eq:AuxFun3} \\
f(z) = \frac{z^2}{2} - \sqrt{\frac{T z^2}{2 \pi} + \frac{z^4}{4}} - \frac{T}{\pi} \arsinh{\left( z \sqrt{\frac{\pi}{2 T}}\right)}, \label{eq:AuxFun4} \\
\Phi(x)=x^{-\alpha}, \quad F(x)\equiv T, \quad \mu(x)=\frac{x}{2}. \label{eq:AuxFun5}
\end{gather}
Then $f(z)$ and $\varphi(z)$ are holomorphic functions on the set $\bigcup_{x\in[a,b]}\cD{x,\mu(x)}$, where $a$ is from Proposition~\ref{prop:SecondIntegral} and $a\leq b$. Also, for $x\in[a,b]$ the function $f(x)$ is real and
\[
f''(x) = 1 - \left(1+\frac{2T}{\pi x^2}\right)^{-\frac{1}{2}} > 0.
\]

As before, let us provide the constants $A_1$, $A_2$, and $A_3$ from Proposition~\ref{prop:saddle} for this particular set of functions. Observe that these constants will depend on the positive parameters $A$ and $\alpha$ from Proposition~\ref{prop:SecondIntegral}.

\begin{lemma}
With $A_1=\mathcal{A}_1\left(\alpha,A\right)$, $A_2=\mathcal{A}_2$ and $A_3=\mathcal{A}_3\left(A\right)$, where
\begin{equation}
\label{eq:calA}
\mathcal{A}_1\left(\alpha,A\right) \de 2^{\alpha+\frac{1}{2}}\left(\arsinh{\left(\frac{A}{2}\sqrt{\frac{\pi}{2}}\right)}\right)^{-1}, \quad \mathcal{A}_2\de \frac{1}{\pi}, \quad \mathcal{A}_3\left(A\right)\de \frac{8}{A^2}+9\pi,
\end{equation}
the conditions of Proposition~\ref{prop:saddle} are true for the functions defined in~\eqref{eq:AuxFun3},~\eqref{eq:AuxFun4}, and~\eqref{eq:AuxFun5}.
\end{lemma}

\begin{proof}
We are going to show that $\left|\varphi(z)\right|\leq A_1\Phi(x)$, $\left|f'(z)\right| \leq A_2 F(x)/\mu(x)$, and $1/|f''(z)| \leq A_3\mu^2(x)/F(x)$ for $x\in[a,b]$ and $z\in\cD{x,x/2}$.

Because $\Re\left\{z^2\right\}\geq (x/2)^2\geq (A/2)^{2}T$ and $\Re\left\{\sqrt{w}\right\}\geq\sqrt{\Re\left\{w\right\}}$ for $\Re\left\{w\right\}>0$, we have
\[
\left|\arsinh{\left(z \sqrt{\frac{\pi}{2 T}}\right)}\right|\geq \arsinh{\left(\frac{A}{2}\sqrt{\frac{\pi}{2}}\right)}.
\]
Because $\Re\left\{1/z^2\right\}>0$, we also have
\[
\left|\sqrt{\frac{T}{2 \pi z^2} + \frac{1}{4}} + \frac{1}{2}\right|\geq 1, \quad \left|\frac{T}{2 \pi z^2} + \frac{1}{4}\right|\geq \frac{1}{4}.
\]
Now, $|z|^{-\alpha}\leq 2^\alpha\Phi(x)$, which gives $\left|\varphi(z)\right|\leq \mathcal{A}_1\left(\alpha,A\right)\Phi(x)$, thus proving the first bound.

Because
\[
\left|f'(z)\right| = \frac{2T}{\pi\left|z+\sqrt{z^2+\frac{2 T}{\pi}} \right|} \leq \frac{2T}{\pi x} = \mathcal{A}_2\frac{F(x)}{\mu(x)},
\]
the second bound follows.

We have
\[
\frac{1}{\left|f''(z)\right|} = \left|1-\left(1+\frac{2T}{\pi z^2}\right)^{-\frac{1}{2}}\right|^{-1} = \left|1+\frac{\pi z^2}{2T}\right|\cdot\left|1+\left(1+\frac{2T}{\pi z^2}\right)^{-\frac{1}{2}}\right|.
\]
Because
\[
\left|1+\frac{\pi z^2}{2T}\right| \leq \mathcal{A}_3(A)\frac{x^2}{8T}, \quad \left|1+\left(1+\frac{2T}{\pi z^2}\right)^{-\frac{1}{2}}\right| \leq 2,
\]
we obtain
\[
\frac{1}{\left|f''(z)\right|} \leq \mathcal{A}_3(A)\frac{x^2}{4T} = \mathcal{A}_3(A)\frac{\mu^2(x)}{F(x)},
\]
thus proving also the third bound.
\end{proof}

\begin{proof}[Proof of Proposition~\ref{prop:SecondIntegral}]

Let $b\geq T$ and
\[
x_0 \de \frac{1}{\sqrt{n}}\left(\frac{T}{2\pi}-n\right).
\]
Then $f'\left(x_0\right)=-2\sqrt{n}$ and $x_0\in[a,b)$, the latter being true because of~\eqref{eq:SecondIntegralCond} and $n\geq 1$. According to Proposition~\ref{prop:saddle} we have
\begin{equation}
\label{eq:ESecondIntegral}
\mathcal{E}\left(a,b,2\sqrt{n};\varphi,f\right)= \frac{\varphi\left(x_0\right)}{\sqrt{f''\left(x_0\right)}}e^{2\pi\ie\left(f\left(x_0\right)+2\sqrt{n}x_0+\frac{1}{8}\right)} + \mathcal{R}_1 + \mathcal{R}_2 + \mathcal{R}_3,
\end{equation}
where $\mathcal{E}\left(a,b,2\sqrt{n};\varphi,f\right)$, after taking $b\to\infty$, is exactly the integral from Proposition~\ref{prop:SecondIntegral}. We need to estimate each remainder term in~\eqref{eq:ESecondIntegral}. We are using estimates~\eqref{eq:R1},~\eqref{eq:R2} and~\eqref{eq:R3} with $B_i\left(\mathcal{A}_2,\mathcal{A}_3(A),\alpha_0\right)$ for $i\in\{1,2,3,4\}$ and constants $A_j=\mathcal{A}_j$ for $j\in\{1,2,3\}$, where $\mathcal{A}_j$ are defined by~\eqref{eq:calA} and $\alpha_0$ satisfies the condition~\eqref{eq:alphacond}, i.e.,~\eqref{eq:alpha0}. We are also using the functions $\mathcal{B}_i\left(A,\alpha_0\right)$ which are defined by~\eqref{eq:CalB}.

Because
\[
\varphi(x_0) = \frac{4\pi\sqrt{2} n^{\frac{\alpha}{2}} }{T \log \frac{T}{2 \pi n}}\left(\frac{T}{2\pi} - n \right)^{\frac{3}{2}-\alpha} \left( \frac{T}{2\pi} + n \right)^{-\frac{1}{2}}, \quad
f''(x_0) = 2n\left(\frac{T}{2\pi}+ n\right)^{-1},
\]
we can easily verify that the main parts in~\eqref{eq:SecondIntegral} and~\eqref{eq:ESecondIntegral} coincide. In what follows are estimations of the terms $\mathcal{R}_i$.

Beginning with $\mathcal{R}_1$, we have
\[
\frac{\Phi(a)}{\left|f'(a)+2\sqrt{n}\right|+\sqrt{f''(a)}} \leq \min \left\{\frac{\left(A\sqrt{T}\right)^{-\alpha}}{\left|a - \sqrt{a^2 + \frac{2 T}{\pi}} + 2\sqrt{n} \right|}, \frac{\left(A\sqrt{T}\right)^{-\alpha}}{\sqrt{1 - \frac{a}{\sqrt{a^2 + \frac{T}{2 \pi}}}}} \right\}
\]
and
\[
\frac{\Phi(b)}{\left|f'(b)+2\sqrt{n} \right|+\sqrt{f''(b)}} \leq
\frac{b^{-\alpha}}{\left| b - \sqrt{b^2 + \frac{2 T}{\pi}} + 2\sqrt{n}\right| + \sqrt{1 - \frac{b}{\sqrt{b^2 + \frac{T}{2 \pi}}}}}.
\]
Observe that the last inequality implies
\[
\lim_{b\to\infty} \frac{\Phi(b)}{\left|f'(b)+k \right|+\sqrt{f''(b)}} = 0.
\]
Hence, in the limit $b\to\infty$ we obtain
\[
\left|\mathcal{R}_1\right| \leq \mathcal{V}_1\left(\alpha,\alpha_0,A,A'\right) \cdot T^{-\frac{\alpha}{2}} \min{\left\{1, \left| 2\sqrt{n} + a - \sqrt{a^2 + \frac{2 T}{\pi} }\right|^{-1}\right\}},
\]
where $\mathcal{V}_1\left(\alpha,\alpha_0,A,A'\right)$ is defined by~\eqref{eq:V1}. It is clear that this estimate contributes the first term in~\eqref{eq:V}.

Proceeding to $\mathcal{R}_2$, we have
\begin{flalign*}
\left|\mathcal{R}_2\right| &\leq \mathcal{A}_1\left(\alpha,A\right)\int_{a}^{b}\Phi(x)e^{\mathcal{B}_2\left(A,\alpha_0\right)\left(-|2\sqrt{n}|\mu(x)-F(x)\right)}\left(1+\alpha_0\sqrt{2}\left|\mu'(x)\right|\right)\dif{x} \nonumber \\
&\leq \mathcal{A}_1\left(\alpha,A\right)\left(1+\frac{\alpha_0}{\sqrt{2}}\right)a^{-\alpha}e^{-\mathcal{B}_2\left(A,\alpha_0\right)T}\int_{a}^{\infty}e^{-\mathcal{B}_2\left(A,\alpha_0\right)x\sqrt{n}}\dif{x} \\
&\leq \mathcal{V}_3\left(\alpha,\alpha_0,A\right)\cdot T^{-\frac{\alpha}{2}}n^{-\frac{1}{2}}e^{-\mathcal{B}_2(A,\alpha_0)\left(T+A\sqrt{T}\right)},
\end{flalign*}
where $\mathcal{V}_3\left(\alpha,\alpha_0,A\right)$ is defined by~\eqref{eq:V3}. This estimate thus contributes the last term in~\eqref{eq:V}.

Turning to $\mathcal{R}_3$, it is straightforward to see that
\[
\left|\mathcal{R}_3\right| \leq \mathcal{V}_2\left(\alpha,\alpha_0,A\right) \cdot T^{-\frac{3}{2}}n^{\frac{\alpha - 1}{2}} \left(\frac{T}{2 \pi} - n \right)^{1 - \alpha},
\]
where $\mathcal{V}_2\left(\alpha,\alpha_0,A\right)$ is defined by~\eqref{eq:V2}, thus contributing the second and final term in~\eqref{eq:V}. With this, the inequality~\eqref{eq:V} is proved and consequently also the first part of Proposition~\ref{prop:SecondIntegral}.

To prove the second part, observe that if one of the conditions~\eqref{eq:SecondIntegralCond} is not true, then $x_0\notin[a,b]$. Also, the equation $f'(x)=2\sqrt{n}$ does not have a solution in the interval $[a,b]$. According to Proposition~\ref{prop:saddle}, in both cases we are left only with terms $\mathcal{R}_1$ and $\mathcal{R}_2$ on the right-hand side of~\eqref{eq:ESecondIntegral}.
\end{proof}

\section{Estimation of terms $\mathcal{E}_i$ in equation~\eqref{eq:MainForE}}
\label{sec:proof}

We are going to estimate each $\mathcal{E}_i$ in order to provide the proof of Theorem~\ref{thm:OurAtkinson}. Throughout this section we are using the following restriction
\begin{equation}
\label{eq:TCond}
T\geq T_0 \geq \max\left\{4\pi,\frac{1}{A''},\frac{\sqrt{e}}{A'},5\pi^2\left(A''+\frac{1}{2}\right)\left(1+\sqrt{1+\frac{2}{\pi A''}}\right)^2\right\}
\end{equation}
on $T$ and $T_0$.

\subsection{Bounding $\mathcal{E}_1$}
\label{sec:BoundingE1}

We are using Proposition~\ref{prop:FirstIntegral} with $\alpha\in(1/2,\infty)\setminus\{1\}$, $\beta=1/2$ and $\gamma=1$, while taking $k=\pm n\in\Z$ for $n\leq N\leq A''T$, and $a$, $b$ and $T$ satisfy the conditions of this proposition. Because
\[
2\sin{x}\cos{y}=\sin{(x+y)}+\sin{(x-y)}, \quad \sin(x-\pi n)=(-1)^{n}\sin{x},
\]
we obtain, after taking the imaginary parts of~\eqref{eq:FirstIntegral} and~\eqref{eq:FirstIntegral2},
\begin{equation*}
\int_{a}^{b} \frac{\sin{\left(T\log{\frac{1+y}{y}}\right)}\cos{\left(2\pi ny\right)}}{y^{\alpha}\sqrt{1+y}\log{\frac{1+y}{y}}}\dif{y} = \frac{\sqrt{2}}{4}\left(\frac{T}{2\pi}\right)^{\frac{1}{4}}\frac{(-1)^{n}}{n^{\frac{3}{4}}}e(T,n)\cos{f(T,n)} + \mathcal{O}_{11}
\end{equation*}
with
\begin{multline*}
\left|\mathcal{O}_{11}\right| \leq \mathcal{U}_1\left(\alpha,\frac{1}{2},1,a\right)a^{1-\alpha}T^{-1} + \mathcal{U}_2\left(\alpha,\frac{1}{2},1\right)b^{\frac{1}{2}-\alpha}n^{-1} \\
+ \mathcal{U}_3\left(\alpha,\frac{1}{2},1,b,n,T\right) + \frac{1}{2}\mathcal{U}_4\left(\alpha,\frac{1}{2},1,A'',n,T\right) + \frac{1}{2}\widehat{\mathcal{U}}_4\left(\alpha,\frac{1}{2},1,A'',n,T\right),
\end{multline*}
where $e(T,n)$ and $f(T,n)$ are defined by~\eqref{eq:Atke} and~\eqref{eq:Atkf}, respectively, while $\mathcal{U}_1,\ldots,\mathcal{U}_4$ and $\widehat{\mathcal{U}}_4$ are the functions from Proposition~\ref{prop:FirstIntegral}. Taking firstly $a\to0$ and $b\to\infty$, and then $\alpha\to1/2$, we get
\[
\mathcal{E}_1 = \Sigma_1(T,N) + \mathcal{O}_{12},
\]
where $\Sigma_1(T,N)$ is defined by~\eqref{eq:AtkSigma1}, and
\begin{equation}
\label{eq:CalO12}
\left|\mathcal{O}_{12}\right| \leq E_{1}\left(A'',T\right)T^{-\frac{1}{4}},
\end{equation}
where
\begin{equation}
\label{eq:E1}
E_{1}\left(A'',T\right) \de \zeta^2\left(\frac{5}{4}\right)E_{11}\left(A''\right)
+A''E_{12}(T)T^{\frac{5}{4}}\left(\log{\left(A''T\right)}+2\gamma-1+\frac{1}{\sqrt{A''T}}\right)
\end{equation}
with
\begin{gather}
E_{11}(x) \de 6\left(2\cdot\mathcal{U}_{41}\left(\frac{1}{2},\frac{1}{2},1,x\right)+\mathcal{U}_{42}\left(\frac{1}{2},\frac{1}{2},1,x\right)\right), \label{eq:E11} \\
E_{12}(T) \de 12\sqrt{2}\left(1.02\cdot Te^{-9.58\cdot10^{-6}\sqrt{T}}+212232.3\cdot e^{-4.79\cdot10^{-6}T}\right). \label{eq:E12}
\end{gather}
The second part of~\eqref{eq:E1} comes from~\eqref{eq:divisorSum} with $\left|\Delta(x)\right|\leq\sqrt{x}$. It is clear that for $T\geq T_0$, while $T_0$ and $A''$ are positive and fixed, the function $E_{1}\left(A'',T\right)$ is bounded.

\subsection{Bounding $\mathcal{E}_2$}

Having
\begin{equation}
\label{eq:BoundOnN}
A'T < N + \frac{1}{2} \leq \left(A''+\frac{1}{2T_0}\right)T
\end{equation}
for $T\geq T_0\geq 1/A''$, the inequality~\eqref{eq:explicitDeltaStar} and the same method as in Section~\ref{sec:BoundingE1} yield
\begin{equation}
\label{eq:CalE2}
\left|\mathcal{E}_2\right| \leq \left(A''+\frac{1}{2T_0}\right)^{\frac{1}{3}}E_2\left(A',A'',T_0,T\right)T^{-\frac{1}{6}}\log{\left(\frac{3e A''}{2}T\right)},
\end{equation}
where
\begin{multline*}
E_2\left(A',A'',T_0,T\right) = \frac{\pi^{\frac{1}{4}}\sqrt{A''+\frac{1}{2T_0}}}{\left(A'\right)^{\frac{3}{4}}\left(2+\pi A'\right)^{\frac{1}{4}}\arsinh{\sqrt{\frac{\pi A'}{2}}}} \\
+\left(A'\right)^{-\frac{5}{4}}E_{11}\left(A''+\frac{1}{2T_0}\right)T^{-1}+E_{12}(T)\sqrt{T}. 
\end{multline*}
Here, $E_{11}$ and $E_{12}$ are defined by~\eqref{eq:E11} and~\eqref{eq:E12}, respectively. It is also clear that for $T\geq T_0$, while $T_0$, $A'$ and $A''$ are positive and fixed, the function $E_2\left(A',A'',T_0,T\right)$ is bounded.

\subsection{Bounding $\mathcal{E}_3$}

For $y>0$ define
\begin{gather*}
\phi_1(y)\de \frac{\sin\left(2\pi\left(N+\frac{1}{2}\right)y\right)}{y}, \quad \phi_2(y)\de \frac{\sqrt{y(1+y)}}{\log{\frac{1+y}{y}}}, \quad \phi_3(y)\de \frac{\sin{\left(T\log{\frac{1+y}{y}}\right)}}{y(1+y)}, \\
\phi_4(y)\de y\phi_1(y), \quad \phi_5(y)\de \frac{1}{y}\int_{\frac{1}{2}-\ie T}^{\frac{1}{2}+\ie T}\frac{1}{u}\left(\frac{1+y}{y}\right)^{u}\dif{u}.
\end{gather*}
Then we can write $\mathcal{E}_3$ as
\begin{multline*}
\mathcal{E}_3 = -\frac{2\left(\log\left(N + \frac{1}{2}\right) + 2 \gamma\right)}{\pi}\left(\int_{0}^{\frac{1}{2N+1}}+\int_{\frac{1}{2N+1}}^{\infty}\right)\phi_1(y)\phi_2(y)\phi_3(y)\dif{y} \\
+ \frac{1}{\pi \ie}\left(\int_{0}^{1}+\int_{1}^{\infty}\right)\phi_4(y)\phi_5(y)\dif{y}.
\end{multline*}
Denote by $\mathcal{E}_{31}$, $\mathcal{E}_{32}$, $\mathcal{E}_{33}$ and $\mathcal{E}_{34}$ the above integrals in the same order.

Because $\phi_1(y)$ is a positive monotonically decreasing function on $[0,1/(2N+1)]$, and $\phi_2(y)$ is a monotonically increasing function, by using the second mean-value theorem twice we obtain
\begin{align*}
\mathcal{E}_{31} &= \pi (2N+1) \int_{0}^{\xi} \phi_2(y)\phi_3(y) \dif{y} = \pi (2N + 1) \phi_2\left(\xi\right) \int_{\eta}^{\xi} \phi_3(y) \dif{y} \\
&= \frac{\pi(2N + 1) \phi_2\left(\xi\right)}{T}\left( \cos{\left(T\log{\frac{1+\xi}{\xi}}\right)}-
\cos{\left(T\log{\frac{1+\eta}{\eta}}\right)}\right)
\end{align*}
for some $0 \leq \eta \leq \xi \leq 1/(2N+1)$. Using~\eqref{eq:BoundOnN} we get from this that
\begin{equation}
\label{eq:calE31}
\left|\mathcal{E}_{31}\right| \leq \frac{2\pi \left(A''+\frac{1}{2 T_0}\right)\sqrt{1 + 2 A'}}{A' \log\left( 1 + 2 A' \right)} T^{-\frac{1}{2}}.
\end{equation}

We apply Proposition~\ref{prop:FirstIntegral} in order to bound $\mathcal{E}_{32}$. Observe that the conditions of Proposition~\ref{prop:FirstIntegral} are satisfied since $T_0\geq 4\pi$ by~\eqref{eq:TCond}. In combination with~\eqref{eq:calE31} it gives us
\[
\left|\mathcal{E}_{31}+\mathcal{E}_{32}\right| \leq E_{31}\left(A',A'',T_0\right)T^{-\frac{1}{2}} + E_{32}\left(A',A'',T_0\right)T^{-\frac{3}{2}} + E_{33}\left(A',T\right),
\]
where
\begin{multline}
\label{eq:E31}
E_{31} \de \frac{2\pi \left(A''+\frac{1}{2 T_0}\right)\sqrt{1 + 2 A'}}{A' \log\left( 1 + 2 A' \right)} + \mathcal{U}_1\left( \frac{3}{2}, \frac{1}{2}, 1, \frac{1}{2 A' T_0} \right) \sqrt{2A'' + \frac{1}{T_0}} \\
+ \frac{\pi\sqrt{A'' + \frac{1}{2 T_0}}}{2\sqrt{2}\arsinh{\sqrt{\frac{\pi A'}{2}}}}\left( \left( \frac{1}{2\pi A'} + \frac{1}{4} \right)^{\frac{1}{4}} + \frac{1}{2} \left( \frac{1}{2\pi A'} + \frac{1}{4} \right)^{-\frac{1}{4}}\right),
\end{multline}
\begin{equation}
\label{eq:E32}
E_{32} \de 3\left(A'\right)^{-\frac{3}{4}} \left(2\cdot\mathcal{U}_{41}\left( \frac{3}{2}, \frac{1}{2}, 1, A'' + \frac{1}{2 T_0}\right) + \mathcal{U}_{42}\left(\frac{3}{2}, \frac{1}{2}, 1, A'' + \frac{1}{2 T_0} \right)\right),
\end{equation}
and
\begin{equation}
\label{eq:E33}
E_{33} \de 6\sqrt{2}\left(1.02\cdot \log{T} e^{-9.58 \cdot 10^{-6} T\sqrt{A'}}
+ \frac{212232.3}{A'T^2} e^{-4.79\cdot 10^{-6}A'T^2}\right).
\end{equation}
The third term on the right-hand side of~\eqref{eq:E31} comes from estimating the modulus of the main term in Proposition~\ref{prop:FirstIntegral}.

Let us consider $\mathcal{E}_{33}$ and $\mathcal{E}_{34}$. By residue calculus we have
\begin{equation}
\label{eq:phi5}
y\phi_5(y) = 2\pi \ie + \int_{-\infty + \ie T}^{\frac{1}{2} + \ie T} \left( \frac{1+y}{y} \right)^u \frac{\dif{u}}{u} + \int_{\frac{1}{2} - \ie T}^{-\infty - \ie T} \left( \frac{1+y}{y} \right)^u \frac{\dif{u}}{u}
\end{equation}
for $0< y\leq 1$. We will bound the right-hand side of the equation above step by step.
\begin{align*}
\left| \int_{-\infty + \ie T}^{\frac{1}{2} + \ie T} \left( \frac{1+y}{y} \right)^u \frac{\dif{u}}{u} \right| \leq \frac{1}{T} \int_{-\infty}^{1/2} \left( \frac{1 + y}{y} \right)^{t} \dif{t} \leq \frac{\sqrt{2}}{\log{2}} \cdot T^{-1} y^{-\frac{1}{2}}.
\end{align*}
Similarly we get the same upper bound for the second integral term on the right-hand side of~\eqref{eq:phi5}. Therefore,
\begin{equation}
\label{eq:c1E33}
\left| \mathcal{E}_{33} \right| = \left| 2\pi \ie \int_0^1 \frac{\sin\left( 2\pi\left( N + \frac{1}{2} \right) y \right)}{y} \dif{y} \right| + \frac{2\sqrt{2}}{\log 2} T^{-1} \int_0^1 \frac{\left| \sin\left( 2\pi\left( N + \frac{1}{2} \right) y \right) \right|}{y^{3/2}} \dif{y},
\end{equation}
where
\begin{multline*}
\left| 2\pi \ie \int_0^1 \frac{\sin\left( 2\pi\left( N + \frac{1}{2} \right) y \right)}{y} \dif{y} \right| = 2 \pi \left| \frac{\pi}{2} - \int_{2 \pi (N + 1/2)}^{\infty} \frac{\sin{v}}{v} \dif{v} \right| \\
\leq 2 \pi \left( \frac{\pi}{2} + \frac{1}{\pi \left( N + 1/2 \right)} \right) = \pi^2 + \frac{2}{N + 1/2} \leq \pi^2 + \frac{2}{A' T}.
\end{multline*}
We used the second mean-value theorem in the inequality above. We split the second integral on the right-hand side of~\eqref{eq:c1E33} into two parts:
\begin{multline*}
\left(\int_{0}^{\frac{1}{N + 1/2}}+\int_{\frac{1}{N + 1/2}}^{1}\right)\frac{\left| \sin\left( 2\pi\left( N + \frac{1}{2} \right) y \right) \right|}{y^{\frac{3}{2}}} \dif{y} \leq 2\pi \left(N + \frac{1}{2}\right) \int_{0}^{\frac{1}{N + 1/2}} \frac{\dif{y}}{\sqrt{y}} \\
+ 2 \left( N + \frac{1}{2}\right)^{\frac{1}{2}} - 2 \leq
2\left(2\pi+1\right)\left( N + \frac{1}{2}\right)^{\frac{1}{2}} - 2.
\end{multline*}
Hence,
\begin{equation*}
\left| \mathcal{E}_{33} \right| \leq \pi^2 + \frac{4\sqrt{2}(2\pi+1)}{\log 2}\left( A'' + \frac{1}{2 T_0} \right)^{\frac{1}{2}} T^{-\frac{1}{2}} + \left( \frac{2}{A'} - \frac{4\sqrt{2}}{\log 2} \right)T^{-1}.
\end{equation*}

Finally, let us estimate also $\mathcal{E}_{34}$. Observe that for $y > 1$ we have
\[
\left|y\phi_5(y)\right| \leq \left(1+\frac{1}{y}\right)^{\frac{1}{2}}\int_{-T}^{T}\frac{\dif{t}}{\sqrt{1/4+t^2}} \leq 2\sqrt{2}\arsinh{\left(2T\right)}.
\]
This implies $\lim_{y\to\infty}\phi_5(y)=0$. Integration by parts thus gives
\begin{flalign*}
\mathcal{E}_{34} &= \frac{\cos{\left(2\pi\left(N+\frac{1}{2}\right)\right)}}{2\pi \left(N + \frac{1}{2}\right)}\phi_5(1) 
\\
&-\frac{1}{2\pi \left(N + \frac{1}{2}\right)}\int_{1}^{\infty}\frac{\cos{\left(2\pi\left(N+\frac{1}{2}\right)y\right)}}{y^2}\int_{\frac{1}{2}-\ie T}^{\frac{1}{2}+\ie T}\frac{1}{u}\left(\frac{1+y}{y}\right)^{u}\dif{u}\dif{y} 
\\
&-\frac{1}{2\pi \left(N + \frac{1}{2}\right)}\int_{1}^{\infty}\frac{\cos{\left(2\pi\left(N+\frac{1}{2}\right)y\right)}}{y^3}\int_{\frac{1}{2}-\ie T}^{\frac{1}{2}+\ie T}\left(\frac{1+y}{y}\right)^{u-1}\dif{u}\dif{y}.
\end{flalign*}
For $y > 1$ we also have
\[
\left|\int_{\frac{1}{2}-\ie T}^{\frac{1}{2}+\ie T}\left(\frac{1+y}{y}\right)^{u-1}\dif{u}\right| \leq 2\left(\log{\frac{1+y}{y}}\right)^{-1}.
\]
Therefore,
\begin{flalign*}
\left|\mathcal{E}_{34}\right| &\leq \frac{1}{2\pi\left(N+\frac{1}{2}\right)}\left(2\sqrt{2}\arsinh{(2T)}\left(1+\int_{1}^{\infty}\frac{\dif{y}}{y^2}\right)+2\int_{1}^{\infty}\frac{\dif{y}}{y^3\log{\frac{1+y}{y}}}\right) \\
&\leq \frac{2\sqrt{2}\arsinh{(2T)}+\frac{1}{\log{2}}}{\pi A'T},
\end{flalign*}
where we used the fact that $y\log{\left(1+1/y\right)}\geq\log{2}$.

After putting all together we obtain
\begin{flalign}
\left|\mathcal{E}_{3}\right| &\leq \frac{2}{\pi}\left(E_{31}+E_{32}T^{-1}+\sqrt{T}E_{33}\right)T^{-\frac{1}{2}}\log{\left(\frac{3e^{2\gamma}A''}{2}T\right)} \nonumber \\
&+ \pi + \frac{4\sqrt{2}\left(2\pi+1\right)\sqrt{A''+\frac{1}{2T_0}}}{\pi\log{2}}T^{-\frac{1}{2}} + \frac{2\sqrt{2}\arsinh{(2T)}}{\pi^2 A'T} \nonumber \\
&+ \frac{1}{\pi}\left(\frac{2}{A'}+\frac{1}{\log{2}}\left(\frac{1}{\pi A'}-4\sqrt{2}\right)\right)T^{-1}, \label{eq:CalE3}
\end{flalign}
where $E_{31}\left(A',A'',T_0\right)$, $E_{32}\left(A',A'',T_0\right)$ and $E_{33}\left(A',T\right)$ are defined by~\eqref{eq:E31},~\eqref{eq:E32} and~\eqref{eq:E33}, respectively. It is clear that $\sqrt{T}E_{33}$ is bounded for $T\geq T_0$ and fixed $A'$.

\subsection{Bounding $\mathcal{E}_4$}

We can write $\mathcal{E}_4 = \mathcal{E}_{41} + \mathcal{E}_{42} + \mathcal{E}_{43}$, where
\begin{gather*}
\mathcal{E}_{41} \de \int_{N+\frac{1}{2}}^{\infty}\frac{\Delta^{\ast}(x)}{x}\int_{0}^{\infty}\frac{4T\cos{\left(2\pi xy\right)}\cos{\left(T\log{\frac{1+y}{y}}\right)}}{y^{\frac{1}{2}}(1+y)^{\frac{3}{2}}\log{\frac{1+y}{y}}}\dif{y}\dif{x}, \\
\mathcal{E}_{42} \de -\int_{N+\frac{1}{2}}^{\infty}\frac{\Delta^{\ast}(x)}{x}\int_{0}^{\infty}\frac{2\cos{\left(2\pi xy\right)}\sin{\left(T\log{\frac{1+y}{y}}\right)}}{y^{\frac{1}{2}}(1+y)^{\frac{3}{2}}\log{\frac{1+y}{y}}}\dif{y}\dif{x}, \\
\mathcal{E}_{43} \de -\int_{N+\frac{1}{2}}^{\infty}\frac{\Delta^{\ast}(x)}{x}\int_{0}^{\infty}\frac{4\cos{\left(2\pi xy\right)}\sin{\left(T\log{\frac{1+y}{y}}\right)}}{y^{\frac{1}{2}}(1+y)^{\frac{3}{2}}\log^2{\frac{1+y}{y}}}\dif{y}\dif{x}.
\end{gather*}
Our plan is to estimate each of the above double integrals by using Proposition~\ref{prop:FirstIntegral} to bound inner integrals, and Proposition~\ref{prop:SecondIntegral} in combination with Vorono\"{i}'s summation formula (see Lemma~\ref{lem:Voronoi}) on the first integral to obtain the second term $\Sigma_2$ from Atkinson's formula.

By Proposition~\ref{prop:FirstIntegral} for $\alpha=1/2$, $\beta=3/2$, $\gamma=1$, and $a\to0$, $b\to\infty$, we obtain
\begin{multline*}
\int_{0}^{\infty}\frac{4T\cos{\left(2\pi xy\right)}\cos{\left(T\log{\frac{1+y}{y}}\right)}}{y^{\frac{1}{2}}(1+y)^{\frac{3}{2}}\log{\frac{1+y}{y}}}\dif{y} = \\ \frac{T\cos{\left(2T\arsinh{\sqrt{\frac{\pi x}{2T}}}+\sqrt{(\pi x)^2+2\pi xT}-\pi x+\frac{\pi}{4}\right)}}{\sqrt{2x}\left(\arsinh{\sqrt{\frac{\pi x}{2T}}}\right)\left(\sqrt{\frac{T}{2\pi x}+\frac{1}{4}}+\frac{1}{2}\right)\left(\frac{T}{2\pi x}+\frac{1}{4}\right)^{\frac{1}{4}}} + \mathcal{O}_{41}(x)
\end{multline*}
with
\[
\left|\mathcal{O}_{41}(x)\right| \leq E_{41}\left(A',T_0,1\right)x^{-\frac{1}{2}} + E_{42}\left(T,x\right),
\]
where
\begin{gather*}
E_{41}\left(A',T_0,\mu\right) \de 12\left(2\cdot\mathcal{U}_{43}\left(\frac{1}{2},\frac{3}{2},\mu,A',T_0\right)+\mathcal{U}_{44}\left(\frac{1}{2},\frac{3}{2},\mu,A'\right)\right), \\
E_{42}\left(T,x\right) \de 24\sqrt{2}\left(1.02\cdot T\log{T}e^{-9.58\cdot10^{-6}\sqrt{Tx}}+\frac{212232.3}{x}e^{-4.79\cdot10^{-6}Tx}\right).
\end{gather*}
Using Lemma~\ref{lem:Voronoi}, we have $\mathcal{E}_{41} = \mathcal{E}_{411} + \mathcal{E}_{412} + \mathcal{E}_{413} + \mathcal{E}_{414}$, where
\begin{gather*}
\mathcal{E}_{411} \de \frac{T}{\pi}\sum_{n=1}^{\infty}\frac{d(n)}{n^{\frac{3}{4}}}
\int_{\sqrt{N+\frac{1}{2}}}^{\infty}\frac{\cos{\left(f\left(T,x^2\right)-\pi x^2+\frac{\pi}{2}\right)}\cos{\left(4\pi x\sqrt{n}-\frac{\pi}{4}\right)}}{x^{\frac{3}{2}}\arsinh{\left(x\sqrt{\frac{\pi}{2T}}\right)}\left(\sqrt{\frac{T}{2\pi x^2}+\frac{1}{4}}+\frac{1}{2}\right)\left(\frac{T}{2\pi x^2}+\frac{1}{4}\right)^{\frac{1}{4}}}\dif{x}, \\
\mathcal{E}_{412} \de \frac{-3T}{32\pi^2}\sum_{n=1}^{\infty}\frac{d(n)}{n^{\frac{5}{4}}}
\int_{\sqrt{N+\frac{1}{2}}}^{\infty}\frac{\cos{\left(f\left(T,x^2\right)-\pi x^2+\frac{\pi}{2}\right)}\sin{\left(4\pi x\sqrt{n}-\frac{\pi}{4}\right)}\dif{x}}{x^{\frac{5}{2}}\arsinh{\left(x\sqrt{\frac{\pi}{2T}}\right)}\left(\sqrt{\frac{T}{2\pi x^2}+\frac{1}{4}}+\frac{1}{2}\right)\left(\frac{T}{2\pi x^2}+\frac{1}{4}\right)^{\frac{1}{4}}}, \\
\mathcal{E}_{413} \de \frac{T}{\sqrt{2}}\int_{N+\frac{1}{2}}^{\infty}\frac{V\cdot\cos{\left(f\left(T,x\right)-\pi x+\frac{\pi}{2}\right)}}{x^{\frac{9}{4}}\left(\arsinh{\sqrt{\frac{\pi x}{2T}}}\right)\left(\sqrt{\frac{T}{2\pi x}+\frac{1}{4}}+\frac{1}{2}\right)\left(\frac{T}{2\pi x}+\frac{1}{4}\right)^{\frac{1}{4}}}\dif{x}, \\
\mathcal{E}_{414} \de \int_{N+\frac{1}{2}}^{\infty}\frac{\Delta^{\ast}(x)}{x}\mathcal{O}_{41}(x)\dif{x},
\end{gather*}
and $f(T,n)$ is defined by~\eqref{eq:Atkf} while $V$ with $|V|\leq V(x)$ is from Lemma~\ref{lem:Voronoi}. Firstly, we will estimate $\mathcal{E}_{413}$. Because $1\leq A'T$ by~\eqref{eq:TCond}, we have
\begin{equation}
\label{eq:CalE413}
\left|\mathcal{E}_{413}\right| \leq \frac{T}{\arsinh{\sqrt{\frac{\pi A'}{2}}}}\int_{A'T}^{\infty}\frac{V(x)}{x^{\frac{9}{4}}}\dif{x} \leq E_{43}\left(A',T_0\right)T^{-\frac{1}{4}},
\end{equation}
where
\begin{equation*}
E_{43}\left(A',T_0\right) \de \frac{4\left(\frac{3\zeta^2\left(\frac{7}{4}\right)}{2^{11}\pi^{3}\sqrt{2}}+\frac{15\zeta^2\left(\frac{9}{4}\right)}{2^{16}\pi^{4}\sqrt{2A'T_0}}+\frac{\left(1+\frac{3}{32\pi}\right)\zeta^2\left(\frac{11}{4}\right)}{9\pi\sqrt{2}A'T_0}\right)}{\left(A'\right)^{\frac{5}{4}}\arsinh{\sqrt{\frac{\pi A'}{2}}}}.
\end{equation*}
Next, by~\eqref{eq:TCond} we have $\sqrt{e}\leq A'T$ and thus
\begin{equation*}
\left|\mathcal{E}_{414}\right| \leq E_{41}\left(A',T_0,1\right)\int_{A'T}^{\infty}\frac{\log{(ex)}}{x^{\frac{7}{6}}}\dif{x} + \frac{\log{\left(eA'T\right)}}{\left(A'T\right)^{\frac{2}{3}}}\int_{A'T}^{\infty} E_{42}(T,x)\dif{x},
\end{equation*}
where we used~\eqref{eq:explicitDeltaStar}. Before proceeding to the terms $\mathcal{E}_{411}$ and $\mathcal{E}_{412}$, we will firstly state bounds for $\left|\mathcal{E}_{42}\right|$ and $\left|\mathcal{E}_{43}\right|$. The approach is similar as before, using Proposition~\ref{prop:FirstIntegral}, but we bound the main term trivially. We obtain
\begin{multline*}
\left|\mathcal{E}_{42}\right| \leq \left(\frac{1}{2\arsinh{\sqrt{\frac{\pi A'}{2}}}}+\frac{1}{2T}E_{41}\left(A',T_0,1\right)\right)\int_{A'T}^{\infty}\frac{\log{(ex)}}{x^{\frac{7}{6}}}\dif{x} \\
+ \frac{\log{\left(eA'T\right)}}{2T\left(A'T\right)^{\frac{2}{3}}}\int_{A'T}^{\infty} E_{42}(T,x)\dif{x}
\end{multline*}
and
\begin{multline*}
\left|\mathcal{E}_{43}\right| \leq \left(\frac{1}{2\left(\arsinh{\sqrt{\frac{\pi A'}{2}}}\right)^2}+\frac{3}{2T}E_{41}\left(A',T_0,2\right)\right)\int_{A'T}^{\infty}\frac{\log{(ex)}}{x^{\frac{7}{6}}}\dif{x} \\
+ \frac{3\log{\left(eA'T\right)}}{2\left(A'T\right)^{\frac{2}{3}}}\int_{A'T}^{\infty} E_{42}(T,x)\dif{x} + 4\mathcal{U}_2\left(\frac{1}{2},\frac{3}{2},2\right)\int_{A'T}^{\infty}\frac{\log{(ex)}}{x^{\frac{5}{3}}}\dif{x}.
\end{multline*}
Therefore,
\begin{equation}
\label{eq:CalE42E43E414}
\left|\mathcal{E}_{42}\right| + \left|\mathcal{E}_{43}\right| + \left|\mathcal{E}_{414}\right| \leq E_{44}\left(A',T,T_0\right) + 3E_{45}\left(A',T\right),
\end{equation}
where
\begin{flalign*}
E_{44}\left(A',T,T_0\right) &\de  \frac{6\left(7+\log{\left(A'T\right)}\right)}{\left(A'T\right)^{\frac{1}{6}}}\Bigg(\left(1+\frac{1}{2T_0}\right)E_{41}\left(A',T_0,1\right)+\frac{3}{2T_0}E_{41}\left(A',T_0,2\right) \nonumber \\
&+\frac{1}{2\arsinh{\sqrt{\frac{\pi A'}{2}}}}\left(1+\frac{1}{\arsinh{\sqrt{\frac{\pi A'}{2}}}}\right)\Bigg) 
\\
&+3\mathcal{U}_2\left(\frac{1}{2},\frac{3}{2},2\right)\frac{5+2\log{\left(A'T\right)}}{\left(A'T\right)^{\frac{2}{3}}}.
\end{flalign*}
and
\begin{flalign*}
E_{45}\left(A',T\right) &\de 7.55\cdot10^{11}\left(\frac{1}{T\sqrt{A'}}+\frac{9.58}{10^6}\right)\frac{T^{\frac{1}{3}}\log{\left(eA'T\right)}\log{T}}{\left(A'\right)^{\frac{1}{6}}}e^{-9.58\cdot10^{-6}T\sqrt{A'}} 
\\
&+ 1.51\cdot10^{12}\frac{\log{\left(eA'T\right)}}{\left(A'\right)^\frac{5}{3}T^{\frac{8}{3}}}e^{-4.79\cdot10^{-6}A'T^2}. 
\end{flalign*}
It remains to estimate $\mathcal{E}_{411}$ and $\mathcal{E}_{412}$, where we will apply Proposition~\ref{prop:SecondIntegral}.

\begin{lemma}
\label{lem:summation}
Let $1\leq X<Y<Z$. Then
\[
\sum_{X<n\leq Y}\frac{d(n)}{Z-n} \leq \left(\log{Z}+2\gamma\right)\log{\frac{Z-X}{Z-Y}}+\frac{2\sqrt{Y}}{Z-Y} + \left|\Li2\left(1-\frac{X}{Z}\right)\right| + \left|\Li2\left(1-\frac{Y}{Z}\right)\right|,
\]
where $\Li2(\cdot)$ is the dilogarithm. Also,
\[
\sum_{X<n\leq Y}\frac{d(n)}{n-Z} \leq \left(\log{Z}+2\gamma\right)\log{\frac{Y-Z}{X-Z}}+\frac{2\sqrt{X}}{X-Z} + \left|\Li2\left(1-\frac{X}{Z}\right)\right| + \left|\Li2\left(1-\frac{Y}{Z}\right)\right|
\]
for $1\leq Z<X<Y$.
\end{lemma}

\begin{proof}
Let $1\leq X<Y<Z$ or $1\leq Z<X<Y$. By partial summation we have
\[
\sum_{X<n\leq Y}\frac{d(n)}{Z-n} = \frac{D(Y)}{Z-Y} - \frac{D(X)}{Z-X} - \int_{X}^{Y}\frac{D(u)}{\left(Z-u\right)^2}\dif{u},
\]
where $D(u)$ is defined by~\eqref{eq:divisorSum}. It is not hard to see that
\begin{flalign*}
\int_{X}^{Y}\frac{u\log{u}}{\left(Z-u\right)^2}\dif{u} &= \left(1+\log{Z}\right)\log{\frac{Z-Y}{Z-X}} + Z\left(\frac{\log{Y}}{Z-Y}-\frac{\log{X}}{Z-X}\right) \\
&+ \log{\frac{X}{Y}}
+ \Li2\left(1-\frac{X}{Z}\right) - \Li2\left(1-\frac{Y}{Z}\right),
\end{flalign*}
\[
\int_{X}^{Y}\frac{u\dif{u}}{\left(Z-u\right)^2} = \frac{Y}{Z-Y} - \frac{X}{Z-X} + \log{\frac{Z-Y}{Z-X}},
\]
and
\[
\int_{X}^{Y}\frac{\left|\Delta(u)\right|}{\left(Z-u\right)^2}\dif{u} \leq \int_{X}^{Y}\frac{\sqrt{u}\dif{u}}{\left(Z-u\right)^2} \leq \frac{\sqrt{Y}}{Z-Y} - \frac{\sqrt{X}}{Z-X}.
\]
From this we obtain
\begin{multline*}
\sum_{X<n\leq Y}\frac{d(n)}{Z-n} = \left(\log{Z}+2\gamma\right)\log{\frac{Z-X}{Z-Y}} + \frac{\Delta(Y)}{Z-Y}-\frac{\Delta(X)}{Z-X} \\
+ \Li2\left(1-\frac{Y}{Z}\right) - \Li2\left(1-\frac{X}{Z}\right) - \int_{X}^{Y}\frac{\Delta(u)}{\left(Z-u\right)^2}\dif{u}.
\end{multline*}
If $1\leq X<Y<Z$, the first inequality from Lemma~\ref{lem:summation} clearly follows. If $1\leq Z<X<Y$, the second inequality from Lemma~\ref{lem:summation} also easily follows from the last equality after multiplying it by $-1$. Lemma~\ref{lem:summation} is thus proved.
\end{proof}

We need Lemma~\ref{lem:summation} for $X=Z/2$, $Y=Z-\sqrt{Z}$, and for $X=Z+\sqrt{Z}$, $Y=2Z$. Let $Z>4$. Taking into account that $\left|\Li2(x)\right|$ is a strictly increasing function for $0\leq x<1$, and is a strictly decreasing function for $x\leq0$, we obtain
\begin{equation}
\label{eq:summation1}
\sum_{Z/2<n\leq Z-\sqrt{Z}}\frac{d(n)}{Z-n} \leq \frac{1}{2}\log^2{Z} - \left(\log{2}-\gamma\right)\log{Z} + 2.365
\end{equation}
and
\begin{equation}
\label{eq:summation2}
\sum_{Z+\sqrt{Z}<n\leq 2Z}\frac{d(n)}{n-Z} \leq \frac{1}{2}\log^2{Z} + \gamma\log{Z} + 3.721.
\end{equation}
Observe that both inequalities are asymptotically correct since the inequalities from Lemma~\ref{lem:summation} are asymptotically sharp.

\begin{lemma}
\label{lem:divsum1}
Let $Y\geq 5$ and $1\leq X<Y$. Then
\[
\sum_{n\leq Y}\frac{d(n)}{\sqrt{n}} \leq 2\sqrt{Y}\log{Y}
\]
and
\begin{multline*}
\sum_{X\leq n\leq Y}\frac{d(n)}{\sqrt{n}} \leq \left(Y-X\right)X^{-\frac{1}{2}}\log{X} + 2\left(Y-X\right)X^{-\frac{1}{2}} \\
- 2\left(1-\gamma\right)\left(Y-X\right)Y^{-\frac{1}{2}} + \log{\sqrt{\frac{Y}{X}}} + 2 + \sqrt{3}.
\end{multline*}
\end{lemma}

\begin{proof}
By partial summation and~\eqref{eq:divisorSum} we have
\begin{flalign*}
\sum_{X<n\leq Y}\frac{d(n)}{\sqrt{n}} &= \frac{D(Y)}{\sqrt{Y}} - \frac{D(X)}{\sqrt{X}} + \frac{1}{2}\int_{X}^{Y}\frac{D(u)}{u\sqrt{u}}\dif{u} \\
&= 2\left(\sqrt{Y}\log{Y}-\sqrt{X}\log{X}\right) - 4\left(1-\gamma\right)\left(\sqrt{Y}-\sqrt{X}\right) \\
&+ \frac{\Delta(Y)}{\sqrt{Y}} - \frac{\Delta(X)}{\sqrt{X}} + \frac{1}{2}\int_{X}^{Y}\frac{\Delta(u)}{u\sqrt{u}}\dif{u}.
\end{flalign*}
Take $X=1$. Because $\Delta(1)=2-2\gamma$, the above equality and $\left|\Delta(u)\right|\leq\sqrt{u}$ imply
\[
\sum_{n\leq Y}\frac{d(n)}{\sqrt{n}} \leq 2\sqrt{Y}\log{Y} - 4\left(1-\gamma\right)\sqrt{Y}+\frac{1}{2}\log{Y}+2\left(2-\gamma\right).
\]
The first bound from Lemma~\ref{lem:divsum1} now follows because the remainder in the previous inequality is a decreasing function for $Y\geq 1$ and its value for $Y=5$ is negative. Concerning the second bound, the mean-value theorem implies
\[
\sqrt{Y}\log{Y}-\sqrt{X}\log{X} \leq \frac{2+\log{X}}{2\sqrt{X}}\left(Y-X\right), \quad \sqrt{Y}-\sqrt{X} \geq \frac{Y-X}{2\sqrt{Y}}.
\]
Also,
\[
\max\left\{\frac{d(n)}{\sqrt{n}}\colon n\in\N\right\} = \sqrt{3},
\]
because this is true for $1\leq n \leq 10^3$ by numerical verification, and
\[
\frac{d(n)}{\sqrt{n}} \leq n^{\frac{1.538\log{2}}{\log{\log{n}}}-\frac{1}{2}} < \sqrt{3}
\]
for $n\geq 10^3$ by~\cite{NicolasRobin}. This concludes the proof of Lemma~\ref{lem:divsum1}.
\end{proof}

\begin{lemma}
\label{lem:divsum2}
Let $Y\geq 2$ and $1\leq X<Y$. Then
\[
\sum_{n\leq Y}\frac{d(n)}{n^{\frac{3}{4}}} \leq 4Y^{\frac{1}{4}}\log{Y}
\]
and
\begin{multline*}
\sum_{X<n\leq Y} \frac{d(n)}{n^{\frac{3}{4}}} \leq \left(Y-X\right)X^{-\frac{3}{4}}\log{X} + 4\left(Y-X\right)X^{-\frac{3}{4}} \\
- 2\left(2-\gamma\right)\left(Y-X\right)Y^{-\frac{3}{4}} + 4X^{-\frac{1}{4}} - 2Y^{-\frac{1}{4}}.
\end{multline*}
\end{lemma}

\begin{proof}
By partial summation and~\eqref{eq:divisorSum} we have
\begin{flalign*}
\sum_{X<n\leq Y} \frac{d(n)}{n^{\frac{3}{4}}} &= \frac{D(Y)}{Y^{\frac{3}{4}}} - \frac{D(X)}{X^{\frac{3}{4}}} + \frac{3}{4}\int_{X}^{Y}\frac{D(u)}{u^{\frac{7}{4}}}\dif{u} \\
&= 4\left(Y^{\frac{1}{4}}\log{Y}-X^{\frac{1}{4}}\log{X}\right) - 8\left(2-\gamma\right)\left(Y^{\frac{1}{4}}-X^{\frac{1}{4}}\right) \\
&+ \frac{\Delta(Y)}{Y^{\frac{3}{4}}} - \frac{\Delta(X)}{X^{\frac{3}{4}}} + \frac{3}{4}\int_{X}^{Y}\frac{\Delta(u)}{u^{\frac{7}{4}}}\dif{u}.
\end{flalign*}
Take $X=1$. Because $\Delta(1)=2-2\gamma$, the above equality and $\left|\Delta(u)\right|\leq\sqrt{u}$ imply
\[
\sum_{n\leq Y} \frac{d(n)}{n^{\frac{3}{4}}} \leq 4Y^{\frac{1}{4}}\log{Y} - 8(2-\gamma)Y^{\frac{1}{4}} + 6(3-\gamma) - 2Y^{-\frac{1}{4}}.
\]
The first bound from Lemma~\ref{lem:divsum2} now follows because the remainder in the last inequality is a decreasing function for $Y\geq 1$ and its value for $Y=2$ is negative. Concerning the second bound, the mean-value theorem implies
\[
Y^{\frac{1}{4}}\log{Y}-X^{\frac{1}{4}}\log{X} \leq \frac{4+\log{X}}{4X^{\frac{3}{4}}}\left(Y-X\right), \quad
Y^{\frac{1}{4}} - X^{\frac{1}{4}} \geq \frac{Y-X}{4Y^{\frac{3}{4}}}.
\]
From these inequalities the second bound from Lemma~\ref{lem:divsum2} easily follows.
\end{proof}

Applying the second inequality from Lemma~\ref{lem:divsum2} for $X=Z-\sqrt{Z}$ and $Y=Z+\sqrt{Z}$ gives
\begin{equation}
\label{eq:summation3}
\sum_{Z-\sqrt{Z}<n\leq Z+\sqrt{Z}} \frac{d(n)}{n^{\frac{3}{4}}} \leq 2^{\frac{7}{4}}\cdot Z^{-\frac{1}{4}}\log{Z} + 12.21\cdot Z^{-\frac{1}{4}}
\end{equation}
for $Z\geq 4$.

\begin{lemma}
\label{lem:divsum3}
Let $X\geq 1$. Then
\[
\sum_{n>X}\frac{d(n)}{n^{\frac{5}{4}}} \leq 4X^{-\frac{1}{4}}\log{X} + 8\left(\gamma+2\right)X^{-\frac{1}{4}} + \frac{8}{3}X^{-\frac{3}{4}}.
\]
\end{lemma}

\begin{proof}
By partial summation we have
\[
\sum_{X<n\leq Y}\frac{d(n)}{n^{\frac{5}{4}}} = \frac{D(Y)}{Y^{\frac{5}{4}}} - \frac{D(X)}{X^{\frac{5}{4}}} + \frac{5}{4}\int_{X}^{Y}\frac{D(u)}{u^{\frac{9}{4}}}\dif{u}
\]
for $Y>X$. By~\eqref{eq:divisorSum} and $\left|\Delta(u)\right|\leq\sqrt{u}$ we can take $Y\to\infty$ in the above equality, which then easily implies the stated inequality from Lemma~\ref{lem:divsum3}.
\end{proof}

We are ready to estimate the remaining terms. Remembering that $\mathcal{Z}(T,N)$ is defined by~\eqref{eq:AtkZ}, it can be rewritten as
\begin{equation*}
\mathcal{Z}(T,N) = \frac{T}{2\pi} + \frac{1}{2}\left(N+\frac{1}{2}\right) - \sqrt{\frac{1}{4}\left(N+\frac{1}{2}\right)^2+\frac{T}{2\pi}\left(N+\frac{1}{2}\right)}.
\end{equation*}
For $x>0$ define
\begin{equation}
\label{eq:nu}
\nu(x) \de 1 + \sqrt{1 + \frac{2}{\pi x}}.
\end{equation}
By~\eqref{eq:BoundOnN} we have
\begin{equation}
\label{eq:lowboundZ}
\frac{T}{2\pi}-\mathcal{Z}(T,N) = \frac{\frac{T}{2\pi}\left(N+\frac{1}{2}\right)}{\sqrt{\frac{1}{4}\left(N+\frac{1}{2}\right)^2+\frac{T}{2\pi}\left(N+\frac{1}{2}\right)}+\frac{1}{2}\left(N+\frac{1}{2}\right)} \geq \frac{T}{\pi\nu\left(A'\right)}
\end{equation}
and
\begin{equation}
\label{eq:ToverZ}
\frac{T}{2\pi \mathcal{Z}(T,N)} \geq 1 + A'\pi\nu\left(A'\right)
\end{equation}
since
\begin{equation}
\label{eq:boundsZ}
\frac{T}{\pi^2\left(A''+\frac{1}{2T_0}\right)\nu\left(A''\right)^2} \leq \mathcal{Z}(T,N) \leq \frac{T}{A'\pi^2\nu\left(A'\right)^2}.
\end{equation}
It follows that if $n\leq \mathcal{Z}(T,N)$, then $n<T/(2\pi)$ and
\[
\left(\frac{T}{2\pi}-n\right)^2 \geq \left(\frac{T}{2\pi}-\mathcal{Z}(T,N)\right)^2 = \left(N+\frac{1}{2}\right)\mathcal{Z}(T,N).
\]
But if $n>\mathcal{Z}(T,N)$, then $n>T/(2\pi)$ or
\[
\left(\frac{T}{2\pi}-n\right)^2 < \left(\frac{T}{2\pi}-\mathcal{Z}(T,N)\right)^2 = \left(N+\frac{1}{2}\right)\mathcal{Z}(T,N).
\]
We are doing this in accordance with~\eqref{eq:SecondIntegralCond} for $a=\sqrt{N+1/2}$. Therefore, we will split summation into two parts, $n\leq\mathcal{Z}(T,N)$ and $n>\mathcal{Z}(T,N)$, in order to use Proposition~\ref{prop:SecondIntegral}. Let
\[
\widehat{\alpha}\left(A'\right) \de \min\left\{\frac{1}{\left(1+\frac{2\mathcal{A}_3\left(\sqrt{A'}\right)}{3\pi}\right)\sqrt{2}},\frac{\mathcal{A}_3\left(\sqrt{A'}\right)\left(1+\frac{2}{\pi}\right)}{2},\left(\frac{48\mathcal{A}_3\left(\sqrt{A'}\right)^3}{\pi^3}\right)^{\frac{1}{6}}\right\},
\]
where $\mathcal{A}_3(u)$ is defined in~\eqref{eq:calA}. We obtain
\[
\mathcal{E}_{411} = -\Sigma_2\left(T,N\right) + \mathcal{O}_{42},
\]
where $\Sigma_2\left(T,N\right)$ is defined by~\eqref{eq:AtkSigma2}, and
\begin{flalign}
\left|\mathcal{O}_{42}\right| &\leq \frac{\mathcal{V}_{11}\left(A',A'',T_0,\alpha_0\right)}{\pi} T^{\frac{1}{4}}\sum_{n=1}^{\infty}\frac{d(n)}{n^{\frac{3}{4}}}\min\left\{1,\left|2\sqrt{n}-2\sqrt{\mathcal{Z}(T,N)}\right|^{-1}\right\} \nonumber \\
&+\frac{\mathcal{V}_{12}\left(A',\alpha_0\right)}{2\pi} T^{-\frac{1}{2}}\sum_{n\leq\mathcal{Z}(T,N)}\frac{d(n)}{\sqrt{n}}\left(\frac{T}{2\pi}-n\right)^{-\frac{1}{2}} \nonumber \\
&+ \frac{\zeta^2\left(\frac{5}{4}\right)\mathcal{V}_{13}\left(A',\alpha_0\right)}{\pi} T^{\frac{1}{4}}e^{-\mathcal{B}_2\left(\sqrt{A'},\alpha_0\right)\left(T+\sqrt{A'T}\right)}. \label{eq:O42}
\end{flalign}
Moreover,
\begin{flalign}
\left|\mathcal{E}_{412}\right| &\leq \frac{3}{16\pi}\sum_{n\leq\mathcal{Z}(T,N)}\frac{d(n)}{\sqrt{n}}\left(\log{\frac{T}{2\pi n}}\right)^{-1}\left(\frac{T}{2\pi}-n\right)^{-1} \nonumber \\
&+ \frac{3\mathcal{V}_{22}\left(A',\alpha_0\right)}{64\pi^2}T^{-\frac{1}{2}}\sum_{n\leq\mathcal{Z}(T,N)}\frac{d(n)}{\sqrt{n}}\left(\frac{T}{2\pi}-n\right)^{-\frac{3}{2}} \nonumber \\
&+ \frac{3}{32\pi^2}\left(\mathcal{V}_{21}\left(A',A'',T_0,\alpha_0\right)\zeta^2\left(\frac{5}{4}\right)+\mathcal{V}_{23}\left(A',\alpha_0\right)\zeta^2\left(\frac{7}{4}\right)\right)T^{-\frac{1}{4}}, \label{eq:CalE412}
\end{flalign}
where $0<\alpha_0<\widehat{\alpha}\left(A'\right)$,
\begin{gather*}
\mathcal{V}_{11}\left(A',A'',T_0,\alpha_0\right) \de \mathcal{V}_1\left(\frac{3}{2},\alpha_0,\sqrt{A'},\sqrt{A''+\frac{1}{2T_0}}\right), \\
\mathcal{V}_{12}\left(A',\alpha_0\right) \de \mathcal{V}_2\left(\frac{3}{2},\alpha_0,\sqrt{A'}\right), \quad \mathcal{V}_{13}\left(A',\alpha_0\right) \de \mathcal{V}_3\left(\frac{3}{2},\alpha_0,\sqrt{A'}\right),
\end{gather*}
and
\begin{gather*}
\mathcal{V}_{21}\left(A',A'',T_0,\alpha_0\right) \de \mathcal{V}_1\left(\frac{5}{2},\alpha_0,\sqrt{A'},\sqrt{A''+\frac{1}{2T_0}}\right), \\
\mathcal{V}_{22}\left(A',\alpha_0\right) \de \mathcal{V}_2\left(\frac{5}{2},\alpha_0,\sqrt{A'}\right), \quad \mathcal{V}_{23}\left(A',\alpha_0\right) \de \mathcal{V}_3\left(\frac{5}{2},\alpha_0,\sqrt{A'}\right).
\end{gather*}
Also, the function $\mathcal{B}_2$ is defined by~\eqref{eq:CalB}, and the functions $\mathcal{V}_1$, $\mathcal{V}_2$ and $\mathcal{V}_3$ are defined by~\eqref{eq:V1},~\eqref{eq:V2} and~\eqref{eq:V3}, respectively.

Firstly, we are going to estimate each term in~\eqref{eq:CalE412}. Let $\mu>0$ and observe that $\mathcal{Z}(T,N)\geq 5$ by~\eqref{eq:TCond} and~\eqref{eq:boundsZ}. By Lemma~\ref{lem:divsum1}, and~\eqref{eq:lowboundZ} and~\eqref{eq:boundsZ}, we have
\begin{equation}
\label{eq:divisorMu}
\sum_{n\leq\mathcal{Z}(T,N)}\frac{d(n)}{\sqrt{n}}\left(\frac{T}{2\pi}-n\right)^{-\mu} \leq \frac{2}{\sqrt{A'}}\left(\pi\nu\left(A'\right)\right)^{\mu-1}T^{\frac{1}{2}-\mu}\log{\frac{T}{A'\pi^2\nu\left(A'\right)^2}}.
\end{equation}
Taking $\mu\in\{1,3/2\}$ in~\eqref{eq:divisorMu}, together with~\eqref{eq:ToverZ}, estimate~\eqref{eq:CalE412} assures that
\begin{multline*}
\left|\mathcal{E}_{412}\right| \leq E_{46}\left(A',A'',T_0,\alpha_0\right)T^{-\frac{1}{4}} \\
+ \left(E_{47}\left(A'\right)T^{-\frac{1}{2}}+E_{48}\left(A',\alpha_0\right)T^{-\frac{3}{2}}\right)\log{\frac{T}{A'\pi^2\nu\left(A'\right)^2}},
\end{multline*}
where
\begin{gather*}
E_{46}\left(A',A'',T_0,\alpha_0\right) \de \frac{3}{32\pi^2}\left(\mathcal{V}_{21}\left(A',A'',T_0,\alpha_0\right)\zeta^2\left(\frac{5}{4}\right)
+\mathcal{V}_{23}\left(A',\alpha_0\right)\zeta^2\left(\frac{7}{4}\right)\right), 
\\
E_{47}\left(A'\right)\de \frac{3}{8\pi\sqrt{A'}\log{\left(1+A'\pi\nu\left(A'\right)\right)}}, 
\\
E_{48}\left(A',\alpha_0\right)\de \frac{3\mathcal{V}_{22}\left(A',\alpha_0\right)}{32\pi^2}\sqrt{\frac{\pi}{A'}\nu\left(A'\right)}. 
\end{gather*}

We need to estimate the right-hand side of~\eqref{eq:O42}. Write $\mathcal{Z}=\mathcal{Z}(T,N)$ and let
\begin{flalign*}
\mathcal{S} &\de \sum_{n=1}^{\infty}\frac{d(n)}{n^{\frac{3}{4}}}\min\left\{1,\left|2\sqrt{n}-2\sqrt{\mathcal{Z}}\right|^{-1}\right\} \\
&=\left(\sum_{n\leq \mathcal{Z}/2}+\sum_{\mathcal{Z}/2<n\leq \mathcal{Z}-\sqrt{\mathcal{Z}}}+\sum_{\mathcal{Z}-\sqrt{\mathcal{Z}}<n\leq \mathcal{Z}+\sqrt{\mathcal{Z}}}+\sum_{\mathcal{Z}+\sqrt{\mathcal{Z}}<n\leq 2\mathcal{Z}}+\sum_{n>2\mathcal{Z}}\right) \times \\
&\times\frac{d(n)}{n^{\frac{3}{4}}}\min\left\{1,\left|2\sqrt{n}-2\sqrt{\mathcal{Z}}\right|^{-1}\right\}.
\end{flalign*}
Denote by $\mathcal{S}_1,\ldots,\mathcal{S}_5$ the above sums. Because
\[
\left|2\sqrt{n}-2\sqrt{\mathcal{Z}}\right| \geq 2\left(1-\frac{1}{\sqrt{2}}\right)\sqrt{\mathcal{Z}}
\]
for $n\leq \mathcal{Z}/2$, we have
\[
\mathcal{S}_1 \leq 2^{\frac{3}{4}}\left(1-\frac{1}{\sqrt{2}}\right)^{-1}\mathcal{Z}^{-\frac{1}{4}}\log{\frac{\mathcal{Z}}{2}}
\]
by Lemma~\ref{lem:divsum2}. Because
\[
\left|2\sqrt{n}-2\sqrt{\mathcal{Z}}\right|^{-1} = \frac{\sqrt{\mathcal{Z}}+\sqrt{n}}{2\left(\mathcal{Z}-n\right)} \leq \frac{\sqrt{\mathcal{Z}}}{\mathcal{Z}-n}
\]
and $n^{-3/4}<2^{3/4}\mathcal{Z}^{-3/4}$ for $\mathcal{Z}/2<n\leq \mathcal{Z}-\sqrt{\mathcal{Z}}$, we have
\[
\mathcal{S}_2 \leq \left(2\mathcal{Z}\right)^{-\frac{1}{4}}\log^2{\mathcal{Z}} - 2^{\frac{3}{4}}\left(\log{2}-\gamma\right)\mathcal{Z}^{-\frac{1}{4}}\log{\mathcal{Z}} + 2.365\cdot2^{\frac{3}{4}}\mathcal{Z}^{-\frac{1}{4}}
\]
by inequality~\eqref{eq:summation1} for $Z=\mathcal{Z}$. An upper bound for $\mathcal{S}_3$ is given by the right-hand side of~\eqref{eq:summation3} for $Z=\mathcal{Z}$. Because
\[
\left|2\sqrt{n}-2\sqrt{\mathcal{Z}}\right|^{-1} = \frac{\sqrt{\mathcal{Z}}+\sqrt{n}}{2\left(n-\mathcal{Z}\right)} \leq \frac{\left(1+\sqrt{2}\right)\sqrt{\mathcal{Z}}}{2\left(n-\mathcal{Z}\right)}
\]
and $n^{-3/4}<\mathcal{Z}^{-3/4}$ for $\mathcal{Z}+\sqrt{\mathcal{Z}}<n\leq 2\mathcal{Z}$, we have
\[
\mathcal{S}_4 \leq \frac{1+\sqrt{2}}{4}\mathcal{Z}^{-\frac{1}{4}}\log^2{\mathcal{Z}} + \frac{\gamma\left(1+\sqrt{2}\right)}{2}\mathcal{Z}^{-\frac{1}{4}}\log{\mathcal{Z}} + \frac{3.721\left(1+\sqrt{2}\right)}{2}\mathcal{Z}^{-\frac{1}{4}}
\]
by inequality~\eqref{eq:summation2} for $Z=\mathcal{Z}$. Because
\[
\left|2\sqrt{n}-2\sqrt{\mathcal{Z}}\right| \geq 2\left(1-\frac{1}{\sqrt{2}}\right)\sqrt{n}
\]
for $n>2\mathcal{Z}$, we have
\[
\mathcal{S}_5 \leq 2^{\frac{3}{4}}\left(1-\frac{1}{\sqrt{2}}\right)^{-1}\mathcal{Z}^{-\frac{1}{4}}\log{\left(2\mathcal{Z}\right)} + \left(1-\frac{1}{\sqrt{2}}\right)^{-1}\left(2^{\frac{7}{4}}\left(\gamma+2\right)+\frac{2^{\frac{5}{4}}}{3}\mathcal{Z}^{-\frac{1}{2}}\right)\mathcal{Z}^{-\frac{1}{4}}
\]
by Lemma~\ref{lem:divsum3} for $X=2\mathcal{Z}$. All together gives us
\[
\mathcal{S} \leq 1.445\cdot\mathcal{Z}^{-\frac{1}{4}}\log^2{\mathcal{Z}} + 15.35\cdot\mathcal{Z}^{-\frac{1}{4}}\log{\mathcal{Z}} + 50.276\cdot\mathcal{Z}^{-\frac{1}{4}} + 2.71\cdot\mathcal{Z}^{-\frac{3}{4}}.
\]
This inequality and~\eqref{eq:divisorMu} for $\mu=1/2$ finally implies
\begin{flalign}
\left|\mathcal{O}_{42}\right| &\leq
E_{49}\cdot 1.445\log^{2}{\frac{T}{A'\pi^2\nu\left(A'\right)^2}} + E_{49}\cdot 15.35\log{\frac{T}{A'\pi^2\nu\left(A'\right)^2}} + E_{410} \nonumber \\
&+ E_{411}T^{-\frac{1}{2}}\log{\frac{T}{A'\pi^2\nu\left(A'\right)^2}} + E_{412}T^{-\frac{1}{2}}, \label{eq:CalO42}
\end{flalign}
where
\begin{equation*}
E_{49}\left(A',A'',T_0,\alpha_0\right) \de \frac{\mathcal{V}_{11}\left(A',A'',T_0,\alpha_0\right)}{\sqrt{\pi}} \left(A''+\frac{1}{2T_0}\right)^{\frac{1}{4}}\sqrt{\nu\left(A''\right)},
\end{equation*}
\begin{multline}
\label{eq:E410}
E_{410}\left(A',A'',T_0,T,\alpha_0\right) \de 50.276\cdot E_{49} \left(A',A'',T_0,\alpha_0\right) \\
+ \frac{\zeta^2\left(\frac{5}{4}\right)\mathcal{V}_{13}\left(A',\alpha_0\right)}{\pi} T^{\frac{1}{4}}e^{-\mathcal{B}_2\left(\sqrt{A'},\alpha_0\right)\left(T+\sqrt{A'T}\right)},
\end{multline}
\begin{equation}
\label{eq:E411}
E_{411}\left(A',\alpha_0\right) \de \frac{\mathcal{V}_{12}\left(A',\alpha_0\right)}{\pi\sqrt{\pi A'\nu\left(A'\right)}},
\end{equation}
\begin{equation*}
E_{412}\left(A',A'',T_0,\alpha_0\right) \de 2.71\pi\sqrt{A''+\frac{1}{2T_0}}\nu\left(A''\right) E_{49}\left(A',A'',T_0,\alpha_0\right).
\end{equation*}
We are now in position to provide the proof of our main result.

\begin{proof}[Proof of Theorem~\ref{thm:OurAtkinson}]
In order to obtain the values from Table~\ref{tab:thm}, we collect all terms from the above expressions for $\mathcal{E}_{j}$ that are distinct from $\Sigma_1$ and $\Sigma_2$, together with $R(T)$ from~\eqref{eq:RStirling}, and split them into three groups: terms which are asymptotically $\log^2{T}$, $\log{T}$, and other terms. Let us suppose that after collection we obtain
\[
b_1\left(A',A'',T_0,\alpha_0\right)\log^2{T} + b_2\left(A',A'',T_0,\alpha_0\right)\log{T} + b_3\left(A',A'',T_0,T,\alpha_0\right).
\]
Then the function $b_1$ arises from the first term of~\eqref{eq:CalO42} and is equal to
\[
b_1\left( A', A'', T_0, \alpha_0\right) := 1.445\cdot E_{49}\left( A', A'', T_0, \alpha_0\right),
\]
while $b_2$ arises from the first two terms of~\eqref{eq:CalO42}, therefore
\begin{flalign*}
b_2\left( A', A'', T_0, \alpha_0\right) \de &- 2.89\cdot E_{49}\left( A', A'', T_0, \alpha_0\right) \log\left({A'\pi^2\nu\left(A'\right)^2} \right) \\
&+ 15.35\cdot E_{49}\left( A', A'', T_0, \alpha_0\right).
\end{flalign*}
Remembering that
\[
\mathcal{E}_4 = -\Sigma_2 + \mathcal{O}_{42} + \mathcal{E}_{412} + \mathcal{E}_{413} + \left(\mathcal{E}_{42}+\mathcal{E}_{43}+\mathcal{E}_{414}\right),
\]
the function $b_3\left( A', A'', T_0, T, \alpha_0\right)$ is the sum of all remaining terms from the estimations for $\mathcal{E}_1$, $\mathcal{E}_2$ and $\mathcal{E}_3$, i.e.,~\eqref{eq:CalO12},~\eqref{eq:CalE2} and~\eqref{eq:CalE3}, for $\mathcal{E}_4$, i.e.,~\eqref{eq:CalE413},~\eqref{eq:CalE42E43E414},~\eqref{eq:CalE412} and~\eqref{eq:CalO42}, and for $R(T)$ from~\eqref{eq:RStirling}.

We are interested in three cases for possible pairs of $A'$ and $A''$, namely $\left(A',A''\right)=(0.9,1.1)$, $\left(A',A''\right)=(0.5,1.5)$, $\left(A',A''\right)=(0.1,1.9)$. Let $T\geq T_0\geq 2 \cdot 10^{11}$, assertion which clearly satisfies condition~\eqref{eq:TCond}. Then for fixed $\alpha_0 > 0$ all terms of $b_3$ but one decrease in the variable $T$. The behaviour of this one term $E_{410}(A',A'',T_0,T,\alpha_0)$, which is given by~\eqref{eq:E410}, depends on the parameter $\alpha_0$.

For each $A'\in\{0.9,0.5,0.1\}$ we choose $\alpha_0$ in the range determined by~\eqref{eq:alpha0} to be close to the minimum of function $E_{411}\left(A',\alpha_0\right)$ which is given by~\eqref{eq:E411}. Then we check in each case that for the corresponding $\alpha_0$ the function $E_{410}(A',A'',T_0,T,\alpha_0)$ decreases in $T$. With such procedure we establish the upper bound $a_3(A', A'', T_0)$ for $b_3(A',A'',T_0,T,\alpha_0)$. Afterwards, upper bounds $a_1(A',A'',T_0)$ and $a_2(A',A'',T_0)$ for $b_1(A',A'',T_0, \alpha_0)$ and $b_2(A',A'',T_0,\alpha_0)$, respectively, can be deduced.

Table~\ref{tab:thm} provides values for $a_1(A',A'',T_0)$, $a_2(A',A'',T_0)$, $a_3(A',A'',T_0)$, where $A'$ and $A''$ are as above, and $T_0\in\left\{10^{20},10^{30},\ldots,10^{100},10^{1000}\right\}$. We also listed values for $\alpha_0$ in each case. All computations are done in \emph{Mathematica}.
\end{proof}

\section{Proof of Corollary~\ref{cor:main}}
\label{sec:Coro}

Before proceeding to the proof of Corollary~\ref{cor:main}, we will provide an explicit version of Ingham's estimate~\eqref{eq:InghamE}. Define
\[
\widehat{e}(x) \de (1+x)^{-\frac{1}{4}}\sqrt{x}\left(\arsinh{\sqrt{x}}\right)^{-1}.
\]
Then $e(T,n)=\widehat{e}\left(\pi n/(2T)\right)$. Because $0<\widehat{e}(x)<1$ for $x\in(0,50]$, we have $\left|e(T,n)\right|<1$ for $n\leq T$. Theorem~\ref{thm:OurAtkinson} for $N=\lfloor{T}\rfloor$, $A'=0.9$, $A''=1.1$ and $T_0=10^{30}$ gives
\begin{flalign}
\left|E(T)\right| &\leq \left|\Sigma_1\left(T,\lfloor T\rfloor\right)\right| + \left|\Sigma_2\left(T,\lfloor T\rfloor\right)\right| + \left|\mathfrak{E}(T)\right| \nonumber \\
&\leq 3.87\sqrt{T}\log{T} - 1.1\sqrt{T} + 95\log^2{T} + 275\log{T} + 7.3\cdot10^{18} \nonumber \\
&\leq 110\sqrt{T}\log{T} \label{eq:explIngham}
\end{flalign}
for $T\geq 10^{30}$. Here we estimate the first two terms in Atkinson's formula by Lemmas~\ref{lem:divsum1} and~\ref{lem:divsum2}, while using also inequalities~\eqref{eq:ToverZ} and~\eqref{eq:boundsZ}.

Let
\begin{equation}
\label{eq:JutilaCond}
2\leq t_1\leq T\leq t_2\leq \frac{3}{2}T.
\end{equation}
Because $I_{\frac{1}{2}}(t)$ is an increasing function, we obtain from~\eqref{eq:JutilaCond} and~\eqref{eq:Littlewood} that
\begin{equation}
\label{eq:BoundsOnE}
E\left(t_1\right) - \left(T-t_1\right)\log{T} \leq E(T) \leq E\left(t_2\right) + \left(t_2-T\right)\log{T}
\end{equation}
holds. Let $Y\de T^{\frac{1}{3}}\log^{\mu_1}{T}$ and $G\de T^{-\frac{1}{6}}\log^{\mu_2}{T}$ for certain real numbers $\mu_1$ and $\mu_2$, yet to be determined. Let us assume that $1-T^{-\frac{2}{3}}\log^{\mu_1}{T}>0$, which is clearly true for large $T$. Define
\[
t_1 \de \left(\sqrt{T-Y}+u\right)^{2}, \quad
t_2 \de \left(\sqrt{T+Y}+u\right)^{2},
\]
where $|u|\leq H\de G\log{T}$. We would like~\eqref{eq:JutilaCond} to be valid for large $T$. This will be true if
\begin{gather}
\sqrt{1-T^{-\frac{2}{3}}\log^{\mu_1}{T}}+T^{-\frac{2}{3}}\log^{\mu_2+1}{T} \leq 1, \label{eq:Jutcond1} \\ \sqrt{1+T^{-\frac{2}{3}}\log^{\mu_1}{T}}-T^{-\frac{2}{3}}\log^{\mu_2+1}{T} \geq 1 \label{eq:Jutcond2}
\end{gather}
is satisfied for large $T$. By derivative analysis we can see that this happens if $\mu_1-\mu_2>1$. Take $\varepsilon>0$ and put $\mu_2=\mu_1-1-2\varepsilon$, thus satisfying the previous condition. Then~\eqref{eq:BoundsOnE} implies
\begin{gather}
E\left(t_1\right) - \left(1+\frac{2}{\log^{2\varepsilon}{T}}\right)T^{\frac{1}{3}}\log^{\mu_1+1}{T} \leq E(T), \label{eq:MainBoundOnE1} \\
E(T) \leq E\left(t_2\right) + \left(1+\frac{2\sqrt{1+T^{-\frac{2}{3}}\log^{\mu_1}{T}}}{\log^{2\varepsilon}{T}}+\frac{\log^{\mu_1-4\varepsilon}{T}}{T^{\frac{2}{3}}}\right)T^{\frac{1}{3}}\log^{\mu_1+1}{T}. \label{eq:MainBoundOnE2}
\end{gather}
Define
\[
E_1(x) \de G^{-1}\int_{-H}^{H}E\left((x+u)^2\right)e^{-\left(u/G\right)^2}\dif{u}.
\]
The main idea is to apply the above integral to inequalities~\eqref{eq:MainBoundOnE1} and~\eqref{eq:MainBoundOnE2}, for which we need some good estimate on $E_1(x)$. We will show (see Lemma~\ref{lem:Jutila}) that Lemma 15.4 in~\cite{Ivic} is also true for
\begin{equation}
\label{eq:JutilaM}
M\de G^{-2}\log^{1+2\varepsilon}{T}=T^{\frac{1}{3}}\log^{2(1+2\varepsilon)}{T},
\end{equation}
thus implying
\begin{equation}
\label{eq:JutilaE1}
E_1(x) \ll T^{\frac{1}{3}}\left(\log{T}\right)^{\frac{-2\mu_1+3(1+2\varepsilon)}{4}+1}
\end{equation}
for $\sqrt{T-Y}\leq x\leq \sqrt{T+Y}$. Assume that inequalities~\eqref{eq:Jutcond1} and~\eqref{eq:Jutcond2}, together with
\begin{gather}
\sqrt{1-T^{-\frac{2}{3}}\log^{\mu_1}{T}}-T^{-\frac{2}{3}}\log^{\mu_2+1}{T} \geq \sqrt{\frac{2}{T}}, \label{eq:Jutcon3} \\
\sqrt{1+T^{-\frac{2}{3}}\log^{\mu_1}{T}}+T^{-\frac{2}{3}}\log^{\mu_2+1}{T} \leq \sqrt{\frac{3}{2}}, \label{eq:Jutcon4}
\end{gather}
are satisfied for $T\geq 10^{30}$. After multiplying both sides of~\eqref{eq:MainBoundOnE1} and~\eqref{eq:MainBoundOnE2} by $G e^{-(u/G)^2}$ and integrating them over $u$ from $-H$ to $H$, we get
\begin{flalign}
\left|E(T)\right| &\leq \frac{1}{\sqrt{\pi}}\max\left\{\left|E_1\left(\sqrt{T-Y}\right)\right|,\left|E_1\left(\sqrt{T+Y}\right)\right|\right\} \nonumber \\
&+ \left(1+\frac{2\sqrt{1+T^{-\frac{2}{3}}\log^{\mu_1}{T}}}{\log^{2\varepsilon}{T}}+\frac{\log^{\mu_1-4\varepsilon}{T}}{T^{\frac{2}{3}}}\right)T^{\frac{1}{3}}\log^{\mu_1+1}{T} \nonumber \\
&+ \left(\left(1+\frac{2\sqrt{1+T^{-\frac{2}{3}}\log^{\mu_1}{T}}}{\log^{2\varepsilon}{T}}+\frac{\log^{\mu_1-4\varepsilon}{T}}{T^{\frac{2}{3}}}\right)\frac{\log^{\mu_1}{T}}{T^{\frac{1}{6}}}+110\right)\frac{\sqrt{T/\pi}}{T^{\log{T}}}, \label{eq:E}
\end{flalign}
where we used~\eqref{eq:explIngham} and
\begin{equation}
\label{eq:integral}
\int_{-\infty}^{-H}e^{-\left(u/G\right)^2}\dif{u} = \int_{H}^{\infty}e^{-\left(u/G\right)^2}\dif{u} \leq \frac{G}{2\log{T}}T^{-\log{T}}.
\end{equation}
Therefore, the optimal value for $\mu_1$ is determined by the equation
\[
\frac{-2\mu_1+3(1+2\varepsilon)}{4}+1 = \mu_1+1,
\]
which has a solution $\mu_1=1/2+\varepsilon$. Therefore, having an explicit version of~\eqref{eq:JutilaE1}, which is contained in Lemma~\ref{lem:Jutila}, will produce an explicit version of estimate~\eqref{eq:ourJutila}.

\begin{lemma}
\label{lem:Jutila}
Let $\varepsilon\in(0,1/2]$, $T\geq 1.1T_0$ and $T_0\geq 10^{30}$. Additionally, let $\mathfrak{a}_{n}\left(T_0\right)\de a_{n}\left(0.9,1.1,T_0\right)$ for $n\in\{1,2,3\}$, where $a_1$, $a_2$ and $a_3$ are from Theorem~\ref{thm:OurAtkinson}. Then
\[
\left|E_1(x)\right| \leq \mathfrak{a}\left(T_0\right)T^{\frac{1}{3}}\log^{\frac{3}{2}+\varepsilon}{T}
+ 6.4\cdot T^{\frac{1}{2}-6.28\log^{2\varepsilon}{T}}\log{T}
\]
for
\[
\sqrt{T-T^{\frac{1}{3}}\log^{\frac{1}{2}+\varepsilon}{T}}\leq x\leq \sqrt{T+T^{\frac{1}{3}}\log^{\frac{1}{2}+\varepsilon}{T}},
\]
where
\begin{equation}
\label{eq:fraka}
\mathfrak{a}\left(T_0\right) \de 2.111+\frac{25.4\log{\log{T_0}}}{\log{T_0}}+\frac{\left|\mathfrak{a}_1\right|\sqrt{\pi\log{T_0}}}{T_0^{\frac{1}{3}}}+\frac{\left|\mathfrak{a}_2\right|\sqrt{\pi}}{T_0^{\frac{1}{3}}\sqrt{\log{T_0}}} + \frac{\left|\mathfrak{a}_3\right|\sqrt{\pi}+6.741}{T_0^{\frac{1}{3}}\log^{\frac{3}{2}}{T_0}}.
\end{equation}
\end{lemma}

\begin{proof}
Let $\mu_1=1/2+\varepsilon$ and $t_0=10^{30}$. Then
\[
\left(1-\delta_{-}\left(t_0\right)\right)\sqrt{T}\leq x+u\leq \left(1+\delta_{+}\left(t_0\right)\right)\sqrt{T},
\]
where
\[
\delta_{\pm}(t)\de \mp 1 \pm \sqrt{1\pm t^{-\frac{2}{3}}\log{t}} + t^{-\frac{2}{3}}\sqrt{\log{t}}.
\]
This implies $A'(x+u)^2\leq \lfloor{T}\rfloor\leq A''(x+u)^2$ for
\[
A' = \left(1-\frac{1}{t_0}\right)\left(1+\delta_{+}\left(t_0\right)\right)^{-2}, \quad A'' = \left(1-\delta_{-}\left(t_0\right)\right)^{-2}.
\]
Observe that $A'\geq 0.9$ and $A''\leq 1.1$. Thus, by Theorem~\ref{thm:OurAtkinson} we have
\[
E\left((x+u)^2\right) = \Sigma_1\left((x+u)^2,\lfloor{T}\rfloor\right) + \Sigma_2\left((x+u)^2,\lfloor{T}\rfloor\right) + \mathfrak{E}\left((x+u)^2\right),
\]
where
\[
\left|\mathfrak{E}\left((x+u)^2\right)\right| \leq \mathfrak{a}_1\left(T_0\right)\log^2{T}+\mathfrak{a}_2\left(T_0\right)\log{T}+\mathfrak{a}_3\left(T_0\right).
\]
Using Taylor's expansions $\sqrt{x+u}=\sqrt{x}+\kappa_1(x,u)$, $e\left((x+u)^2,n\right)=e\left(x^2,n\right)+\kappa_2(x,n,u)$,
\[
f\left((x+u)^2,n\right) = f\left(x^2,n\right) + \left(4x\arsinh{\sqrt{\frac{\pi n}{2x^2}}}\right)u + \kappa_3(x,n)u^2 + \kappa_4(x,n,u),
\]
and $\exp{\left(\ie y\right)}=1+\kappa_5(y)$, we can write
\[
\frac{1}{G}\int_{-H}^{H}{\Sigma_1}\cdot e^{-\left(u/G\right)^2}\dif{u} = \Re\left\{\Sigma_{11}\right\} + \Re\left\{\Sigma_{12}\right\} + \Re\left\{\Sigma_{13}\right\} + \Sigma_{14} + \Sigma_{15} + \Sigma_{16},
\]
where
\[
\Sigma_{11} = \left(\frac{2}{\pi}\right)^{\frac{1}{4}}\sqrt{x}\sum_{n\leq T}(-1)^{n}\frac{d(n)}{n^{\frac{3}{4}}}e\left(x^2,n\right)e^{\ie f\left(x^2,n\right)}\frac{1}{G}\int_{-\infty}^{\infty}e^{\widehat{f}(x,n,u)}\dif{u},
\]
\[
\Sigma_{12} = -\left(\frac{2}{\pi}\right)^{\frac{1}{4}}\sqrt{x}\sum_{n\leq T}(-1)^{n}\frac{d(n)}{n^{\frac{3}{4}}}e\left(x^2,n\right)e^{\ie f\left(x^2,n\right)}\frac{1}{G}\left(\int_{-\infty}^{-H}+\int_{H}^{\infty}\right)e^{\widehat{f}(x,n,u)}\dif{u},
\]
\[
\Sigma_{13} = \left(\frac{2}{\pi}\right)^{\frac{1}{4}}\sqrt{x}\sum_{n\leq T}(-1)^{n}\frac{d(n)}{n^{\frac{3}{4}}}e\left(x^2,n\right)e^{\ie f\left(x^2,n\right)}\frac{1}{G}\int_{-H}^{H}\kappa_{5}\left(\kappa_4\right)e^{\widehat{f}(x,n,u)}\dif{u},
\]
\[
\Sigma_{14} = \left(\frac{2}{\pi}\right)^{\frac{1}{4}}\sqrt{x}\sum_{n\leq T}(-1)^{n}\frac{d(n)}{n^{\frac{3}{4}}}\frac{1}{G}\int_{-H}^{H}\kappa_2\cos{f\left((x+u)^2,n\right)}e^{-\left(u/G\right)^2}\dif{u},
\]
\[
\Sigma_{15} = \left(\frac{2}{\pi}\right)^{\frac{1}{4}}\sum_{n\leq T}(-1)^{n}\frac{d(n)}{n^{\frac{3}{4}}}\frac{e\left(x^2,n\right)}{G}\int_{-H}^{H}\kappa_1\cos{f\left((x+u)^2,n\right)}e^{-\left(u/G\right)^2}\dif{u},
\]
\[
\Sigma_{16} = \left(\frac{2}{\pi}\right)^{\frac{1}{4}}\sum_{n\leq T}(-1)^{n}\frac{d(n)}{n^{\frac{3}{4}}}\frac{1}{G}\int_{-H}^{H}\kappa_1\kappa_2\cos{f\left((x+u)^2,n\right)}e^{-\left(u/G\right)^2}\dif{u},
\]
and
\[
\widehat{f}(x,n,u) \de \ie\left(4x\arsinh{\sqrt{\frac{\pi n}{2x^2}}}\right)u - G^{-2}\left(1-\ie G^2\kappa_3(x,n)\right)u^2.
\]
We are going to estimate the moduli of the real functions $\kappa_{1},\ldots,\kappa_{5}$. We obtain
\[
\left|\kappa_1(x,u)\right| \leq m_1\left(t_0\right)T^{-\frac{5}{12}}\sqrt{\log{T}}, \quad m_1\left(t_0\right)\de \frac{1}{2\sqrt{1-\delta_{-}\left(t_0\right)}}
\]
\begin{gather*}
\left|\kappa_2(x,n,u)\right| \leq m_2\left(t_0\right)T^{-\frac{2}{3}}\sqrt{\log{T}}, \\
m_2\left(t_0\right) \de \frac{1}{1-\delta_{-}\left(t_0\right)}\sup_{0<\xi\leq \frac{\pi/2}{\left(1-\delta_{-}\left(t_0\right)\right)^{2}}}\left\{\left|\frac{\xi\sqrt{1+\xi}-\sqrt{\xi}\left(1+\frac{\xi}{2}\right)\arsinh{\sqrt{\xi}}}{\left(1+\xi\right)^{\frac{5}{4}}\arsinh^{2}{\sqrt{\xi}}}\right|\right\},
\end{gather*}
\[
\left|\kappa_3(x,n)\right| \leq m_3\left(t_0\right) \de 2\left(1+\arsinh{\sqrt{\frac{\pi}{2\left(1-\delta_{-}\left(t_0\right)\right)^2}}}\right),
\]
\[
\left|\kappa_4(x,n,u)\right| \leq m_4\left(t_0\right)T^{-1}\log^{\frac{3}{2}}{T}, \quad m_4\left(t_0\right)\de \frac{2}{3}\left(\frac{\pi}{2}\right)^{\frac{3}{2}}\left(1-\delta_{-}\left(t_0\right)\right)^{-4},
\]
and $\left|\kappa_5(y)\right|\leq |y|$ if $y\in\R$. By trivial estimations, using the above bounds together with~\eqref{eq:integral} and the first inequality from Lemma~\ref{lem:divsum2}, we obtain
\begin{flalign}
\sum_{j=2}^{6}\left|\Sigma_{1j}\right| &\leq 4\left(\frac{2}{\pi}\right)^{\frac{1}{4}}\sqrt{1+\delta_{+}\left(t_0\right)}T^{\frac{1}{2}-\log{T}} \nonumber \\
&+ 4(2\pi)^{\frac{1}{4}}\sqrt{1+\delta_{+}\left(t_0\right)}\left(m_4\left(t_0\right)T^{-\frac{1}{2}}\log^{\frac{5}{2}}{T} + m_2\left(t_0\right)T^{-\frac{1}{6}}\log^{\frac{3}{2}}{T}\right) \nonumber \\
&+ 4(2\pi)^{\frac{1}{4}}m_1\left(t_0\right)\left(T^{-\frac{1}{6}}\log^{\frac{3}{2}}{T}+m_2\left(t_0\right)T^{-\frac{5}{6}}\log^{2}{T}\right) \leq 0.02. \label{eq:sigma1215}
\end{flalign}
As before, here we used the estimate $\left|e\left(x^2,n\right)\right|<1$ for $n \leq T$.

The first term $\Sigma_{11}$ requires more careful analysis. Firstly, we have
\[
\left|\frac{1}{G}\int_{-\infty}^{\infty}e^{\widehat{f}(x,n,u)}\dif{u}\right| = \sqrt{\pi}\left(1+G^{4}\left|\kappa_3\right|^2\right)^{-\frac{1}{4}}\exp{\left(\frac{-4G^2\left(x\arsinh{\sqrt{\frac{\pi n}{2x^2}}}\right)^{2}}{1+G^{4}\left|\kappa_3\right|^2}\right)}
\]
since $\Re\left\{1-\ie G^2\kappa_3\right\}=1>0$, and we can thus use the exact form of an appropriate exponential integral, see~\cite[Equation (A.38)]{Ivic}. The idea is to split the range of summation in $\Sigma_{11}$ into two parts, namely $n\leq M$ and $M<n\leq T$ where $M$ is defined by~\eqref{eq:JutilaM}, and then using the latter equality in the second case while estimating trivially in the first case.

Let $n>M$. Then
\begin{flalign*}
4G^2\left(x\arsinh{\sqrt{\frac{\pi n}{2x^2}}}\right)^{2} &\geq 4\left(1-\delta_{-}\left(t_0\right)\right)^{2}TG^2\arsinh^{2}{\left(\sqrt{\frac{\pi}{2}}\frac{\log^{\frac{1}{2}+\varepsilon}{T}}{\left(1+\delta_{+}\left(t_0\right)\right)G\sqrt{T}}\right)} \\
&\geq \frac{4\left(1-\delta_{-}(t_0)\right)^{2}t_0^{\frac{2}{3}}\log^{1+2\varepsilon}{T}}{\log^4{t_0}}\arsinh^{2}{\left(\sqrt{\frac{\pi}{2}}\frac{\log^{2}{t_0}}{\left(1+\delta_{+}\right)t_0^{\frac{1}{3}}}\right)},
\end{flalign*}
where we used also the fact that $y^{-1}\arsinh{y}$ is a decreasing function for $y>0$. Therefore,
\[
\left|\frac{1}{G}\int_{-\infty}^{\infty}e^{\widehat{f}(x,n,u)}\dif{u}\right| \leq \sqrt{\pi}T^{-\rho_{1}\left(t_0\right)\log^{2\varepsilon}{T}},
\]
where
\[
\rho_{1}\left(t_0\right) \de \frac{4\left(1-\delta_{-}\left(t_0\right)\right)^{2}t_0^{\frac{2}{3}}\arsinh^{2}{\left(\frac{\sqrt{\pi}\log^{2}{t_0}}{\sqrt{2}\left(1+\delta_{+}\left(t_0\right)\right)t_0^{1/3}}\right)}}{\log^{4}{t_0}+m_3\left(t_0\right)^2t_0^{-\frac{2}{3}}\log^{2}{t_0}}.
\]
This implies
\begin{multline}
\label{eq:sigma11}
\left|\Sigma_{11}\right| \leq 4(2\pi)^{\frac{1}{4}}\sqrt{1+\delta_{+}\left(t_0\right)}\bigg(\frac{1}{3}T^{\frac{1}{3}}\log^{\frac{3}{2}+\varepsilon}{T} \\
+ 2\left(1+2\varepsilon\right)T^{\frac{1}{3}}\log^{\frac{1}{2}+\varepsilon}{T}\log{\log{T}}
+ T^{\frac{1}{2}-\rho_{1}\left(t_0\right)\log^{2\varepsilon}{T}}\log{T}\bigg).
\end{multline}

We need to estimate also the second term $\Sigma_{2}$. Using Taylor's expansions
\begin{gather*}
\left(\log{\frac{(x+u)^2}{2\pi n}}\right)^{-1} = \left(\log{\frac{x^2}{2\pi n}}\right)^{-1} + \kappa_{6}(x,n,u), \\
g\left((x+u)^2,n\right) = g\left(x^2,n\right) + \left(2x\log{\frac{x^2}{2\pi n}}\right)u + \kappa_{7}(x,n)u^2 + \kappa_{8}(x,n,u),
\end{gather*}
we can write
\[
\frac{1}{G}\int_{-H}^{H}\Sigma_2\cdot e^{-\left(u/G\right)^2}\dif{u} = \Re\left\{\Sigma_{21}\right\} + \Re\left\{\Sigma_{22}\right\} + \Re\left\{\Sigma_{23}\right\} + \Sigma_{24} + \Sigma_{25},
\]
where
\[
\Sigma_{21} = -2\sum_{n\leq \mathcal{Z}\left(x^2,\lfloor{T}\rfloor\right)}
\frac{d(n)}{\sqrt{n}}\left(\log{\frac{x^2}{2\pi n}}\right)^{-1}e^{\ie g\left(x^2,n\right)}\frac{1}{G}\int_{-\infty}^{\infty}e^{\widehat{g}(x,n,u)}\dif{u},
\]
\[
\Sigma_{22} = 2\sum_{n\leq \mathcal{Z}\left(x^2,\lfloor{T}\rfloor\right)}
\frac{d(n)}{\sqrt{n}}\left(\log{\frac{x^2}{2\pi n}}\right)^{-1}e^{\ie g\left(x^2,n\right)}\frac{1}{G}\left(\int_{-\infty}^{-H}+\int_{H}^{\infty}\right)e^{\widehat{g}(x,n,u)}\dif{u},
\]
\[
\Sigma_{23} = -2\sum_{n\leq \mathcal{Z}\left(x^2,\lfloor{T}\rfloor\right)}
\frac{d(n)}{\sqrt{n}}\left(\log{\frac{x^2}{2\pi n}}\right)^{-1}e^{\ie g\left(x^2,n\right)}\frac{1}{G}\int_{-H}^{H}\kappa_{5}\left(\kappa_8\right)e^{\widehat{g}(x,n,u)}\dif{u},
\]
\[
\Sigma_{24} = -2\sum_{n\leq \mathcal{Z}\left(x^2,\lfloor{T}\rfloor\right)}
\frac{d(n)}{\sqrt{n}}\frac{1}{G}\int_{-H}^{H}\kappa_{6}\cos{g\left((x+u)^2,n\right)}e^{-\left(u/G\right)^2}\dif{u},
\]
\[
\Sigma_{25} = \pm \frac{2}{G}\int_{-H}^{H}\sum_{n\in\mathscr{Z}(x,u)}\frac{d(n)}{\sqrt{n}}\left(\log{\frac{(x+u)^2}{2\pi n}}\right)^{-1}\cos{g\left((x+u)^2,n\right)}e^{-\left(u/G\right)^2}\dif{u},
\]
and
\begin{gather*}
\widehat{g}(x,n,u) \de \ie\left(2x\log{\frac{x^2}{2\pi n}}\right)u - G^{-2}\left(1-\ie G^2\kappa_{7}(x,n)\right)u^2, \\
\mathscr{Z}(x,u) \subseteq \left[\mathcal{Z}\left((x-H)^2,\lfloor{T}\rfloor\right),\mathcal{Z}\left((x+H)^2,\lfloor{T}\rfloor\right)\right]\cap\N.
\end{gather*}
In the expression for $\Sigma_{25}$, the sign $\pm$ depends on the one of $u$: we take $+$ if $u$ is positive and $-$ otherwise. By~\eqref{eq:ToverZ} we have
\begin{equation}
\label{eq:m0}
\frac{(x+u)^2}{2\pi\mathcal{Z}\left((x+u)^2,\lfloor{T}\rfloor\right)} \geq m_{0}\left(A'\right) \de 1+A'\pi\nu\left(A'\right),
\end{equation}
where $\nu(x)$ is defined by~\eqref{eq:nu}, and by~\eqref{eq:boundsZ} also
\begin{equation}
\label{eq:zpm}
z_{-}\left(A'',t_0\right)T \leq  \mathcal{Z}\left((x+u)^2,\lfloor{T}\rfloor\right) \leq z_{+}\left(A',t_0\right)T,
\end{equation}
where
\begin{equation*}
z_{-}\left(A'',t_0\right) \de \frac{\left(1-\delta_{-}\left(t_0\right)\right)^2}{\pi^2\left(A''+\frac{1}{2}\right)\nu\left(A''\right)^{2}}, \quad
z_{+}\left(A',t_0\right) \de \frac{\left(1+\delta_{+}\left(t_0\right)\right)^2}{A' \pi^2\nu\left(A'\right)^{2}}.
\end{equation*}
We need to estimate the moduli of the real functions $\kappa_{6}$, $\kappa_{7}$ and $\kappa_{8}$. For $n\leq \mathcal{Z}\left(x^2,\lfloor{T}\rfloor\right)$ we obtain
\begin{gather*}
\left|\kappa_{6}(x,n,u)\right| \leq m_6\left(A',t_0\right) T^{-\frac{2}{3}}\sqrt{\log{T}}, \\ m_6\left(A',t_0\right)\de \frac{2}{1-\delta_{-}\left(t_0\right)}\log^{-2}\left(\left(\frac{1-\delta_{-}\left(t_0\right)}{1+\delta_{+}\left(t_0\right)}\right)^{2}m_0\left(A'\right)\right),
\end{gather*}
\[
\left|\kappa_{7}(x,n)\right| \leq m_7\left(t_0\right)\log{T}, \quad m_7\left(t_0\right)\de 1+\frac{2+\log{\frac{\left(1+\delta_{+}\left(t_0\right)\right)^2}{2\pi}}}{\log{t_0}},
\]
\[
\left|\kappa_{8}(x,n,u)\right| \leq m_8\left(t_0\right)T^{-1}\log^{\frac{3}{2}}{T}, \quad m_8\left(t_0\right)\de \frac{2}{3\left(1-\delta_{-}\left(t_0\right)\right)}.
\]
Because
\begin{flalign*}
\frac{\dif{}}{\dif{x}}\mathcal{Z}\left(x^2,\lfloor{T}\rfloor\right) &= \frac{\left(2/\pi^2\right)x^3}{\lfloor{T}\rfloor+\frac{1}{2}+\frac{2x^2}{\pi}+\sqrt{\left(\lfloor{T}\rfloor+\frac{1}{2}\right)\left(\lfloor{T}\rfloor+\frac{1}{2}+\frac{2x^2}{\pi}\right)}} \\
&\leq \frac{2x^3}{\pi^2\left(2\lfloor{T}\rfloor+1\right)} \leq \frac{x^3}{\pi^2\left(1-\frac{1}{2t_0}\right)T},
\end{flalign*}
we also have
\begin{gather}
\left|\mathcal{Z}\left((x+u)^2,\lfloor{T}\rfloor\right)-\mathcal{Z}\left(x^2,\lfloor{T}\rfloor\right)\right| \leq z\left(t_0\right)T^{\frac{1}{3}}\sqrt{\log{T}}, \label{eq:interval} \\
z\left(t_0\right)\de \frac{1}{\pi^2}\left(1+\delta_{+}\left(t_0\right)\right)^{3}\left(1-\frac{1}{2t_0}\right)^{-1}. \nonumber
\end{gather}
By trivial estimations, using the latter bounds together with $z_{+}\left(A',t_0\right)\leq 1$, inequalities~\eqref{eq:integral} and~\eqref{eq:m0}, and the first inequality from Lemma~\ref{lem:divsum1}, we obtain
\begin{flalign}
\label{eq:sigma2224}
\sum_{j=2}^{4}\left|\Sigma_{2j}\right| &\leq \frac{4}{\log{m_0\left(A'\right)}}T^{\frac{1}{2}-\log{T}} \nonumber \\
&+ 4\sqrt{\pi}\left(\frac{m_8\left(t_0\right)}{\log{m_0\left(A'\right)}}T^{-\frac{1}{2}}\log^{\frac{5}{2}}{T}+m_6\left(A',t_0\right)T^{-\frac{1}{6}}\log^{\frac{3}{2}}{T}\right)\leq 0.02.
\end{flalign}
Similarly as before, we have
\begin{flalign*}
\left|\frac{1}{G}\int_{-\infty}^{\infty}e^{\widehat{g}(x,n,u)}\dif{u}\right| &= \sqrt{\pi}\left(1+G^{4}\left|\kappa_7\right|^2\right)^{-\frac{1}{4}}\exp{\left(\frac{-4G^2\left(x\log{\frac{x^2}{2\pi n}}\right)^{2}}{1+G^{4}\left|\kappa_7\right|^2}\right)} \\
&\leq \sqrt{\pi} T^{-\rho_2\left(t_0\right)T^{2/3}\log^{-3}{T}},
\end{flalign*}
where
\[
\rho_2\left(t_0\right) \de \frac{4\left(1-\delta_{-}\left(t_0\right)\right)^{2}\log^{2}{m_0\left(A'\right)}}{1+\left(m_{7}\left(t_0\right)\right)^{2}t_0^{-\frac{2}{3}}}.
\]
Therefore,
\begin{equation}
\label{eq:sigma21}
\left|\Sigma_{21}\right| \leq \frac{4\sqrt{\pi}}{\log{m_0\left(A'\right)}}T^{\frac{1}{2}-\rho_2\left(t_0\right)T^{2/3}\log^{-3}{T}}\log{T} \leq 0.001.
\end{equation}
We are left only with $\Sigma_{25}$. By the second inequality of Lemma~\ref{lem:divsum1} we have
\[
\sum_{n\in\mathscr{Z}(x,u)} \frac{d(n)}{\sqrt{n}} \leq \sum_{\mathcal{X}\leq n\leq \mathcal{Y}} \frac{d(n)}{\sqrt{n}} \leq \left(\mathcal{Y}-\mathcal{X}\right)\mathcal{X}^{-\frac{1}{2}}\left(\log{\mathcal{X}+2}\right) + \log{\sqrt{\frac{\mathcal{Y}}{\mathcal{X}}}} + 2 + \sqrt{3},
\]
where $\mathcal{X}\de \mathcal{Z}\left(\left(x-H\right)^2,\lfloor{T}\rfloor\right)$ and $\mathcal{Y}\de \mathcal{Z}\left(\left(x+H\right)^2,\lfloor{T}\rfloor\right)$. By~\eqref{eq:zpm} and~\eqref{eq:interval} we obtain
\[
\sum_{n\in\mathscr{Z}(x,u)} \frac{d(n)}{\sqrt{n}} \leq \frac{2z\left(t_0\right)}{\sqrt{z_{-}\left(A'',t_0\right)}}T^{-\frac{1}{6}}\log^{\frac{3}{2}}{T} + \log{\sqrt{\frac{z_{+}\left(A',t_0\right)}{z_{-}\left(A'',t_0\right)}}} + 2 + \sqrt{3} \leq 3.95
\]
since $2+\log{z_{+}\left(A',t_0\right)}<0$. It follows that
\begin{equation}
\label{eq:sigma25}
\left|\Sigma_{25}\right| \leq 7.9\sqrt{\pi}\log^{-1}\left(\left(\frac{1-\delta_{-}\left(t_0\right)}{1+\delta_{+}\left(t_0\right)}\right)^{2}m_0\left(A'\right)\right) \leq 6.7.
\end{equation}
The main inequality from Lemma~\ref{lem:Jutila} now easily follows from~\eqref{eq:sigma11},~\eqref{eq:sigma1215},~\eqref{eq:sigma21}, \eqref{eq:sigma2224} and~\eqref{eq:sigma25}.
\end{proof}

\begin{theorem}
\label{thm:Jutila}
Let $T_0\geq 10^{30}$. Then
\[
\left|E(T)\right| \leq \left(\frac{1}{\sqrt{\pi}}\mathfrak{a}\left(T_0\right)+1+\frac{2.001}{\log^{\frac{1}{3}}{T_0}}\right)T^{\frac{1}{3}}\log^{\frac{5}{3}}{T}
\]
for $T\geq 1.1T_0$, where $\mathfrak{a}\left(T_0\right)$ is defined by~\eqref{eq:fraka}.
\end{theorem}

\begin{proof}
Let $\varepsilon=1/6$. Then $\mu_1=2/3$ and $\mu_2=-2/3$. We can see that inequalities~\eqref{eq:Jutcond1},~\eqref{eq:Jutcond2},~\eqref{eq:Jutcon3}, and~\eqref{eq:Jutcon4} are satisfied for $T\geq 10^{30}$. By Lemma~\ref{lem:Jutila}, Theorem~\ref{thm:Jutila} now follows from inequality~\eqref{eq:E}.
\end{proof}

\begin{proof}[Proof of Corollary~\ref{cor:main}]
The values for $J\left(T_0\right)$ from~\eqref{eq:ExplJut} are calculated for each $T_0$ by Theorem~\ref{thm:Jutila} according to the values from Table~\ref{tab:thm}.
\end{proof}

\subsection*{Acknowledgements} The authors thank Daniele Dona, Harald Helfgott and Sebastian Zuniga Alterman for their interest in the subject, and to Kohji Matsumoto for remarks concerning Motohashi's work, as well as to Roger Heath-Brown, Maksim Korol\"{e}v and Olivier Ramar\'{e}. Finally, the authors are grateful to the anonymous referee for providing many valuable suggestions, and to our supervisor Tim Trudgian for continual guidance and support while writing this manuscript.


\providecommand{\bysame}{\leavevmode\hbox to3em{\hrulefill}\thinspace}
\providecommand{\MR}{\relax\ifhmode\unskip\space\fi MR }
\providecommand{\MRhref}[2]{%
  \href{http://www.ams.org/mathscinet-getitem?mr=#1}{#2}
}
\providecommand{\href}[2]{#2}


\begin{thebibliography}{DHZA19}

\bibitem[Atk39]{Atkinson39}
F.~V. Atkinson, \emph{The mean value of the zeta-function on the critical
  line}, Q. J. Math. \textbf{10} (1939), 122--128.

\bibitem[Atk49]{Atkinson}
\bysame, \emph{The mean-value of the {R}iemann zeta function}, Acta Math.
  \textbf{81} (1949), 353--376.

\bibitem[Bal78]{Balasub}
R.~Balasubramanian, \emph{An improvement on a theorem of {T}itchmarsh on the
  mean square of {$|\zeta (\frac{1}{ 2}+it)|$}}, Proc. London Math. Soc. (3)
  \textbf{36} (1978), no.~3, 540--576.

\bibitem[BBR12]{BBR}
D.~Berkane, O.~Bordell\`es, and O.~Ramar\'{e}, \emph{Explicit upper bounds for
  the remainder term in the divisor problem}, Math. Comp. \textbf{81} (2012),
  no.~278, 1025--1051.

\bibitem[BW18]{BourgainWatt}
J.~Bourgain and N.~Watt, \emph{Decoupling for perturbed cones and the mean square of {$|\zeta (\frac 12+it)|$}}, Int. Math. Res. Not. IMRN (2018), no.~17, 5219--5296.

\bibitem[DHZA22]{DHA}
D.~Dona, H.~A. Helfgott, and S.~Zuniga~Alterman, \emph{Explicit {$L^2$} bounds for the {R}iemann {$\zeta$} function}, J. Th\'{e}or. Nombres Bordeaux~\textbf{34} (2022), no.~1, 91--133.

\bibitem[DZA22]{DonaZA}
D.~Dona and S.~Zuniga~Alterman, \emph{On the {A}tkinson formula for the $\zeta$ function}, J. Math. Anal. Appl.~\textbf{516} (2022), no.~2, Paper No.~126478.

\bibitem[GGL20]{GGYLindelof}
S.~M. Gonek, S.~W. Graham, and Y.~Lee, \emph{The {L}indel\"{o}f hypothesis for
  primes is equivalent to the {R}iemann hypothesis}, Proc. Amer. Math. Soc.
  \textbf{148} (2020), no.~7, 2863--2875.

\bibitem[Goo77]{Good1977}
A.~Good, \emph{Ein {$\Omega $}-{R}esultat f\"{u}r das quadratische {M}ittel der
  {R}iemannschen {Z}etafunktion auf der kritischen {L}inie}, Invent. Math.
  \textbf{41} (1977), no.~3, 233--251.

\bibitem[GR15]{GradRyz}
I.~S. Gradshteyn and I.~M. Ryzhik, \emph{Table of integrals, series, and
  products}, 8th ed., Elsevier/Academic Press, Amsterdam, 2015.

\bibitem[HB78a]{HBTheMeanValue}
D.~R. Heath-Brown, \emph{The mean value theorem for the {R}iemann
  zeta-function}, Mathematika \textbf{25} (1978), no.~2, 177--184.

\bibitem[HB78b]{HBTheTwelfth}
\bysame, \emph{The twelfth power moment of the {R}iemann-function}, Quart. J.
  Math. Oxford Ser. (2) \textbf{29} (1978), no.~116, 443--462.

\bibitem[HB79]{HBFourth}
\bysame, \emph{The fourth power moment of the {R}iemann zeta function}, Proc.
  London Math. Soc. (3) \textbf{38} (1979), no.~3, 385--422.

\bibitem[Ing28]{Ingham1928}
A.~E. Ingham, \emph{Mean-value theorems in the theory of the {R}iemann
  zeta-function}, Proc. Lond. Math. Soc. (2) \textbf{27} (1928), no.~1,
  273--300.

\bibitem[Ivi03]{Ivic}
A.~Ivi\'{c}, \emph{The {R}iemann zeta-function}, Dover Publications, Inc.,
  Mineola, NY, 2003.

\bibitem[Ivi13]{IvicHardy}
\bysame, \emph{The theory of {H}ardy's {$Z$}-function}, Cambridge Tracts in
  Mathematics, vol. 196, Cambridge University Press, Cambridge, 2013.

\bibitem[Ivi14]{Ivic2014}
\bysame, \emph{The mean values of the {R}iemann zeta-function on the critical
  line}, Analytic number theory, approximation theory, and special functions,
  Springer, New York, 2014, pp.~3--68.

\bibitem[Jut83]{Jutila83}
M.~Jutila, \emph{Riemann's zeta function and the divisor problem}, Ark. Mat.
  \textbf{21} (1983), no.~1, 75--96.

\bibitem[Jut97]{Jutila97}
\bysame, \emph{Mean values of {D}irichlet series via {L}aplace transforms},
  Analytic number theory ({K}yoto, 1996), London Math. Soc. Lecture Note Ser.,
  vol. 247, Cambridge Univ. Press, Cambridge, 1997, pp.~169--207.

\bibitem[Luk05]{Lukkarinen}
M.~Lukkarinen, \emph{The {M}ellin transform of the square of {R}iemann's
  zeta-function and {A}tkinson's formula}, vol. 140, Helsinki: Suomalainen
  Tiedeakatemia; Turku: Univ. Turku, Dept. of Mathematics (Dissertation), 2005.

\bibitem[Mat00]{MatsumotoRecent}
K.~Matsumoto, \emph{Recent developments in the mean square theory of the
  {R}iemann zeta and other zeta-functions}, Number theory, Trends Math.,
  Birkh\"{a}user, Basel, 2000, pp.~241--286.

\bibitem[MM93]{MatsumotoMeurman3}
K.~Matsumoto and T.~Meurman, \emph{The mean square of the {R}iemann
  zeta-function in the critical strip. {III}}, Acta Arith. \textbf{64} (1993),
  no.~4, 357--382.

\bibitem[Mot87]{MotohashiLectures}
Y.~Motohashi, \emph{Lectures on the {R}iemann-{S}iegel {F}ormula}, Ulam
  Seminar, Colorado University, Boulder, 1987.

\bibitem[Mot97]{MotohashiSpectral}
\bysame, \emph{Spectral theory of the {R}iemann zeta-function}, Cambridge
  Tracts in Mathematics, vol. 127, Cambridge University Press, Cambridge, 1997.

\bibitem[Nem17]{NemesHankelBessel}
G.~Nemes, \emph{Error bounds for the large-argument asymptotic expansions of
  the {H}ankel and {B}essel functions}, Acta Appl. Math. \textbf{150} (2017),
  141--177.

\bibitem[NR83]{NicolasRobin}
J.-L. Nicolas and G.~Robin, \emph{Majorations explicites pour le nombre de
  diviseurs de {$N$}}, Canad. Math. Bull. \textbf{26} (1983), no.~4, 485--492.

\bibitem[Olv74]{Olver}
F.~W.~J. Olver, \emph{Asymptotics and special functions}, Academic Press, New
  York, 1974.

\bibitem[Sim20]{SimonicEZDE}
A.~Simoni\v{c}, \emph{Explicit zero density estimate for the {R}iemann
  zeta-function near the critical line}, J. Math. Anal. Appl. \textbf{491}
  (2020), no.~1, 124303, 41 pp.

\bibitem[Tit32]{TitchmarshCorput3}
E.~C. Titchmarsh, \emph{On van der {C}orput's method and the zeta-function of
  {R}iemann ({III})}, Q. J. Math. \textbf{3} (1932), 133--141.

\bibitem[Tit34]{TitchmarshCorput5}
\bysame, \emph{On van der {C}orput's method and the zeta-function of {R}iemann
  ({V})}, Q. J. Math. \textbf{5} (1934), 195--210.

\bibitem[Tit86]{Titchmarsh}
\bysame, \emph{The theory of the {R}iemann zeta-function}, 2nd ed., The
  Clarendon Press, Oxford University Press, New York, 1986.

\end{thebibliography}
\end{document}